\documentclass[11pt]{article}
\usepackage{amsmath,amsthm,amssymb,amscd,fancybox,ifthen,float,epsfig,subfigure}
\usepackage[all]{xy}
\usepackage{times}

\floatstyle{plain}
\restylefloat{figure}

\newcommand\w{\wedge}

\newcommand\Span{\operatorname{span}}
\newcommand\tot{\operatorname{tot}}
\newcommand\In{\operatorname{in}}

\newcommand\hone{h^{(1)}}

\newcommand\cT{\mathcal T}
\newcommand\cM{\mathcal M}
\newcommand\cN{\mathcal N}
\newcommand\cH{\mathcal H}
\newcommand\cI{\mathcal I}

\newcommand\cQ{\mathcal Q}

\newcommand\balpha{\boldsymbol \alpha}
\newcommand\bgamma{\boldsymbol \gamma}

\newcommand\bEta{\boldsymbol \eta}
\newcommand\tbEta{\widetilde{\boldsymbol \eta}}
\newcommand\bBeta{\boldsymbol \beta}
\newcommand\bj{\boldsymbol j}

\newcommand\tN{\widetilde N}

\newcommand\bm{\boldsymbol m}

\newcommand\bv{\boldsymbol v}
\newcommand\bn{\boldsymbol n}
\newcommand\br{\boldsymbol r}

\newcommand\bq{\boldsymbol q}
\newcommand\bx{\boldsymbol x}
\newcommand\by{\boldsymbol y}
\newcommand\hbx{\hat{\boldsymbol x}}
\newcommand\bA{\boldsymbol A}
\newcommand\bB{\boldsymbol B}
\newcommand\bD{\boldsymbol D}
\newcommand\bH{\boldsymbol H}

\newcommand\bE{\boldsymbol E}
\newcommand\bV{\boldsymbol V}
\newcommand\bJ{\boldsymbol J}
\newcommand\bN{\boldsymbol N}
\newcommand\bP{\boldsymbol P}
\newcommand\bR{\boldsymbol R}

\newcommand\cE{\mathcal E}

\newcommand\bxi{\boldsymbol \xi}
\newcommand\tbxi{\widetilde{\bxi}}

\newcommand\bXi{\boldsymbol \Xi}

\newcommand{\THME[1]}{THME(#1)}

\newcommand\Ker{\operatorname{ker}}

\newcommand\cC{\mathcal{C}}
\newcommand\cL{\mathcal{L}}

\newcommand\Zup{\overline{\bbC_+}}

\renewcommand\Re{\operatorname{Re}}
\renewcommand\Im{\operatorname{Im}}

\newcommand\bbC{\mathbb C}

\newcommand\bbN{\mathbb N}

\newcommand\bbR{\mathbb R}

\newcommand\pa{\partial}

\newcommand\restrictedto{\upharpoonright}

\newcommand\CI{{\mathcal C}^{\infty}}

\newcommand\Id{\operatorname{Id}}

\DeclareMathOperator{\Ind}{Ind}

\DeclareMathOperator{\dR}{dR}

\newtheorem{theorem}{Theorem}
\newtheorem{proposition}{Proposition}
\newtheorem{corollary}{Corollary}
\newtheorem{lemma}{Lemma}

\theoremstyle{definition}
\newtheorem{definition}{Definition}

\newtheorem{example}{Example}

\theoremstyle{remark}
\newtheorem{remark}{Remark}

\begin{document}

\title{Debye Sources and the Numerical Solution\\
of the Time Harmonic Maxwell Equations} 

\author{Charles L. Epstein\footnote{
    Depts. of Mathematics and Radiology, University of Pennsylvania,
    209 South 33rd Street, Philadelphia, PA 19104. E-mail:
    {cle@math.upenn.edu}.
    Research partially supported by NSF grant
    DMS06-03973 and NIH grants R21 HL088182 and R01AR053156.} \,\,
and Leslie Greengard\footnote{Courant Institute,
    New York University, 251 Mercer Street, New York, NY 10012.
    E-mail: {greengard@cims.nyu.edu}. 
    Research partially supported by the U.S. Department of Energy under
    contract DEFG0288ER25053 and by the AFOSR under MURI grant
    FA9550-06-1-0337. 
    }}  \date{January 14, 2008}

\maketitle
\begin{abstract} In this paper, we develop a new representation for outgoing
  solutions to the time harmonic Maxwell equations in unbounded domains in
  $\bbR^3.$ This representation leads to a Fredholm integral equation of the
  second kind for solving the problem of scattering from a perfect conductor,
  which does not suffer from spurious resonances or low frequency breakdown,
  although it requires the inversion of the scalar surface Laplacian on the
  domain boundary. In the course of our analysis, we give a new proof of the
  existence of non-trivial families of time harmonic solutions with vanishing
  normal components that arise when the boundary of the domain is not simply
  connected.  We refer to these as $k$-Neumann fields, since they generalize,
  to non-zero wave numbers, the classical harmonic Neumann fields.  The
  existence of $k$-harmonic fields was established earlier by Kress.
\end{abstract}

\section*{Introduction}

Electromagnetic wave propagation in a uniform, nonconducting,
isotropic medium in $\bbR^3$
is described by the Maxwell equations
\begin{equation}
\nabla \times {\cal E}(\bx,t) = - \mu \frac{\partial {\cal H}}{\partial t} \, ,
\quad 
\nabla \times {\cal H}(\bx,t) =  \epsilon \frac{\partial {\cal E}}{\partial t} \, ,
\label{maxwelltd}
\end{equation}
\[
\nabla \cdot {\cal E}(\bx,t) = 0 \, ,
\quad 
\nabla \cdot {\cal H}(\bx,t) =  0 \, ,
\]
where ${\cal E}$ and ${\cal H}$ denote the electric and magnetic fields,
respectively and 
$\epsilon, \mu$ are the 
electrical permittivity and 
magnetic permeability of the medium.
We will restrict our attention to the time-harmonic case and 
write 
\begin{equation}
{\cal E}(\bx,t) = \Re \left\{\frac{\bE^{\tot}(\bx)}{\sqrt{\epsilon}} 
e^{-i \omega t} \right\}
\quad {\rm and}\quad
{\cal H}(\bx,t) = \Re \left\{ \frac{\bH^{\tot}(\bx)}{\sqrt{\mu}} e^{-i \omega t}  
\right\} \, . 
\label{maxwellscaling}
\end{equation}
The superscript is used to emphasize that
$\bE^{\tot}$ and $\bH^{\tot}$ define the {\em total} electric and magnetic
fields, respectively. 
In electromagnetic scattering, they are generally 
written as a sum 
\begin{equation}
\bE^{\tot}(\bx) = \bE^{\In}(\bx) + \bE(\bx), \qquad 
\bH^{\tot}(\bx) = \bH^{\In}(\bx) + \bH(\bx),
\end{equation}
where $\{ \bE^{\In}, \bH^{\In} \}$ describe a known incident field and 
$\{ \bE, \bH \}$ denote the scattered field of interest. 
With the scaling in \eqref{maxwellscaling}, the Maxwell equations
take the simpler form
\begin{eqnarray}
\label{maxwellfd}
\nabla \times \bH^{\tot} &=& - i k \bE^{\tot}  \\
\nabla \times \bE^{\tot} &=&  i k \bH^{\tot} \, . \nonumber
\end{eqnarray}
We are particularly interested in the problem of 
scattering from a perfect conductor in an 
exterior region, which we denote by $\Omega$.
For a perfect conductor \cite{jackson,Papas}, 
the conditions to be enforced on $\Gamma$, the 
boundary of $\Omega$, are
\begin{equation}
\bn \times \bE^{\tot} = 0, \qquad  \bn \cdot \bH^{\tot} = 0.
\label{EHbc}
\end{equation}
The scattered field is assumed to satisfy the 
Silver-M\"uller radiation condition:
\begin{equation}
\label{silvermuller}
 \lim_{r\rightarrow \infty}
 \left( \bH \times \frac{\br}{r} - 
\bE \right) = o \left( \frac{1}{r} \right).
\end{equation}

This problem has been studied rather intensively for 
many decades, and we do not seek to review the literature here, except 
to observe that there are two distinct approaches in widespread use.
When the scatterer is a sphere, a simple and elegant theory exists 
due to Lorenz, Debye and Mie \cite{BouwkampCasimir,Debye,Lorenz,Mie}.  
It is based on two scalar 
potentials (generally called Debye potentials), and the mathematical
machinery of vector spherical harmonics. 
In particular, one represents $\bE$, $\bH$ as

\begin{eqnarray} 
\bE &=& \nabla \times \nabla \times (\br v) 
\, + \, i k \nabla \times (\br u) \nonumber \\
\bH &=& \nabla \times \nabla \times (\br u) 
- i k \nabla \times (\br v) 
\label{debyerep}
\end{eqnarray}
where the Debye potentials $u,v$ satisfy the scalar Helmholtz equation 
\[ \Delta u + k^2 u = 0, \ \Delta v + k^2 v = 0\ , \]
with Helmholtz parameter (wave number) $k^2 = \omega^2 \epsilon \mu$. 
Salient features
of this approach are (a) that the boundary value problem 
\begin{eqnarray}
\bn \times \bE &=& - \bn \times \bE^{\In} 
\label{EHbcscat} \\
\bn \cdot \bH &=& -\bn \cdot \bH^{\In}
\label{EHbcscat2}
\end{eqnarray}
is uniquely solvable for any $k$ with non-negative imaginary part and (b) that,
as $k \rightarrow 0$ ($\omega \rightarrow 0$), the electric and magnetic fields
uncouple gracefully. In the static limit, $\bE$ is due to the scalar 
potential $v$ alone, which is, in turn, determined by the boundary 
data $-\bn \times
\bE^{\In}$.  Likewise, $\bH$ is due to the scalar potential $u$ 
alone, which is determined by the boundary data $-\bn \cdot \bH^{\In}$.

For regions of arbitrary shape, on the other hand,
integral formulations of the Maxwell equations 
are generally based on the classical vector and scalar potentials
(in the Lorenz gauge):
\begin{eqnarray}
\bE &=& i k \bA - \nabla \phi  \label{Epotrep} \\
\bH &=& \nabla \times \bA \label{Hpotrep} 
\end{eqnarray}
where
\[
\bA(\bx) = \int_\Gamma g_k(\bx-\by) \bJ(\by) dA_{\by}  
\]
\[ \phi(\bx) = \frac{1}{i k}
\int_\Gamma g_k(\bx-\by) (\nabla_{\Gamma} \cdot \bJ)(\by) dA_{\by}  \]
with
\[ g_k(\bx) = \frac{e^{ik |\bx|}}{ 4 \pi | \bx|} \, . \]
Here, $\bJ$ is a surface current (a tangential vector field) 
and $\nabla_{\Gamma} \cdot \bJ$ denotes its surface divergence.

Maue \cite{Maue} proposed the
electric field integral equation (EFIE) for the unknown current $\bJ$
by enforcing the condition (\ref{EHbcscat}) using the 
representation (\ref{Epotrep}). Because of the $\nabla \phi$ term, 
however, the result is a hypersingular equation.
Maue also proposed the magnetic field integral equation (MFIE),
based on (\ref{Hpotrep}).
The boundary condition for $\bH$
can be derived from
the Maxwell equations and an appropriate limiting process on the 
surface of a perfect conductor \cite{jackson,Papas}:
\begin{equation}
\bJ = \bn \times \bH^{\In} + \bn \times \bH \, ,
\label{Hbc2}
\end{equation}
where $\bn$ points into $\Omega$.  Enforcing this condition for the unknown
current $\bJ$ yields the MFIE, a second kind Fredholm equation.  Unfortunately,
both the MFIE and the EFIE have spurious resonances; that is, there exists a
countable set of frequencies $\{k_j\} \subset \bbR$ for which the integral
equations are not invertible. As the $\{k_j\}$ are the eigenvalues of a
self adjoint, elliptic boundary value problem on the bounded complement of
$\Omega,$ they are often referred to as \emph{interior resonances}. Below the
smallest such $k_j$, the MFIE is well-conditioned.  Spurious resonances,
however, are only one difficulty.  A second problem stems from the
representation of the electric field itself.  Unlike the Debye representation,
the electric field does not uncouple naturally from the magnetic field as
$k\rightarrow 0$.  Note that in (\ref{Epotrep}), $\bE$ involves one term
of order $k$ and one term of order $k^{-1}$.  This results in what is
referred to as ``low-frequency breakdown'' \cite{ZC}. While low frequency
breakdown is a more transparent problem in the context of the EFIE, the MFIE is
not immune \cite{ZCCZ}.  Knowing the current $\bJ$ is sufficient for computing
$\bH$, but not the electric field.
The normal component of $\bE$, for example, is determined by the electric
charge:
\begin{equation}
\bn \cdot \bE = \rho = \frac{\nabla_{\Gamma} \cdot \bJ}{i k} \, .
\label{escatnJ}
\end{equation}
As $k \rightarrow 0$, accuracy degrades dramatically - a phenomenon
called ``catastrophic cancellation" in numerical analysis.

This state of affairs is both odd and unsatisfactory.  For the
exterior of a sphere, there is a simple, clean representation of the solution
based on two scalar unknowns that results in a diagonal linear system. It has
no spurious resonances and, at zero frequency,  decouples naturally 
(with no loss of precision) into magnetostatic and electrostatic problems.
The standard integral equation approaches available for general geometries do
not reduce to a Debye-like formalism when restricted to a sphere. Instead, a
sequence of modifications have been introduced to address the three problems
discussed above: the existence of spurious resonances, the lack of a second
kind integral equation valid for all frequencies, and the loss of accuracy due
to low-frequency breakdown.

An important step in addressing the first problem
was the introduction in the 1970's 
of the combined field integral equation (CFIE)
\cite{MKM,PoggioMiller}.
The CFIE avoids spurious resonances by
taking a complex linear combination of the EFIE and the MFIE, 
both of which involve the surface current as the unknown.
It is not a Fredholm equation of the second kind, however,
and still suffers from low frequency breakdown.
One alternative approach, due to 
Yaghjian \cite{Yaghjian}, involves 
augmenting the MFIE with the condition (\ref{EHbcscat2}) or the EFIE
with the condition (\ref{escatnJ}). He showed that (for geometries other
than the sphere) 
the augmented equations yield unique solutions
at all frequencies. Of the many formulations that have been introduced
to overcome spurious resonances, variants of the CFIE have 
emerged as the most frequently used in practice.

For the second problem, the principal issue is that of overcoming 
the hypersingular behavior of the CFIE. 
One solution is to introduce electric charge $\rho$ 
as an additional variable \cite{TaskinenYla}. In this
approach, one defines
the scalar potential $\phi$ by 
\begin{equation}
 \phi(\bx) = \int_\Gamma g_k(\bx-\by) \rho(\by) dA_{\by} 
\label{phicharge}
\end{equation}
and imposes the continuity condition
\begin{equation}
\nabla \cdot \bJ = i k \rho. 
\label{contcondition}
\end{equation}
While the hypersingular term
is avoided, one must solve
a Fredholm integral equation
subject to a differential-algebraic constraint (\ref{contcondition}).
During the last several years, several promising approaches have
been developed based on the construction of preconditioners.
Christiansen and N\'ed\'elec  
\cite{ChristiansenNedelec} designed effective strategies for the EFIE
based on Calderon formulas and the Helmholtz decomposition.
Adams and  Contopanagos {\em et al.} \cite{Adams,CDEORVW} made use
of the fact that
the EFIE operator serves as its own preconditioner; more precisely,
the composition of the hypersingular operator with itself equals 
the sum of the identity operator and a compact operator. 
A combined field integral equation using
this preconditioned EFIE is both resonance-free and takes the
form of a Fredholm equation of the second kind.
Preconditioners have also been designed through the use of 
high frequency asymptotics \cite{ABL}.
Unfortunately, the implementation of these schemes 
can be rather involved on arbitrary surfaces and, like the MFIE,
they still suffers from
a form of low-frequency breakdown in the evaluation of $\bE$ 
once the integral equation has been solved.

Finally, the third problem - namely the 
low-frequency breakdown of the integral equations -
has generally been handled through the use of 
specialized basis functions in the discretization of the current, such as
the ``loop and tree'' method of \cite{WiltonGlisson,WGK}. This is
a kind of discrete surface Helmholtz decomposition of $\bJ$. 
As the frequency goes to zero, the irrotational and solenoidal 
discretization elements become uncoupled, avoiding the scaling problem
that causes loss of precision.

We have chosen to investigate a rather different line of thought,
motivated largely by the desire to extend the Debye potentials
to surfaces of arbitrary shape. In essence, the Lorenz-Debye-Mie approach
is based on expanding the potentials $u,v$ from (\ref{debyerep})
as
\begin{eqnarray*}
v(r,\theta,\phi) &=& \sum_{n,m} a_{n,m} \hone_n(kr) Y_n^m(\theta,\phi) \\
u(r,\theta,\phi) &=& \sum_{n,m} b_{n,m} \hone_n(kr) Y_n^m(\theta,\phi) 
\end{eqnarray*}
where $\hone_n(x)$ is the spherical Hankel function of order $n$, and
$Y_n^m(\theta,\phi)$ is the usual 
spherical harmonic of order $n$ and degree $m$. This separation of variables
approach is clearly not suitable in general. 
 From a mathematical viewpoint, it works because of the close connection
between the Laplacian in $\bbR^3$ and the surface Laplace-Beltrami operator
on the sphere.
It is also worth noting that the Lorenz-Debye-Mie approach is not equivalent to a 
Fredholm equation of the second kind. It is invertible, resonance free
and behaves properly at low frequencies, but it is hypersingular.
Numerical difficulties are avoided simply because it is in diagonal form.

The features of the Debye potentials that we wish to retain are their symmetry
and the fact that, at zero frequency, the system 
uncouples into separate electrostatic
and magnetostatic problems. For symmetry, we begin by using both potentials
($\bA, \phi$) and ``antipotentials''  $(\bA_m,\phi_m)$ as a formal 
representation of the electromagnetic fields \cite{Papas}:

\begin{eqnarray}\label{potrepre1}
\bE &=& i k \bA - \nabla \phi  - 
 \nabla \times \bA_m
\label{Epotrepsymm} \\
\bH &=& \nabla \times \bA 
+ i k \bA_m - \nabla \phi_m 
\label{Hpotrepsymm} 
\end{eqnarray}
where
\begin{eqnarray}
 \bA(\bx) &=& \int_\Gamma g_k(\bx-\by) \bj(\by) dA_{\by} \, , \nonumber  \\
\phi(\bx) &=& \int_\Gamma g_k(\bx-\by) r(\by) dA_{\by} \, ,
\label{Vecpotdef} \\
 \bA_m(\bx) &=& \int_\Gamma g_k(\bx-\by) \bm(\by) dA_{\by} \, , \nonumber  \\
\phi_m(\bx) &=&  \int_\Gamma g_k(\bx-\by) q(\by) dA_{\by} \, ,
\nonumber
\end{eqnarray}
together with the continuity conditions 
\begin{equation}
\nabla_{\Gamma} \cdot \bj = i k r\ ,\quad 
\nabla_{\Gamma} \cdot \bm = i k q. 
\label{contconditionm}
\end{equation}
Such a symmetric formulation is commonly used for scattering
from a dielectric. For the perfect conductor, 
it underlies the combined source integral equation
method (CSIE) \cite{MautzHarrington}, 
where $\bj$ and $\bm$ are both assumed to be derived from a
"parent" current distribution $\widetilde{\bj}$:
\[ \bj = \alpha \widetilde{\bj} , \quad 
 \bm = (1-\alpha) \bn \times \widetilde{\bj} \,  \]
for some parameter $\alpha$.
More precisely, the CSIE is derived using (\ref{Epotrepsymm}) 
with the vector unknown $\widetilde{\bj}$ and
enforcing the condition (\ref{EHbcscat}). 
Like the CFIE, it is a resonance-free but hypersingular equation.
It is important to recognize that, in this construction, the unknowns
$\bj, r, \bm, q$ are no longer physical quantities.
$\bj$ and $r$ correspond to 
{\em fictitious} electric current and electric charge, while
$\bm$ and $q$ correspond to {\em fictitious} magnetic current 
and magnetic charge. Perfect conductors do not support the latter.
If the ``physical'' current supported on the surface $\Gamma$ is desired,
it must be computed in a second step. From (\ref{Hbc2}), for example,
we have $\bJ = \bn \times (\bH^{\In} + \bH)$.
This will {\em not} be the unknown $\bj$. 

The second and critical feature of our method is that 
we will use $r$ and $q$ as unknowns and 
{\em construct} 
$\bj$ and $\bm$ from them in such a way that the continuity conditions
(\ref{contconditionm}) are automatically satisfied.
In particular, for simply connected domains, we will let
\begin{align}
\label{helmdecompsimple}
\bj &= \nabla_{\Gamma} \Psi + \bn \times \nabla_{\Gamma} \Psi_m \\
\bm &= \bn \times \bj \nonumber
\end{align}
where
\begin{eqnarray}
\label{surflapsolves}
\Delta_{\Gamma} \Psi \equiv \nabla_{\Gamma}^2 \Psi &=& ik r \\
\Delta_{\Gamma} \Psi \equiv \nabla_{\Gamma}^2 \Psi_m &=& -ik q \, . \nonumber 
\end{eqnarray}
We wil refer to $\Delta_{\Gamma}$ as the surface Laplacian
or Laplace-Beltrami operator. (In geometry, this name is usually
applied to $-\Delta_{\Gamma}$ so that it is a non-negative operator,
but we will use the the convention above consistently.)
In any case, we will obtain the Helmholtz decomposition of the 
currents on the surface by construction. 
This avoids the obvious cause of low-frequency breakdown, since we
never compute the $O(1)$ quantities $r,q$  from the $O(k)$
quantities $\bj, \bm$ with its attendant loss of accuracy.

An obvious drawback of our approach, of course, is that
it will require the inversion of a partial differential equation
on the surface of the scatterer to compute $\Psi$ and $\Psi_m$.
It is interesting to note that Scharstein proposed an 
investigation along these
lines some years ago \cite{Scharstein},
using only the electric current $\bJ$, but a detailed investigation
of the theory was not carried out.

We  show below
that our representation yields a second-kind integral equation for
$r$ and $q$ that, in the simply connected case, has a unique solution 
for all frequencies with 
non-negative imaginary part. Furthermore, it behaves gracefully
in the low frequency limit,
uncoupling into an electrostatic problem involving 
$r$ and a magnetostatic problem involving $q$. 
Because of the connection with the Debye theory,
we refer to $r$ and
$q$ as \emph{generalized Debye sources}.
We also present an analysis of the (more delicate) multiply-connected case.

\section{Geometric Preliminaries}
\label{sec:mcd}

\begin{definition}
Let $D$ denote a bounded (not necessarily connected) region in $\bbR^3$ and 
let $\Omega$ denote the unbounded component of $\bbR^3\setminus \overline{D}$.
We will refer to $\Omega$ as the {\em exterior region} and to its boundary as
$\Gamma$. We assume, without loss of generality, that $\bbR^3\setminus D$
has no bounded components (that is, holes within the interior of $D$).
\end{definition}

Using standard topological terminology, let us assume that
$D$ is multiply connected with genus $g$. Then there exist 
surfaces $S_1,\dots,S_g$ in $D$ such that 
$D\setminus \cup_{j=1}^g S_j$ is simply connected and 
surfaces $T_1,\dots,T_g$ in $\bbR^3\setminus D$ such that 
$\bbR^3\setminus D\setminus \cup_{j=1}^g T_j$ is simply connected.
We denote by 
$A_j$ the boundary of $S_j$ and by $B_j$ the 
boundary of $T_j$ (see Fig. \ref{fig1}). 

\begin{remark} \label{firsthomol}
We will refer to the curves $\{ A_j | j = 1,\dots, g \}$ as 
{\em A-cycles}. (They form a basis for the first homology
group of $\bbR^3\setminus D$.) We will refer to the 
curves $\{ B_j | j = 1,\dots, g \}$ as 
{\em B-cycles}. (They form a basis for the first homology
group of $D$.)
\end{remark}

\begin{figure}[hh]
\centering
{\epsfig{file=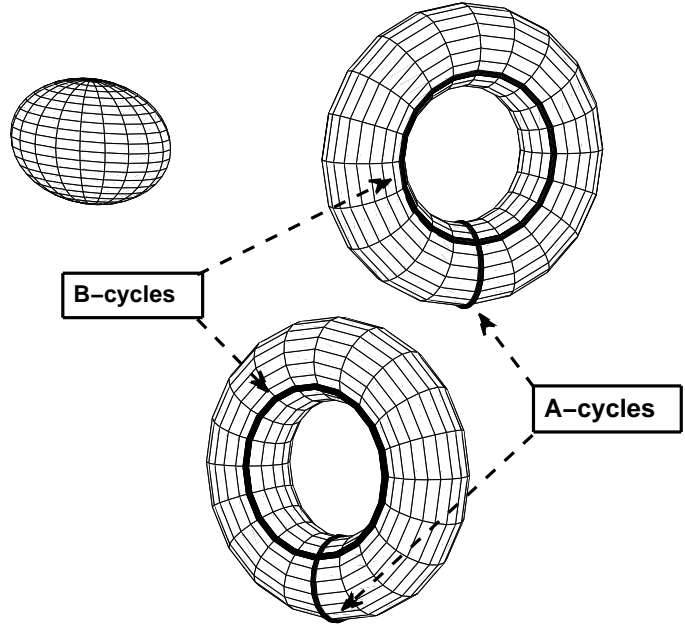,width=12cm}}
\caption{A multiply connected domain $D$, consisting of three components (two
  of genus 1 and one of genus 0), with exterior $\Omega$.  Cutting along the
  surfaces bounded by the ``A-cycle'' from $D$ makes it simply connected.
  Adding the surfaces $T_j$ bounded by the B-cycles makes $\Omega \setminus D$
  simply connected.  }\label{fig1}
\end{figure}

\begin{definition}
\label{meanzerodef}
Let $\Gamma_j$ denote a component of the boundary $\Gamma$.
If 
\[ \int_{\Gamma_j} f(\bx) dA(\bx) = 0 \]
we refer to it as having {\em mean zero} on that component.
We denote by $\cM_{\Gamma,0}$ the set of pairs of functions $(f,g)$ on
$\Gamma$ with mean zero on every component.
\end{definition}

\begin{lemma} [The mean zero condition]
\label{meanzerolemma}
Let $(r,q)$ be generalized Debye sources defined on a 
boundary $\Gamma$. Then $(r,q) \in \cM_{\Gamma,0}$.
\end{lemma}
\noindent
This is proven in section
\ref{potbdrforms} (see eqs. (\ref{eqn53}) and (\ref{eqn96a})).

In multiply connected domains, the Helmholtz decomposition
(\ref{helmdecompsimple}) is incomplete. From Hodge theory,
however, we can write a surface vector field $\bj$ as an orthogonal
decomposition (in $L_2$), the Hodge-Helmholtz decomposition:
\begin{align}
\label{helmdecompmc}
\bj &= \bj_R + \bj_H \\
\bj_R &= \nabla_{\Gamma} \Psi + \bn \times \nabla_{\Gamma} \Psi_m  \nonumber 
\end{align}
for some $\Psi, \Psi_m$,
where $\bj_H$ satisfies 
\[ \nabla_{\Gamma} \cdot \bj_H = 0, \quad
 \nabla_{\Gamma} \cdot (\bn \times \bj_H) = 0. 
\]
See Appendix~\ref{adjintprt} and~\cite{Vick}.
Such vector fields are called {\em harmonic vector fields} (dual to harmonic
1-forms).  We let
\begin{eqnarray*}
\bm_R &=& \bn \times \bj_R \\
\bm_H &=& \bn \times \bj_H 
\end{eqnarray*}
and (as before)
\[ \bm = \bn \times \bj \, . \]
Harmonic vector fields arise, in essence, because the Laplace-Beltrami operator
on vector fields
\begin{equation}
  \Delta_{\Gamma}^1\bj \equiv \nabla_{\Gamma}\nabla_{\Gamma}\cdot\bj-
\bn\times\nabla_{\Gamma}\nabla_{\Gamma}\cdot(\bn\times\bj),
\end{equation}
has a non-trivial nullspace on multiply connected surfaces.  The dimension of
the nullspace of $\Delta_{\Gamma}^1$ is equal to twice the genus, $g,$ of the
surface. We may, therefore, choose harmonic vector fields $\{ \bj_H^l \, | \, l
= 1,\dots, 2g \},$ which form an orthogonal basis,w.r.t. $L^2(\Gamma),$ for
this nullspace.  In the multiply connected case, we can then define the
harmonic components of $\bj$ and $\bm$ by
\begin{align}
\label{harmonicfieldsjH}
\bj_H &= \sum_{l=1}^{2g} c_l \bj_H^l \\
\bm_H &= \bn \times \bj_H \nonumber \, .
\end{align}
The space of harmonic vector fields will be denoted $\cH^1(\Gamma)$.

Given that the Laplace-Beltrami operator is not invertible, 
one must be careful in defining $\Psi$ and $\Psi_m$.
From Hodge theory, however, we know that it is invertible as
a map from the space of mean zero functions $\cM_{\Gamma,0}$ to itself. 
We denote by $R_0$ the {\em partial inverse} of $\Delta_{\Gamma}$ acting on this
space.  Thus, we replace (\ref{surflapsolves}) with
\begin{eqnarray}
\label{surflapsolvesmc}
\Psi &=& ik R_0  r \\
\Psi_m &=& -ik R_0 q \, , \nonumber 
\end{eqnarray}
where $r,q$ are the generalized Debye sources.

\begin{example}
Consider a torus in $\bbR^3$, parametrized by 
\[ \bx(s,t) = [ (R+r \cos t) \cos s, (R+r \cos t) \sin s, 
r \sin t ] \, ,
\]
with the $z$-axis as the axis of symmetry.
A straightforward calculation shows that
\[ \bj^1_H = \frac{1}{(R+r \cos t)^2} \frac{\partial \bx(s,t)}{\partial s} 
\]
and 
\[ \bj^2_H = \bn \times \bj^1_H
\]
are both harmonic vector fields.
Since the genus of a torus is 1, they form a basis for the 
two-dimensional space of harmonic vector fields on the surface.
\end{example}

\begin{remark}
  Much of the formal analysis in this paper is simplified through the use of
  differential forms and homology theory. In order to be accessible to a
  broader audience, however, we state the main results using the notation of
  vector calculus and defer most proofs to Sections~\ref{extfrm1}
  and~\ref{MEfrms2}, where we do make use of the language and power of this
  theory. 
\end{remark}

\section{Uniqueness Theorems for Exterior Electromagnetic Fields}
\label{sec:uniq}

Let us denote by $\Zup$ the closed upper half plane:
\[ \Zup = \{z\in\bbC:\: \Im z\geq 0\}. \]

\begin{definition}
A solution to the time harmonic Maxwell equations in $\Omega$ that
satisfies the Silver-M\"{u}ller radiation condition will be referred to 
as an {\em outgoing solution}.
\end{definition}

That an outgoing solution to \THME[$k$] is determined by either the
tangential components of the electric or magnetic fields is classical
\cite{ColtonKress}:

\begin{theorem}\label{thm4class} 
Suppose that $(\bE,\bH)$ is an outgoing
solution to the \THME[$k$] in an exterior region $\Omega$ 
for nonzero  $k \in \Zup.$ If either
$\bn \times \bE$ or $\bn \times \bH$ vanishes on $\Gamma$, 
then the solution is identically zero in $\Omega.$
\end{theorem}

Although we will eventually address the problem of 
scattering from a perfect conductor ($\bn \times \bE = 0$),
we turn our
attention for the moment to the Maxwell equations in exterior domains
with normal components specified on the boundary. While this
is not a standard physical boundary value problem, there is prior work
on uniqueness and it is a natural starting point for the analysis of 
symmetric representations of the fields.

\begin{theorem} \label{thmyee} [Yee, 1970].
Let $(\bE,\bH)$ be an outgoing
solution to the \THME[$k$] in an exterior region $\Omega$
for nonzero  $k \in \Zup.$ 
Suppose $\Gamma$ is simply connected, and that
\[ \bn \cdot \bE \restrictedto_{\Gamma} = 0 \ , \ 
 \bn \cdot \bH \restrictedto_{\Gamma} = 0 \, . \] 
Then
$\bE = \bH = 0 \quad \hbox{in $\Omega$}.$
\end{theorem}

When the boundary has non-trivial topology, a rather subtle argument shows 
that, in general, this is not
true. In particular, if the sum of the genera of the boundary components of the
exterior domain is $g>0,$ then for all frequencies with non-negative imaginary
part, there is a $2g$-dimensional space of outgoing solutions to the THME with
vanishing normal components.
The existence of these fields was proven by Kress 
(Theorem \ref{thm4} below).

\begin{remark}
In the static (harmonic) case, this fact has been known for 
decades~\cite{Werner66}. More precisely, at 
$k=0$, the THME separate into the system
\[ \nabla \cdot \bE = 0,\ \nabla \times \bE = 0 \qquad
 \nabla \cdot \bH = 0,\ \nabla \times \bH = 0 \, , \]
solutions to which are called harmonic vector fields (if they decay 
at infinity). When their normal
components vanish on the boundary, they are called
harmonic Neumann fields. If their tangential components vanish, they are
called harmonic Dirichlet fields.
\end{remark}

\begin{theorem}\label{thm4} 
[Kress, 1986 \cite{Kress1}].
Suppose that $(\bE,\bH)$ is an outgoing solution to the
\THME[$k$] in the exterior region $\Omega,$ for nonzero $k \in \Zup .$ 
If every component of the boundary $\Gamma$ is simply connected, 
then the solution is determined by the normal components 
$\bn \cdot \bE \restrictedto_{\Gamma}$ and
$\bn \cdot \bH \restrictedto_{\Gamma}.$  
If the sum of the genera of the components of
$\Gamma$ equals $g> 0,$ then there is a subspace of outgoing solutions to
  \THME[$k$] with 
\begin{equation}\label{6.16.08.1}
\bn \cdot \bE \restrictedto_{\Gamma}=0\text{ and }
\bn \cdot \bH \restrictedto_{\Gamma}=0
\end{equation}
of dimension exactly $2g.$
\end{theorem}

\begin{lemma}\label{prop3}
[Kress, 1986 \cite{Kress1}].
  Let $(\bE,\bH)$ be a solution to the \THME[$k$] for nonzero $k\in \Zup,$ 
  in a region $\Omega$ with a smooth bounded boundary $\Gamma$. Then
  the normal components $(\bn \cdot \bE, \bn \cdot \bH)$ lie in 
 $\cM_{\Gamma,0}$.  
\end{lemma}

\begin{corollary}\label{kresscor}
[Kress, 1986 \cite{Kress1}].
Suppose that $(\bE,\bH)$ is an outgoing solution to the
\THME[$k$] with $k \in \Zup$ and satisfying 
\[
\bn \cdot \bE \restrictedto_{\Gamma}=f \text{ and }
\bn \cdot \bH \restrictedto_{\Gamma}=h 
\]
where $f,h \in \cM_{\Gamma,0}$ are H\"{o}lder continuous. 
Then $(\bE,\bH)$ are uniquely determined
subject to the specification of their circulations on the 
A-cycles (see Fig. \ref{fig1}):
\begin{equation}\label{kress_thm1a}
\int_{A_j} {\bf \tau} \cdot \bE \, ds = p_j \text{ and }
\int_{A_j} {\bf \tau} \cdot \bH \, ds = q_j\, ,
\end{equation}
where $p_j, q_j \in \bbC$ are given numbers.
\end{corollary}

\begin{remark} We call solutions to \THME[$k$] that satisfy~\eqref{6.16.08.1}
\emph{$k$-Neumann fields,} and denote the space of such solutions by
$\cH_k(\Omega).$ 
The conditions in Corollary \ref{kresscor} are familiar from the 
static (zero frequency) case,
where $g$ conditions must be specified for each of $\bE$ and 
$\bH$ separately, since the equations are uncoupled.
For nonzero $k$, this symmetry is not required. We provide
a different proof of existence in Theorem \ref{thm6.1}
and a somewhat more general analysis of uniqueness in section 
\ref{unMEfrm}.

Our representation also provides
an effective means for numerically computing the $k$-Neumann fields. 
These solutions will be needed
in solving the problem of scattering from a mutiply-connected 
perfect conductor.
\end{remark}

First, however, we need to recall some classical facts about layer
potentials.

\section{Jump Relations and Boundary Values of the Potentials}\label{sec5}

In order to use the integral representations discussed above to solve boundary
value problems, we need to find expressions for the restrictions of $\bE$ and
$\bH$ to the boundary, in terms of the various potentials.  In this section we
collect the relevant results.  Recall that $\Gamma\hookrightarrow\bbR^3$ is a
smooth, bounded surface (possibly disconnected). The unbounded component of
$\bbR^3\setminus\Gamma$ which we have denoted by $\Omega,$ will be referred to
as the ``$+$'' side of the boundary. The domain $D$ (the union of the bounded
components) will be referred to as the ``$-$'' side. We use $\bn$ to denote the
unit normal vector field along $\Gamma,$ pointing into the unbounded component.

The relevant limits are given in the following lemma, proofs of which
can be found, for example, in~\cite{ColtonKress}.

\begin{lemma}\label{lemma2} 
  Let $\bA$ and $\phi$ denote the vector and scalar potentials in
  (\ref{Vecpotdef}) and let $\Gamma$ be a smooth bounded surface in $\bbR^3.$
  For $\bx_0\in\Gamma,$ let ${\bx\to \bx_0^{\pm}}$ indicate approach from
  $\Omega$ ($+$) or $D$ ($-$), respectively, and let $\bn_0$ denote the normal
  at $\bx_0$, with $\frac{\partial}{\partial n_0}$ the normal derivative
  at $\bx_0$..  Then, for the scalar potential, we have
\begin{align}
\label{eq:1.23.1}
    \lim_{\bx\to \bx_0^{\pm}} \bn \cdot \nabla \phi_{\pm}(\bx) &=
    \mp \frac{1}{2}r(\bx_0)+K_0[r](\bx_0) \\
    \lim_{\bx\to \bx_0^{\pm}} \bn \times \nabla \phi_{\pm}(\bx) &=
    K_1[r](\bx_0), \nonumber
\end{align}
where 
\begin{align*}
K_0[r](\bx_0) &= \int_\Gamma \frac{\partial g_k}{\partial n_0}(\bx_0 - \bx) \,
r(\bx) \, dA(\bx) \\
K_1[r](\bx_0) &= \bn_0 \times \nabla \int_\Gamma g_k(\bx_0 - \bx) \, 
r(\bx) \, dA(\bx) \, ,
\end{align*}
$K_0$ is an integral operator of order $-1$ and $K_1$ is an integral
operator of order $0$, which is defined in a principal value sense. 

For the vector potential we have
\begin{align}
    \label{eq:1.23.2}
    \lim_{\bx\to \bx_0^{\pm}} \bn \cdot \bA_{\pm}(\bx) &=
    K_{2,n}[\bj](\bx_0), \\
    \lim_{\bx\to \bx_0^{\pm}} \bn \times \bA_{\pm}(\bx) &=
    K_{2,t}[r](\bx_0), \nonumber
\end{align}
where 
\begin{align*}
K_{2,n} [\bj](\bx_0) &= \int_\Gamma  g_k(\bx_0 - \bx) \, 
(\bn_0 \cdot \bj(\bx)) \, dA(\bx) \, , \\
K_{2,t} [\bj](\bx_0) &= \int_\Gamma   g_k(\bx_0 - \bx) \, 
(\bn_0 \times \bj(\bx)) \, dA(\bx) \, .
\end{align*}
$K_{2,n}$ and $K_{2,t}$ are both 
integral operators of order $-1$.

The vector potential also satisfies
\begin{align}
    \lim_{\bx\to \bx_0^{\pm}} \bn \cdot \nabla \times \bA_{\pm}(\bx) &=
     K_3[\bj](\bx_0) \\
  \label{eq:1.23.3} 
    \lim_{\bx\to \bx_0^{\pm}} \bn \times \nabla \times \bA_{\pm}(\bx) &=
    \pm \frac{1}{2} \bj(\bx_0) + K_4[\bj](\bx_0) \, , \nonumber
\end{align}
where 
\begin{align*}
K_3[\bj](\bx_0) &= \int_\Gamma 
 \nabla g_k(\bx_0 - \bx) \cdot (\bj(\bx) \times \bn_0)   
\, dA(\bx)  \, , \\
K_4[\bj](\bx_0) &= \int_\Gamma 
\left[ \nabla g_k(\bx_0 - \bx) \, (\bj(\bx) \cdot \bn_0) -  
\frac{\partial g_k}{\partial n_0} (\bx_0 - \bx) \, \bj(\bx) \right] 
\, dA(\bx) \, .
\end{align*}
$K_3$ is an integral operator of order $0$.
$K_4$ is an integral operator of order $-1$.
\end{lemma}

\begin{proof}
These results follow from classical potential theory, the observation that
the kernels in $K_{2,n}$ and  
$K_{2,t}$ are weakly singular, and the fact that,
at the singular
point $\bx_0 = \bx$ in $K_4$, $\bn_0$ is orthogonal to $\bj(\bx)$).
\end{proof}

\begin{remark}
We have abused notation slightly in the preceding Lemma. The operators
are functions of the Helmholtz parameter $k$. When the explicit dependence
is relevant, we will occasionally write
$K_0(k), K_1(k), ...,K_4(k)$ instead.
\end{remark}

The limits of the anti-potentials $\bA_m$ and $\phi_m$ are
analogous.  Recall, however, that we have chosen not
to work with $\bj$ and $\bm$ as unknowns, but rather 
the generalized Debye sources $r,q$ complemented by the harmonic
vector fields. $\bj$ and $\bm$ 
are computed from
\[ \bj = \nabla_{\Gamma} \Psi + \nabla_{\Gamma} \times (\bn \Psi_m) + \bj_H \]
\[ \bm = \bn \times \bj \]
where $\Psi$ and $\Psi_m$ satisfy the Laplace-Beltrami equations
(\ref{surflapsolvesmc}) with $r,q$ viewed as source data.

\begin{lemma}\label{lemma2a} 
The integral operators $K_2, K_3$ and $K_4$ are all of
order $-1$ or $-2$, and hence compact, 
when viewed as operators acting on $r,q$.
\end{lemma}

\begin{proof} 
This follows immediately from Lemma \ref{lemma2} and the fact that
$\bj$, $\bm$ are of order $-1$ in terms of $(r,q)$.
\end{proof} 

From Lemma \ref{lemma2}, we obtain the following jump relations:

\begin{corollary}\label{thm3} Suppose that the fields $\bE, \bH,$  
are defined in terms of potentials and anti-potentials. Then they satisfy
\begin{equation}\label{jmprels}
\begin{matrix} \bn \cdot (\bE_{+} - \bE_{-}) = r \\
\bn \cdot (\bH_{+} - \bH_{-}) = q \end{matrix}
\qquad
\begin{matrix}
\bn \times (\bE_{+} - \bE_{-}) = -\bm \\
\bn \times (\bH_{+} - \bH_{-}) =  \bj
\end{matrix} \, .
\end{equation}
\end{corollary}

\section{The Maxwell Equations with Normal Components Specified}
\label{normalcompsec}

We note that, for $k \in \Zup$, the fundamental
solution
\begin{equation}
 g_k(\bx) = \frac{e^{ik |\bx|}}{ 4 \pi | \bx|} 
\end{equation}
is outgoing; that is, it satisfies the Silver-M\"{u}ller radiation condition.
Thus, all of the corresponding potentials defined over bounded regions are 
outgoing as well. 

\begin{theorem}\label{thm1} 
Let $\bE$ and $\bH$ be outgoing fields represented in terms of 
Debye sources $(r,q)$ and currents $\bj,\bm$.
Then the limiting values of their normal components are given 
by the following integral representation for $\bx_0 \in \Gamma$.
\begin{multline}
\lim_{\bx\to \bx_0^{\pm}} \left(\begin{matrix} \bE \cdot \bn\\
  \bH \cdot \bn \end{matrix}\right)
=\\
\left(\begin{matrix} \pm \frac{1}{2}I -K_0&0\\
0& \mp \frac{1}{2}I +K_0\end{matrix}\right)
\left(\begin{matrix} r\\q\end{matrix}\right)+
\left(\begin{matrix} i k K_{2,n} & - K_3\\ 
 K_3 & ik K_{2,n} \end{matrix}\right)
 \left(\begin{matrix} \bj\\
  \bm\end{matrix}\right) \, ,
\label{eqn74}
\end{multline}
where $I$ denotes the identity operator.
If we assume that $\bj=\bj_R(r,q,k)$ and $\bm=\bn \times \bj,$ 
then these are Fredholm integral operators of the second kind in the  
generalized Debye sources
$(r,q)$. 
As $k$ tends to zero, this representation converges to
\begin{equation}
\lim_{\bx\to \bx_0^{\pm}} \left(\begin{matrix} \bE \cdot \bn \\
  \bH \cdot \bn \end{matrix}\right)
=\\
\left(\begin{matrix} \pm \frac{1}{2}I -K_0(0)&0\\
0&\mp \frac{1}{2}I+K_0(0)\end{matrix}\right)\left(\begin{matrix}
r\\q\end{matrix}\right)
\end{equation}
The equation $[-I+2K_0(0)]f=h$ is uniquely solvable for all $h,$ and
the  equation $[I+2K_0(0)]f=h$ for all $h$ of mean zero.
\end{theorem}

\begin{proof}
  The equations follow from the formul{\ae} in the previous section. The
  solvability properties of $I\pm 2K_0(0)$ are classical and can be found
  in~\cite{ColtonKress}.
\end{proof}

\begin{definition}
\label{cndef}
We let $\cN^{\pm}(k)$ denote the operator on the right hand side
of~\eqref{eqn74}, with $\cN^{\pm}_{\bE}(k)$ the first row and
$\cN^{\pm}_{\bH}(k)$ the second. 
\end{definition}

If we seek to impose the boundary conditions
$\bE \cdot \bn\restrictedto_{\Gamma} = f$, 
$\bH \cdot \bn\restrictedto_{\Gamma} = h$,
then
we obtain the following system of equations, 
which is analytic in $k.$ 

\begin{multline}
\left(\begin{matrix} \pm \frac{1}{2}I -K_0&0\\
0& \mp \frac{1}{2}I +K_0\end{matrix}\right)
\left(\begin{matrix} r\\q\end{matrix}\right)+
\left(\begin{matrix} i k K_{2,n} & - K_3\\ 
 K_3 & ik K_{2,n} \end{matrix}\right)
 \left(\begin{matrix} \bj\\
  \bm\end{matrix}\right) \, 
= \left(\begin{matrix} f\\ h  \end{matrix}\right)
  \label{eq:81n}
\end{multline}
When $\bj=\bj_R(r,q,k)$ and $\bm=\bn \times \bj$ are obtained from 
the Debye sources via (\ref{surflapsolvesmc}), 
we denote the left-hand side of
(\ref{eq:81n}) by $\cN^{\pm}(k)\left(r, q\right)$ and observe that 
it is a compact perturbation of the operator $J_{\pm}(0)$ where
\begin{equation}
  J_{\pm}(k)=\left(\begin{matrix} \pm \frac{1}{2}I-K_0(k)&0\\
0&\mp \frac{1}{2}I+K_0(k)\end{matrix}\right).
\end{equation}

\begin{theorem} 
\label{mcnormalthm}
Let $(f,h)\in\cM_{\Gamma,0}.$ For $k\notin E_{+},$
a discrete set in the complex plane, the equation
  \begin{equation}
  \label{frednormal}
    \cN^+(k)(r,q)=(f,h)
  \end{equation}
has a unique solution. 
The outgoing solution of the \THME[$k$] defined by this data
  satisfies
  \begin{equation}
    \label{eq:nrmbc2}
 \bE_{+} \cdot \bn\restrictedto_{\Gamma}=f\quad 
 \bH_{+}   \cdot \bn\restrictedto_{\Gamma}=h.
  \end{equation}
\end{theorem}

\begin{proof}
We know from Theorem \ref{thm1} that the operator 
$J_{+}(0)$ is invertible. By analytic Fredholm theory, therefore,
we know that there is a discrete set $E_+ \in \bbC$ so that for 
$k \notin E_+$, 
$\cN^+(k)$ is also invertible. The result now follows from the fact that
$(r,q)\in\cM_{\Gamma,0}$ (Lemma \ref{meanzerolemma}) and that 
$(f,h)\in\cM_{\Gamma,0}$ (Lemma \ref{prop3}), so that
\[ 
\cN^+(k): \cM_{\Gamma,0} \rightarrow \cM_{\Gamma,0}.
\]
\end{proof}

\begin{corollary} 
\label{scnormalthm}
When $\Gamma$ is simply connected,
the Fredholm equation (\ref{frednormal}) provides a unique solution
to the \THME[$k$] for any $k \in \Zup$.
\end{corollary} 

\begin{proof}
This follows immediately from Theorem
\ref{mcnormalthm} and the fact that
$E_{+}$ lies in the complex lower half-plane, as follows from Corollary~\ref{cor4} in 
Section \ref{MEfrms2}. 
\end{proof}

\begin{remark}
\label{cfkress}
The problem of solving the THME with specified normal components was
previously analyzed by G\"{u}lzow \cite{Gulzow}, who constructed
a hypersingular integral equation method.
Using generalized Debye sources instead leads to a well-conditioned
integral equation of the second kind. 
In the multiply connected case,
Theorem \ref{mcnormalthm} shows that the problem of scattering with normal
components specified is invertible, if the solution is sought
with zero projection onto the harmonic vector fields. The extra
conditions in Kress' result (Corollary \ref{kresscor}) force uniqueness 
on $\bj_H$ by specifying $g$ conditions on 
the circulation of $\bE$ and $g$ conditions on the circulation of 
$\bH$, where $g$ is the genus of $\Gamma$.
In fact, any set of $2g$ conditions that uniquely determines the 
harmonic component of the current, $\bj_H$, will suffice.
(see Theorems \ref{thm4.1} and \ref{thm6.1}).
\end{remark}

\section{The Perfect Conductor}

We turn now to the problem of scattering from a perfect conductor, 
which requires an analysis of the 
tangential components of $\bE$ and $\bH$. These are easily
expressed in terms of potentials using 
Lemma \ref{lemma2}. 

\begin{theorem} 
\label{tanlimits}
Let $\bx_0 \in \Gamma$.
The limiting values of the tangential components of $\bE,\bH$, in terms
of potentials and antipotentials, are given by
\begin{multline}
\lim_{\bx\to \bx_0^{\pm}} \left(
\begin{matrix} \bn \times \bE \\ \bn \times \bH \end{matrix}\right)=\\
\frac{1}{2}\left(\begin{matrix}\pm \bm\\\mp\bj\end{matrix}\right)+
\left(\begin{matrix} -K_1 & 0 & i k K_{2,t} & 
- K_4\\
0&-K_1 &  K_4 & i k K_{2,t} \end{matrix}\right)
\left(\begin{matrix} r\\ q\\ \bj \\
  \bm\end{matrix}\right).
\label{eqn88}
\end{multline}
where $K_1, K_{2,t}, K_4$ are defined in Lemma \ref{lemma2}. 
\end{theorem}

\begin{definition} 
\label{ctdef}
In the sequel we denote the operator on the right hand side
of~\eqref{eqn88} as $\cT^{\pm}(k).$ We use $\cT^{\pm}_{\bE}(k)$ to denote the
first row, and $\cT^{\pm}_{\bH}(k)$ to denote the second. 
\end{definition}

For scattering from a perfect conductor, let us recall that
both (\ref{EHbcscat}) and 
(\ref{EHbcscat2}) must hold on $\Gamma$. 
As noted in the introduction, 
the EFIE approach involves imposing (\ref{EHbcscat})
using only the classical vector and scalar potentials, $\bA$ and $\phi$,
with the physical current $\bJ$ as the unknown. This
leads to an integral equation
on $\Gamma$ with a hypersingular kernel that has interior resonances
and suffers from low-frequency breakdown.
To avoid these difficulties, we turn again to the Debye sources.
A nonstandard feature of our approach is that we will
extract only one scalar equation
from the tangential conditions on the $\bE$ field and couple it to the 
normal condition (\ref{EHbcscat2}) satisfied by $\bH$.

\subsection{The hybrid system}\label{sec7.1}

The operator defining the tangential components of $\bE_{\pm}$ is given by
\begin{equation}
  \cT^{\pm}_{\bE}(k)\left(\begin{matrix} r\\q\\\bj\\ \bm\end{matrix}\right)
=
\pm \frac{1}{2} \bm +
\left(\begin{matrix} -K_1 & 0 & ik K_{2,t} & -K_4\end{matrix}\right)
\left(\begin{matrix} r\\q\\ \bj \\
  \bm\end{matrix}\right).
\label{eqn888}
\end{equation}
If we restrict to $\bj=\bj_R(r,q,k),$ and $\bm=\bn\times\bj,$ then, acting on
$(r,q)\in\cM_{\Gamma,0},$ the only term of non-negative order is $-K_1 r.$ We
use $\cT^{\pm}_{\bE}(k)(r,q)$ to denote this operator restricted to this
subspace of data. In order to recast $K_1$ in (\ref{eqn888}) as a Fredholm
operator of the second kind, it is convenient to multiply it by a left
parametrix, based on the following standard result.

\begin{lemma} Let $\Gamma\subset \bbR^3$ be a smooth, connected 
closed surface, let $\bx_0 \in \Gamma$, and let $G_0$ denote the 
single layer potential operator based on the Green's function for the 
Laplace equation: 
  \begin{equation}
  \label{slplap}
   G_0[f](\bx_0) = \int_{\Gamma} g_0( \bx - \bx_0)
   f(\bx) \, dA(\bx) \, ,
  \end{equation}
and let $\phi$ denote the usual scalar potential with density $r$. 
Then 
\begin{equation} 
\lim_{\bx\to \bx_0^{\pm}} 
  G_0 [ \Delta_{\Gamma} \phi ](\bx) =
\frac{1}{4} r(\bx_0) + \widetilde{K_1}(\bx_0) \, 
\end{equation} 
where $\widetilde{K_1}$ is a compact operator.
\end{lemma}

\begin{proof} 
To see this, we observe  that the surface Laplacian satisfies the 
identity 
\[ \Delta_{\Gamma} \phi = \Delta \phi  - 2H \frac{\partial \phi}{\partial n} - 
\frac{\partial^2 \phi}{\partial n^2} \, , \] 
where $H$ denotes mean curvature
(see, for example, \cite{Nedelec}).
Since $\Delta \phi = - k^2 \phi$ (by construction), we have
\begin{equation}
\label{calderon1}
\Delta_{\Gamma} \phi = - k^2 \phi  - 2H \frac{\partial \phi}{\partial n} - 
\frac{\partial^2 \phi}{\partial n^2} \, .  
\end{equation}
The composition of $G_0$ with the first two terms on the right-hand side of 
(\ref{calderon1}) are of order $-2$ and $-1$, respectively, hence compact.
It is a classical result (a {\em Calderon relation}) that the composition
of a single layer potential with the second normal derivative of $\phi$
yields a compact perturbation of $\frac{1}{4}I$ \cite{Nedelec}:
\[ 
\lim_{\bx\to \bx_0^{\pm}} 
G_0 [ \frac{\partial^2 \phi}{\partial n^2} ] (\bx) = 
\frac{1}{4} r(\bx_0) + D^2[r](\bx_0) \, 
\]
where 
\[
D[r](\bx_0) = \int_\Gamma \frac{\partial g_k}{\partial n_x}(\bx_0 - \bx) \,
r(\bx) \, dA(\bx) \, .
\]
Here 
$\bn_{\bx}$ denotes the normal at $\bx$,
and $\frac{\partial}{\partial n_x}$ denotes the normal derivative at $\bx$.
$D$ is the usual double layer potential, which is the adjoint of $K_0$.
Since $D$ is an operator of order $-1$, the result follows.
\end{proof}

\begin{lemma} 
\label{Tprojeq} Let $\Gamma\subset \bbR^3$ be a smooth, connected 
closed surface, let $\bx_0 \in \Gamma$, let $G_0$ denote the 
single layer potential operator (\ref{slplap}). 
Then 
\begin{equation} 
\lim_{\bx\to \bx_0^{\pm}} 
  G_0 \nabla_{\Gamma} \cdot [ \bn \times \cT^{\pm}_{\bE}(k)]
\left(\begin{matrix} r\\ q \end{matrix}\right) \, 
   =
 -\frac{1}{4} r(\bx_0) + N_1(k) \,
\left(\begin{matrix} r\\ q \end{matrix}\right) \, ,
\end{equation} 
where $N_1(k)$ is an analytic family of operators of order -1.
\end{lemma}

\begin{proof}
The result follows from Definition \ref{ctdef}, the fact that
\[ \nabla_{\Gamma} \cdot \bn \times K_1[r] = - \Delta_{\Gamma} \phi[r] \, ,
\]
and the preceding Lemma.
\end{proof}

Recalling that $$\bj_R(r,q,0)=0,$$ we see that
\begin{equation}
  G_0 \nabla_{\Gamma} \cdot [\bn \times \cT^{\pm}_{\bE}](0)\left(\begin{matrix}
  r\\q
\end{matrix}\right)= G_0 [\Delta_{\Gamma} G_0 \, r].
\end{equation}
If $\Gamma$ has $M$ components, then the nullspace of this operator is
$M$-dimensional. It is generated by functions $r$ such that $G_0r$ is constant
on each component of $\Gamma.$ As a consequence of Theorem 5.7
in~\cite{ColtonKress}, it follows that this nullspace only intersects
$\cM_{\Gamma,0}$ at $0.$ 

\begin{definition}
Taking the integral operator in Lemma \ref{Tprojeq} and
the integral operator $\cN_{\bH}$ from Definition \ref{cndef},
we define  $\cQ^{\pm}(k)$ as the following hybrid system of integral
operators:
\begin{equation}
\label{hybridopQ}
  \cQ^{\pm}(k)\left(\begin{matrix} r\\q\end{matrix}\right)=
\left(\begin{matrix}
   G_0 \nabla_{\Gamma} \cdot [\bn \times \cT^{\pm}_{\bE}](k)\\\cN_{\bH}^{\pm}(k)
\end{matrix}\right) \left(\begin{matrix} r\\q\end{matrix}\right).
\end{equation}
As divergences, the range of $\nabla_{\Gamma} \cdot [\bn \times \cT^{\pm}_{\bE}](k)$ 
consists of functions of mean zero. 
\end{definition}

\begin{proposition} 
\label{propFplus}
The family of operators $\cQ^{\pm}(k)$ is analytic in $k$
  and Fredholm of the second kind. There is a discrete subset $F_+\subset \bbC,$
  such that $\cQ^{+}(k):\cM_{\Gamma,0}\to\cM_{\Gamma,1}$ 
  is invertible for $k\notin F_+,$ where
$\cM_{\Gamma,1}$ is the $L^2$-closure of
\begin{equation}
  \{(G_0r,q):\:(r,q)\in\cM_{\Gamma,0}\}.
\end{equation}
\end{proposition}

\begin{proof} The analyticity statement is immediate from the
  formula. Examining $\cQ^{\pm}(k)$ we see that
  \begin{equation}\label{hybreqn1}
     \cQ^{\pm}(k)\left(\begin{matrix} r\\q\end{matrix}\right)=
\left(\begin{matrix} \frac{-1}{4}&0\\0&\frac{\mp1}{2}\end{matrix}\right)
\left(\begin{matrix} r\\q\end{matrix}\right)+\tN^{\pm}_1(k)\left(\begin{matrix}
    r\\q\end{matrix}\right),
  \end{equation}
where $\tN_1^{\pm}$ is an analytic family of operators of order $-1.$ Thus,
$\cQ^{\pm}(k)$ is Fredholm of the second kind. The last statement
follows from the facts that 
$\cQ^{+}(0):\cM_{\Gamma,0}\to\cM_{\Gamma,1}$ is
invertible and $\cQ^{+}(k)\cM_{\Gamma,0}\subset\cM_{\Gamma,1}.$ 
\end{proof}

We may now state our principal result in the simply connected case.

\begin{theorem}\label{skie_simpcconn}
If $\Gamma$ is simply connected, then
  $F_+$ is disjoint from the closed upper half plane.
Thus, the integral equation 
\begin{equation}
  \cQ^{+}(k)
  \left(\begin{matrix} r\\q\end{matrix}\right) =
  \left(\begin{matrix} f\\h\end{matrix}\right)
\end{equation}
provides a unique solution to the scattering problem from a perfect
conductor for any $k$ in the closed upper half plane. Here,
\begin{equation}\label{set_data}
f=G_0 \nabla_{\Gamma} \cdot [\bn \times \bn \times \bE^{\In}], \qquad
h= \bn \cdot \bH^{\In} \restrictedto_{T\Gamma}.
\end{equation}
\end{theorem}
\noindent
This is proved as Theorem~\ref{skie_simpcconn.1} in Section~\ref{sec7.1.1}. We
leave the discussion of applying our method in the non-simply connected case to
Section~\ref{sec7.1.1}.

\subsection{Low Frequency Behavior in the Simply Connected Case}
The representation of solutions to the \THME[$k$], using data from
$\cM_{\Gamma,0}\oplus\cH^1(\Gamma),$ afforded by~\eqref{eqn29} behaves well as
the frequency tends to zero.  In the simply connected case we only have data
from $\cM_{\Gamma,0}.$ As $k$ tends to zero, this space of solutions tends to the
orthogonal complement of the harmonic Dirichlet fields, that is outgoing
harmonic fields, with vanishing tangential components on $b\Omega$, 
so that $\bE$ and $\bH$ are recovered from $-\nabla \phi$ and 
$-\nabla \phi_M$ alone.
This is
proven, along with the multiply connected case in Section~\ref{lowfrq2}. It is
therefore apparent that $\cQ^+(k)$ provides a means for finding and
representing solutions to the perfect conductor problem, which has neither
interior resonances, nor suffers from low frequency breakdown.

\section{The Exterior Form Representation}\label{extfrm1}
For the remainder of this paper we represent the electric field $\bE$ as a
1-form $\bxi,$ and the magnetic field $\bH$ as a 2-form, $\bEta.$ This choice
is explained in Appendix~\ref{app1}. If
\begin{equation}
  \bE=e_1\pa_{x_1}+e_2\pa_{x_2}+e_3\pa_{x_3}\text{ and }
\bH=h_1\pa_{x_1}+h_2\pa_{x_2}+h_3\pa_{x_3},
\end{equation}
then
\begin{equation}
\begin{split}
  &\bxi=e_1dx_1+e_2dx_2+e_3dx_3\quad\text{ and }\\
&\bEta=h_1dx_2\wedge dx_3+h_2dx_3\wedge dx_1+h_3dx_1\wedge dx_2.
\end{split}
\end{equation}
If $\langle\cdot,\cdot\rangle$ denotes the Euclidean inner product, i.e. the
metric on $\bbR^3,$ then $\bxi$ is defined by the condition that, for every
vector field $\bV$ we have:
\begin{equation}
  \langle\bV,\bE\rangle=\bxi(\bV).
\end{equation}
That is, $\bxi$ is the metric dual of $\bE,$ and $\star\bEta$ (with $\star$ the
Hodge star-operator, see Remark~\ref{adptfrm} and Appendix~\ref{A2}) is the
metric dual of $\bH.$

The curl-part of time harmonic Maxwell's
equations takes the form:
\begin{equation}
  \label{eq:ME1}
  d\bxi=ik\bEta\quad d^*\bEta=-ik\bxi.
\end{equation}
For $k\neq 0,$ these equations imply the divergence equations, which take the
form:
\begin{equation}
  \label{eq:ME2}
  d^*\bxi=0\quad d\bEta=0.
\end{equation}
An outgoing solution to the Helmholtz equation on 1-forms satisfies the
analog of the Silver-M\"uller radiation conditions:
\begin{equation}
  \label{eq:radcond}
  i_{\hbx}d\bxi-d^*\bxi\,\hbx\cdot d\bx-ik\bxi=o\left(\frac{1}{|x|}\right),
\end{equation}
where $\hbx=\frac{\bx}{\|\bx\|}.$  A magnetic field,
$\bEta$ is outgoing if $\star\bEta$ satisfies~\eqref{eq:radcond}.

The standard integration by parts formula in electromagnetic theory
is derived by considering
the $L^2$-norm of the quantity in the radiation condition. We let $D_R$ denote
the ball of radius $R,$ centered at $0,$ $S_R=bD_R,$ and
$\Omega_R=D_R\cap\Omega.$ Since the quantity in~\eqref{eq:radcond} is
$o(|x|^{-1})$ it follows easily that
\begin{equation}
  \lim_{R\to\infty}\int\limits_{S_R}\|i_{\hbx}d\bxi-d^*\bxi\,\hbx\cdot d\bx-ik\bxi\|^2dA=0.
\end{equation}
We expand the integrand to obtain that
\begin{equation}
\begin{split}
   \lim_{R\to\infty}\Bigg[
\int\limits_{S_R}\left[\|i_{\hbx}d\bxi-d^*\bxi\,\hbx\cdot
  d\bx\|^2+|k|^2\|\bxi\|^2\right]dA+\\
-2\Re \left(ik \int\limits_{S_R}\langle\bxi,
  i_{\hbx}d\bar{\bxi}\rangle-\langle i_{\hbx}\bxi,d^*\bar{\bxi}\rangle dA\right)\Bigg] =0.
\end{split}
\label{eq:radcond2}
\end{equation}
Using Green's formula, we can replace the second integral with
\begin{multline}
  \label{eq:grnsfrm2}
  2\Im(k)\int\limits_{\Omega_R}\left[\|d\bxi\|^2+\|d^*\bxi\|^2+|k|^2\|\bxi\|^2\right]dV+
\\
2\Im\left[k\int\limits_{\Gamma}\left[\langle\bxi,i_{\bn}d\bar{\bxi}\rangle-
\langle i_{\bn}\bxi,d^*\bar{\bxi}\rangle\right]dA\right].
\end{multline}
Here  we use $\bn$ to denote the \emph{inward} pointing normal
field along $\Gamma.$

Combining~\eqref{eq:radcond2} and~\eqref{eq:grnsfrm2} we obtain the standard
integration by parts formula for outgoing solutions to the vector Helmholtz equation:
\begin{lemma}\label{lem1} If $\bxi$ is a 1-form defined in $\Omega$ satisfying
  $\Delta\bxi+k^2\bxi=0$ and the radiation condition~\eqref{eq:radcond}, with
  $\Im(k)\geq 0,$ then
  \begin{equation}
    \begin{split}
    \lim_{R\to\infty}
    \Bigg(&2\Im(k)\int\limits_{\Omega_R}\left[\|d\bxi\|^2+\|d^*\bxi\|^2+
|k|^2\|\bxi\|^2\right]dV+\\
\int\limits_{S_R}\big[\|i_{\hbx}d\bxi-d^*\bxi\,\hbx\cdot
  d\bx\|^2&+|k|^2\|\bxi\|^2\big]dA\Bigg)=\\
-&2\Im\left[k\int\limits_{\Gamma}\left[\langle\bxi,i_{\bn}d\bar{\bxi}\rangle-
\langle i_{\bn}\bxi,d^*\bar{\bxi}\rangle\right]dA\right]
\end{split}
 \label{eq:radcond4}
  \end{equation}
\end{lemma}

\begin{remark} If $(\bxi,\bEta)$ is a solution to the \THME[$k$] in $\Omega,$
  then the outgoing radiation condition can be rewritten as
  \begin{equation}
    i_{\hbx}\bEta-\bxi=o\left(\frac{1}{|x|}\right),
  \end{equation}
in agreement with~\eqref{silvermuller}.
\end{remark}

\subsection{Uniqueness for Maxwell's Equations}\label{unMEfrm}
If $(\bxi,\bEta)$ is a solution to the Maxwell system, then $d^*\bxi=0$ and the
boundary term in~\eqref{eq:radcond4} reduces to
\begin{equation}
  \label{eq:radcond5}
  -2\Im\left(k\int\limits_{\Gamma}\langle\bxi,i_{\bn}d\bar{\bxi}\rangle dA
\right).
\end{equation}
Let $\nu$ denote a 1-form defined along $\Gamma,$ which restricts to zero on
$T\Gamma$ and is normalized by $\nu(\bn)=1.$ 

\begin{remark}\label{adptfrm}
We use $\star$ to denote a Hodge star-operator, see Appendix~\ref{A2}.  In much
of the paper we need to distinguish between the Hodge star-operator acting on
forms defined on $\bbR^3$ and that acting on forms defined on surfaces in
$\bbR^3.$ We denote the $\bbR^3$-operator by $\star_3,$ and a surface
operator by $\star_2.$ Which surface is intended should be clear from the
context.

If $\Gamma$ is a smooth closed surface in $\bbR^3,$ which bounds a region $D,$
then it obtains an orientation from its embedding into $\bbR^3:$ let $\bn$ be
the outward pointing unit normal vector, and $X_1,X_2$ a local oriented
orthonormal frame for $T\Gamma.$ We let $\omega_1,\omega_2,\nu,$ be the local
co-frame for $T^*\bbR^3\restrictedto_{\Gamma},$ dual to $X_1,X_2,\bn.$ Note, in
particular that, $\omega_1(\bn)=\omega_2(\bn)=0.$ We say that the frame
$(X_1,X_2,\bn)$ (or co-frame $(\omega_1,\omega_2,\nu)$) is \emph{adapted to
  $\Gamma.$} The 1-form that is the metric dual of the vector field $aX_1+bX_2$
is $a\omega_1+b\omega_2.$

The volume form on $\bbR^3$ and area form on
  $\Gamma$ are given locally by
\begin{equation}
dV=\omega_1\wedge\omega_2\wedge\nu\quad dA=i_{\bn}dV=\omega_1\wedge\omega_2.
\end{equation}
In terms of the adapted frame, the Hodge star-operator on $\bbR^3$ is
\begin{equation}
\begin{split}
\star_3 1=\omega_1\wedge\omega_2\wedge\nu&\quad \star_3\omega_1\wedge\omega_2\wedge\nu=1\\
\star_3\omega_1=\omega_2\wedge\nu\quad\star_3\omega_2&=-\omega_1\wedge\nu\quad
\star_3\nu=\omega_1\wedge\omega_2\\
\star_3\omega_1\wedge\nu=-\omega_2\quad\star_3\omega_2\wedge\nu&=\omega_1\quad
\star_3\omega_1\wedge\omega_2=\nu.
\end{split}
\end{equation}
The Hodge star-operator on the surface (oriented as the boundary of $D$) is
given by
\begin{equation}
\begin{split}
\star_2 1=\omega_1\wedge\omega_2&\quad \star_2\omega_1\wedge\omega_2=1\\
\star_2\omega_1=\omega_2&\quad\star_2\omega_2=-\omega_1.
\end{split}
\end{equation}
It is useful to note that if $\bq$ is a 1-form defined on $\Gamma,$ then
$\star_2^2\bq=-\bq.$ If $\bv$ is a vector field tangent to $\Gamma$ and
$\omega,$ its metric dual, then $\star_2\omega$ is the metric dual of
$\bn\times\bv.$

  To emphasize the distinction between an exterior form acting on $T\bbR^3$
  restricted to a surface $\Gamma\subset\bbR^3$ and the restriction of this
  form to directions tangent to $\Gamma,$ we sometimes use
  $\balpha\restrictedto_{\Gamma}$ to denote the former notion of restriction,
  and $\balpha\restrictedto_{T\Gamma},$ the latter. For a 1-form,
  $\balpha,$ represented along $\Gamma$ in the adapted co-frame by
  $\balpha\restrictedto_{\Gamma}=a\omega_1+b\omega_2+c\nu,$ we have
\begin{equation}
  \balpha\restrictedto_{T\Gamma}\leftrightarrow a\omega_1+b\omega_2.
\end{equation}
We denote this latter restriction by $\balpha_t.$ We use the notation
$d_{\Gamma}$ to denote the exterior differential acting on forms on $\Gamma.$

For a 2-form
$\bBeta\restrictedto_{\Gamma}=a\omega_1\wedge\nu+b\omega_2\wedge\nu+
c\omega_1\wedge\omega_2,$ we have
\begin{equation}
  \bBeta\restrictedto_{T\Gamma}\leftrightarrow c\omega_1\wedge\omega_2.
\end{equation}
For a 3-form $\bgamma,$ dimensional considerations imply that
\begin{equation}
  \bgamma\restrictedto_{T\Gamma}\equiv 0.
\end{equation}
If $f$ is a 0-form, or scalar function, then
\begin{equation}
  f\restrictedto_{T\Gamma}=f\restrictedto_{\Gamma}.
\end{equation}
\end{remark}

The definition of the inner product on forms, the fact that $\langle
\nu,i_{\bn}d\bar{\bxi}\rangle=0,$ and the definition of $\star_2$ on $\Gamma,$
imply the identity
\begin{equation}
  \label{eq:bndtrm1}
   -2\Im\left(k\int\limits_{\Gamma}\langle\bxi,i_{\bn}d\bar{\bxi}\rangle dA
\right)=
 -2\Im\left(k\int\limits_{\Gamma}[\bxi\restrictedto_{\Gamma}]\wedge\star_2
[i_{\bn}d\bar{\bxi}\restrictedto_{\Gamma}]\right).
\end{equation}
Using this identity and Lemma~\ref{lem1} we can prove the basic uniqueness
theorems.  Note that $\bxi$ is a 1-form so $\bxi_t=\bxi\restrictedto_{T\Gamma}$
corresponds to the tangential components of $\bE,$ and
$i_{\bn}\bxi\restrictedto_{\Gamma},$ the normal component. The magnetic field
$\bEta$ is a 2-form and therefore $\bEta\restrictedto_{T\Gamma}$ gives the
normal component, and $[i_{\bn}\bEta]_t=i_{\bn}\bEta\restrictedto_{T\Gamma},$ gives
the data in the tangential components, corresponding to $n\times\bH.$ 

We restate the classical result (Theorem \ref{thm4class}) that 
an outgoing solution to \THME[$k$] is determined by either the tangential
components of the electric or magnetic fields in the language of forms:

\begin{theorem}\label{thm4class.1} 
Suppose that $(\bxi,\bEta)$ is an outgoing
  solution to the \THME[$k$], in $\Omega,$ for a $k\neq 0,$ in $\Zup.$ If either
  $\bxi_t,$ or $[i_{\bn}\bEta]_t$ vanish, then the solution is identically zero
  in $\Omega.$
\end{theorem}

\noindent
Kress' result (Theorem \ref{thm4}) on the normal components of $(\bxi,\bEta)$ 
is restated as
\begin{theorem}\label{thm4.1} Suppose that $(\bxi,\bEta)$ is an outgoing solution to the
  \THME[$k$], in $\Omega,$ for a $k\neq 0,$ in $\Zup.$ If every
  component of the boundary of $\Omega$ is simply connected, then the solution
  is determined by the normal components  $i_{\bn}\bxi\restrictedto_{\Gamma}$ and
  $\bEta\restrictedto_{T\Gamma}.$  If the sum of the genera of the components of
  $\Gamma$ equals $g> 0,$ then there is a subspace of outgoing solutions to
  \THME[$k$] with 
\begin{equation}\label{6.16.08.1.1}
i_{\bn}\bxi\restrictedto_{\Gamma}=0\text{ and }
  \bEta\restrictedto_{T\Gamma}=0
\end{equation}
of dimension $2g.$
\end{theorem}

\begin{remark} As noted above, we call solutions to \THME[$k$] that satisfy~\eqref{6.16.08.1.1}
  \emph{$k$-Neumann fields,} and denote the space of such solutions by
  $\cH_k(\Omega).$ Here we give a bound on $\dim\cH_k(\Omega),$ and a novel
  description of the additional data needed to specify the projection into this
  space. Later in the paper we give a new proof that $\dim\cH_k(\Omega)=2g,$
  for $k$ in the closed upper half plane, $\Zup.$ In this regard, the case
  $k=0$ is classical. As noted above, this result was proved in~\cite{Kress1}.
\end{remark}

\begin{proof} 
  Suppose that $i_{\bn}\bxi$ and $\bEta\restrictedto_{\Gamma}$ both vanish.
  Let $\balpha=\bxi\restrictedto_{\Gamma}$ and
  $\bBeta=\star_3\bEta\restrictedto_{\Gamma}.$ The usual properties of the
  exterior derivative, the hypothesis $\bEta\restrictedto_{\Gamma}=0,$ and the
  equation $d\bxi=ik\bEta$ imply that
\begin{equation}
  \label{eq:bndeq1}
  d_{\Gamma}\balpha=0.
\end{equation}
We can rewrite $d^*\bEta=-ik\bxi$ as
$d\star_3\bEta=-ik\star_3\bxi.$ The hypothesis, $i_{\bn}\bxi=0$
now implies that
\begin{equation}
  \label{eq:bndeq2}
  d_{\Gamma}\bBeta=0.
\end{equation}
A calculation using a co-frame adapted to $\Gamma$ shows that
\begin{equation}
  \label{eq:srfrel}
  \star_2\bBeta=-[i_{\bn}\bEta]\restrictedto_{\Gamma},
\end{equation}
see~\eqref{tangHfld}.  The equation $d\bxi=ik\bEta$ implies that
\begin{equation} 
\label{eq:srfrel2}
 [ i_{\bn}d\bxi]\restrictedto_{\Gamma}=-ik\star_2\bBeta.
\end{equation}
We can therefore express  the right hand side of~\eqref{eq:bndtrm1} as
\begin{equation}
  \label{eq:brdtrm2}
  2|k|^2
  \Re\left(\int\limits_{\Gamma}\balpha\wedge \bar{\bBeta}\right).
\end{equation}
If $\Gamma$ is simply connected then the equation $d_{\Gamma}\balpha=0$ implies
that $\balpha=d_{\Gamma}u.$ As $d_{\Gamma}\bar{\bBeta}=0$ as well, a simple
application of Stokes formula shows that
\begin{equation}
  \label{eq:brdtrm3}
  \int\limits_{\Gamma}d_{\Gamma}u\wedge \bar{\bBeta}=0.
\end{equation}
This completes the proof of the theorem when $\Gamma$ is simply connected. 

For the general case, let
$H^1_{\dR}(\Gamma)$ denote the De Rham cohomology group with
$\dim H^1_{\dR}(\Gamma)=2g\neq 0$. 
We show that the space of
solutions with vanishing normal components, for which the integral
in~\eqref{eq:brdtrm2} is non-vanishing, depends on at most $2g$ parameters.
Using the wedge product, we define a pairing, $W,$ on closed 1-forms:
\begin{equation}
  W(\eta,\omega)=\int\limits_{\Gamma}\eta\wedge\omega.
\end{equation}
If $d_{\Gamma}\eta=0$ and $\omega=d_{\Gamma}u,$ then, as noted above, Stokes
theorem implies that
\begin{equation}
  W(\eta,\omega)=0.
\end{equation}
Hence $W$ defines a skew-symmetric form on $H^1_{\dR}(\Gamma),$ which is well
known to be non-degenerate. As the $\dim H^1_{\dR}(\Gamma)=2g,$ this
observation completes the proof of the fact that $\dim\cH_k(\Omega)\leq 2g.$ If
the image of either $\balpha$ or $\bBeta$ in $H^1_{\dR}(\Gamma)$ vanishes, then
$W(\balpha,\bar{\bBeta})=0,$ which implies, as above, that the solution in
$\Omega$ is zero.
\end{proof}
\noindent
From this proof we see that the image of either $\balpha$ or $\bBeta$ in
$H^1(\Gamma)$ provides data specifying a $k$-Neumann field. This should be
contrasted with the data used by Kress, given in
equation~\eqref{kress_thm1a}. As $\dim \cH_k(\Omega)=\dim H^1(\Gamma),$ the
maps from $\cH_k(\Omega)$ to $H^1(\Gamma)$ defined by $(\bxi,\bEta)\mapsto
\balpha$ and $(\bxi,\bEta)\mapsto \bBeta$ are both isomorphisms, when $k\neq 0.$

We complete this section by proving Lemma~\ref{prop3}, which shows that the
normal components of $(\bxi,\bEta),$ a solution to \THME[$k$] with $k\neq 0,$
have vanishing mean value over every component of the boundary. While
superficially this might appear analogous to the fact that the normal
derivative of a harmonic function in a bounded domain has mean value over the
boundary, it is actually an elementary consequence of the equations themselves
and Stokes' theorem on a closed surface.

If $k\neq 0,$ then the Maxwell equations,~\eqref{eq:ME1}, imply that
\begin{equation}
  \label{eq:nrmmn0}
  \star_3\bxi\restrictedto_{\Gamma}=
\frac{-1}{ik}d_{\Gamma}[\star_3\bEta\restrictedto_{\Gamma}]\quad
\bEta\restrictedto_{\Gamma}=\frac{1}{ik}d_{\Gamma}[\bxi\restrictedto_{\Gamma}].
\end{equation}
 As 
 \begin{equation}
 \star_3\bxi\restrictedto_{\Gamma}=i_{\bn}\bxi dA
 \end{equation}
 the relations in~\eqref{eq:nrmmn0} imply that these forms are exact and
 therefore Stokes' theorem implies that a solution of \THME[$k$], with
 $k\neq 0$ satisfies:
\begin{equation}
  \int\limits_{\Gamma_m}i_{\bn}\bxi dA=\int\limits_{\Gamma_m}\bEta =0
\end{equation}
for $m=1,\dots,M.$ Note that this is true whether the limit is taken from
$\Omega$ or $D.$ This completes the proof of 
Lemma~\ref{prop3}, restated as
\begin{proposition}\label{prop3frm}
  Let $(\bxi,\bEta)$ be a solution to the \THME[$k$] for a $k\neq 0,$ in a
  region $G\subset\bbR^3,$ with a smooth bounded boundary. The normal
  components $(i_{\bn}\bxi,i_n\star_3\bEta)$ have mean zero over every
  component of $bG.$
\end{proposition}

\subsection{Potentials and Boundary Integral Equations} \label{potbdrforms}

We now re-express~\eqref{potrepre1} in terms of exterior forms. Assuming, as
before, that the time dependence is $e^{-i\omega t},$ the permittivity is
$\epsilon,$ the permeabilty is $\mu,$ and $k=\omega\sqrt{\epsilon\mu};$ we set
\begin{equation}
\bxi=(ik\balpha-d\phi-d^*\balpha_m)\quad
\bEta=(d\balpha+ik\balpha_m+d^*\Phi_m),
\label{eqn29}
\end{equation}
where $\phi$ is a scalar function, $\balpha$ a one form, $\balpha_m$ a two form,
and $\Phi_m=\phi_m dV,$ a three form.  This representation is quite similar to what
one obtains using the fundamental solution for the Dirac operator $d+d^*,$
see~\cite{Axelsson}.

In order for $(\bxi, \bEta)$ to satisfy the \THME[$k$], the potentials must
satisfy:
\begin{equation}
d^*\balpha=-ik \phi\quad d\balpha_m=ik\Phi_m.
\label{pteq3}
\end{equation}

As before $g_k(x-y)$ denotes the outgoing fundamental solution for the scalar Helmholtz
equation, with frequency $k.$
As discussed in the introduction, 
\emph{all} of the potentials are expressed in
terms of a pair of 1-forms $\bj, \bm$ defined on $\Gamma,$ though in the end,
we do \emph{not} use $\bj$ and $\bm$ as the ``fundamental'' parameters. When we
express these 1-forms in terms of the ambient basis from $\bbR^3,$ e.g.,
\begin{equation}
\bj=j_1(x)dx_1+j_2(x)dx_2+j_3(x)dx_3,
\end{equation}
we normalize with the requirement 
\begin{equation}
i_{\bn}\bj=\bj(\bn)\equiv 0.
\label{1frmnrm}
\end{equation}
These 1-forms are the metric duals of the vector fields, tangent to $\Gamma,$`
previously denoted by $\bj$ and $\bm.$

The ``vector'' potentials are given in terms of surface integrals by setting
\begin{equation}
\begin{split}
\balpha&=\int\limits_{\Gamma}g_k(x-y)[j_1(y)dx_1+j_2(y)dx_2+j_3(y)dx_3]dA(y)\\
\balpha_m&=\star_3\left[\int\limits_{\Gamma}g_k(x-y)
[m_1(y)dx_1+m_2(y)dx_2+m_3(y)dx_3]dA(y)\right]. 
\end{split}
\label{srfint1}
\end{equation}
Using the equations in~\eqref{pteq3} we obtain the form of the potentials
defining $\phi$ and $\Phi_m=\phi_mdV,$ letting
\begin{equation}
  \begin{split}
    \phi(x)&=\int\limits_{\Gamma}g_k(x-y)r(y)dA(y)\\
\phi_m(x)&=\int\limits_{\Gamma}g_k(x-y)q(y)dA(y),
  \end{split}
\end{equation}
where we let
\begin{equation}
\frac{1}{ik}d_{\Gamma}\star_2\bj=r dA\quad
\frac{1}{ik}d_{\Gamma}\star_2\bm=qdA.
\label{eqn53}
\end{equation}
The scalar functions, $(r,q)$ are, as before, the Debye sources.  From this
definition, and Stokes' theorem we see that the mean values of $r$ and $q$
vanish on every connected component, $\Gamma_j,$ of $\Gamma,$
\begin{equation}
\int\limits_{\Gamma_j}rdA=\int\limits_{\Gamma_j}qdA=0
\label{eqn96a}
\end{equation}
This proves Lemma \ref{meanzerolemma}. It is
{\bf necessary} for the conditions in~\eqref{eqn53} to hold, and thus, for
$(\bxi,\bEta)$ to satisfy the Maxwell equations. 

As before, we let $\cM_{\Gamma,0}$ denote pairs of functions $(r,q)$ defined on $\Gamma$
with mean zero on every component of $\Gamma.$
If we assume that $\bm=\star_2\bj,$ (as we usually do), then, taking account of the fact that the
(negative) Laplace operator on 1-forms is given by
$-\Delta_1=d_{\Gamma}^*d_{\Gamma}+d_{\Gamma}d_{\Gamma}^*,$ we obtain the
relation:
\begin{equation}
\Delta_1\bj=ik[d_{\Gamma}r-\star_2d_{\Gamma}q].
\label{srflap}
\end{equation}
If the components of $\Gamma$ are all of genus zero, then
equation~\eqref{srflap} always has a unique solution. If $\Gamma$ has positive
genus components, then one needs to deal with the null space of $\Delta_1.$

If $H^1_{\dR}(\Gamma)\neq 0,$ then the nullspace of $\Delta_1,$
$\cH^1(\Gamma),$ which agrees with the space of solutions to
 \begin{equation}
    d_{\Gamma}\balpha=0\quad d^*_{\Gamma}\balpha=0
  \end{equation}
  is isomorphic to $H^1_{\dR}(\Gamma).$ These are the harmonic 1-forms. The right
  hand side in~\eqref{srflap} is orthogonal to $\cH^1(\Gamma),$ and hence lies
  in the range of $\Delta_1.$ Let $R_1$ denote the partial inverse of the
  Laplacian on 1-forms, with range orthogonal to $\cH^1(\Gamma),$  and set
\begin{equation}
\bj_R(r,q,k)=ik R_1[d_{\Gamma}r-\star_2d_{\Gamma}q],
\label{eqn51}
\end{equation}
Because the ranges of $d_{\Gamma}$ and $d^*_{\Gamma}$ are orthogonal to the
null space of $\Delta_1,$ this equation is solvable whether or not $r$ and
$q$ satisfy the mean value condition. Note that the solution
to~\eqref{eqn51} tends to zero as $k\to 0.$ 

Using the relations $\Delta_1d_{\Gamma}=d_{\Gamma}\Delta_0,$ and
$\Delta_2=\star_2\Delta_0\star_2,$ we can re-express $\bj_{R}$ in the form:
\begin{equation}
\bj_R(r,q,k)=ik[d_{\Gamma}R_0r-\star_2d_{\Gamma}R_0q].
\label{eqn51.1}
\end{equation}
Here $R_0$ is the partial inverse of $\Delta_0,$ which annihilates functions
constant on each component of $\Gamma$ and has range equal to the set of
functions with mean zero on each component of $\Gamma.$
Equation~\eqref{eqn51.1} shows that this approach to representing solutions of
Maxwell's equations in terms of the pair $(r,q),$ only requires an inverse for
the \emph{scalar} Laplacian on $\Gamma.$

For any
  $(r,q),$ the solution space to~\eqref{srflap} is isomorphic to
  $\cH^1(\Gamma).$ Adding a harmonic 1-form to $\bj_R$ does not change $r$ and
  $q,$ though it changes the fields $\bxi$ and $\bEta,$ and plays a central
  role in the discussion of $\cH_k(\Omega).$ If $g\neq 0,$ then the space of
  outgoing solutions to \THME[$k$] is parameterized by
  $\cM_{\Gamma,0}\oplus\cH^1(\Gamma).$ So given data $(r,q,\bj_H)$ we often speak
  of the solution to the \THME[$k$] ``defined'' by this data. If $\bj_H$ is
  missing, then it should be understood to be zero, i.e. the solution ``defined
  by $(r,q)$,'' is the solution defined by $(r,q,0)$ in the sense above.

\subsection{Boundary Equations and Jump Relations}
As with the vector field representation, we can take the limits of $\bxi$ and
$\bEta$ in~\eqref{eqn29} as the point of evaluation approaches $\Gamma,$ and
obtain boundary integral equations for the normal and tangential components of
these forms. Indeed there is no necessity to rewrite these equations, we simply
use $\cT_{\bxi}^{\pm}(k)(r,q,\bj,\bm),$ $\cT_{\bEta}^{\pm}(k)(r,q,\bj,\bm)$ to
denote the limiting tangential components and the limiting normal components
are $\cN_{\bxi}^{\pm}(k)(r,q,\bj,\bm),$ $\cN_{\bEta}^{\pm}(k)(r,q,\bj,\bm).$
Keep in mind that, in the vector field representation, the tangential
components are represented as the limits of $\bn\times\bE$ and $\bn\times\bH,$
which correspond to $\star_2\bxi_t$ and $\star_2([\star_3\bEta]_t),$
respectively.  For consistency, we use $\cT_{\bxi}^{\pm},\cT_{\bEta}^{\pm}$ to
denote the boundary values of these quantities. As before, $+$ indicates the
limit taken from $\Omega$ and $-$ the limit taken from $D.$ If
$\bj=\bj_R(r,q,k)$ and $\bm=\star_2\bj,$ then we omit them from the argument
list, e.g., we use the abbreviated notation $\cT_{\bxi}^{\pm}(k)(r,q)$ to
denote $\cT_{\bxi}^{\pm}(k)(r,q,\bj_R(r,q,k),\star_2\bj_R(r,q,k)),$ etc.

Below we use the jump relations to prove a uniqueness theorem. So it is useful
to reformulate them in the form language. The magnetic field is represented by
the 2-form, $\bEta=h_1dx_2\wedge dx_3+h_2dx_3\wedge dx_1+h_3dx_1\wedge dx_2,$
so that $\star_3\bEta=h_1dx_1+h_2dx_2+h_3dx_3.$ The most direct way to define
the normal and tangential components of $\bEta$ along $\Gamma$ is as
$i_{\bn}\star_3\bEta_{\pm} $ and $(\star_3\bEta_{\pm})_t.$ In this formulation
the jump relations then take the form
\begin{equation}
  \begin{split}\label{jmprels2}
    i_{\bn}(\bxi_+-\bxi_-)=r &\quad (\bxi_+-\bxi_-)_t=\star_2\bm\\
i_{\bn}(\star_3\bEta_+-\star_3\bEta_-)=q &\quad (\star_3\bEta_+-\star_3\bEta_-)_t=-\star_2\bj
  \end{split}
\end{equation}
It is also useful to calculate the relationship between $i_{\bn}\bEta$ and
$(\star_3\bEta)_t.$ In an adapted frame $(\omega_1,\omega_2,\nu),$ We have
\begin{equation}
  \bEta=a\omega_1\wedge\omega_2+b\nu\wedge\omega_1+c\nu\wedge\omega_2
\end{equation}
and therefore
\begin{equation}\label{tangHfld}
  \begin{cases}
    i_{\bn} \bEta&=b\omega_1+c\omega_2\\
(\star_3\bEta)_t&=b\omega_2-c\omega_1,
  \end{cases}\text{ which implies that}  \star_2 [i_{\bn} \bEta]=(\star_3\bEta)_t.
\end{equation}

\section{Uniqueness for the Tangential Equations}\label{MEfrms2}
Suppose that there is a
$k\in\Zup\setminus\{0\},$ and non-trivial data $(r,q)\in\cM_{\Gamma,0}$ and $\bj,$
satisfying~\eqref{eqn53}, with $\bm=\star_2\bj,$  so that
\begin{equation}
  \label{eq:nlspc1}
  \cT^{+}_{\bxi}(k)(r, q, \bj, \bm)=0.
\end{equation}
Let $(\bxi_{\pm},\bEta_{\pm})$ be the solutions to the Maxwell equations
defined by this data in the complement of $\Gamma.$ By their definition it is
clear that the tangential components of $\bxi_+$ vanish along $\Gamma.$ As
$\Im(k)\geq 0,$ the solution in $\Omega$ is outgoing and
Theorem~\ref{thm4class.1} shows that $(\bxi_+,\bEta_+)\equiv (0,0).$ The jump
relations,~\eqref{jmprels2}, and the fact that $\bm=\star_2\bj,$ allow us to
determine the tangential boundary data for $(\bxi_-,\bEta_-):$
\begin{equation}
  \label{eq:tngbdr1}
  \bxi_-\restrictedto_{T\Gamma}=\bj\quad
i_{\bn}\bEta_-\restrictedto_{T\Gamma}=\bj.
\end{equation}
It is not difficult to see that the boundary condition on the Maxwell system in
$D,$ implied by these relations,
\begin{equation}
\bxi_-\restrictedto_{T\Gamma}=i_{\bn}\bEta_-\restrictedto_{T\Gamma},
\label{eqn:6.20.2}
\end{equation}
is not formally self adjoint!

Observe that $d^*d\bxi_-=k^2\bxi_-$ and
$i_{\bn}d\bxi_-\restrictedto_{\Gamma}=ik
i_{\bn}\bEta_-\restrictedto_{\Gamma}.$ Using a standard integration by parts
formula, we obtain:
\begin{equation}
  \label{eq:intprts4}
\begin{split}
  \int\limits_{D}(d\bxi_-,d\bxi_-)dV&= \int\limits_{D}(d^*d\bxi_-,\bxi_-)dV+
 \int\limits_{bD}(i_{\bn}d\bxi_-,\bxi_-)dA\\
&= k^2\int\limits_{D}(\bxi_-,\bxi_-)dV+ik
 \int\limits_{bD}(i_{\bn}\bEta_-\restrictedto_{\Gamma},\bxi_-\restrictedto_{\Gamma})dA.
\end{split}
\end{equation}
Combining this with~\eqref{eq:tngbdr1} gives:
\begin{equation}
  \label{eq:intprts5}
  -ik\int\limits_{bD}(\bj,\bj)dA=k^2\int\limits_{D}(\bxi_-,\bxi_-)dV-
  \int\limits_{D}(d\bxi_-,d\bxi_-)dV. 
\end{equation}
We can rewrite this relation as
\begin{equation}
  \label{eq:krel}
  -aik=bk^2-c,
\end{equation}
where $a,b,c$ are non-negative real numbers. If $b$ or $c$ vanishes, then it is
clear that $\bxi_-\equiv 0.$ If $a=0,$ then $\bj\equiv 0.$ For a countable set of
real numbers $\{k_j\},$ there exist non-trivial solutions to the equations
\begin{equation}
  d^*d\bxi_-=k_j^2\bxi_-\quad d^*\bxi_-=0\quad \bxi_-\restrictedto_{TbD}=0.
\end{equation}
In the present circumstance, however, $(r,q)$ are  generalized Debye sources
and therefore
\begin{equation}
  rdA=\frac{1}{ik}d_{\Gamma}\star_2\bj\text{ and }qdA
=\frac{1}{ik}d_{\Gamma}\bj.
\end{equation}
If $a=0,$ then all the boundary potentials vanish, and therefore $\bxi_-\equiv
0$ as well.

Using the quadratic formula we see that
\begin{equation}
k_{\pm}=\frac{-ia\pm\sqrt{4bc-a^2}}{2b}.
\label{eq:127n}
\end{equation}
As $a,b,c$ are all positive,~\eqref{eq:127n} shows  that $\Im k_{\pm}<0.$
This argument applies, \emph{mutatis mutandis} to $\cT^{+}_{\bEta}(k).$ 
Formula~\eqref{eq:127n} and the discussion above complete the proof of the following theorem:
\begin{theorem}\label{thm66} Assuming that $\bm=\star_2\bj,$ and $(r,q)$
  satisfy~\eqref{eqn53}, then, for $\Im k\geq 0,$ $k\neq 0,$ the nullspaces of
  both $\cT^{+}_{\bxi}(k)$ are $\cT^{+}_{\bEta}(k)$ are trivial.
\end{theorem}

In the case that $\Gamma$ is simply connected this implies that $E_+,$ the
exceptional set for $\cN^{+}(k),$ is disjoint from the closed upper half plane. 
\begin{corollary}\label{cor4} If every component of $\Gamma$ is simply connected, then, for
  $k$ with $\Im k\geq 0,$ the Fredholm operator $\cN^+(k)$ is an isomorphism
  from $\cM_{\Gamma,0}$ to itself. For such $k,$ the rows of $\cT^{+}(k)$ are
  also surjective and hence isomorphisms.
\end{corollary}
\begin{proof} If $\Gamma$ is simply connected, then any 1-form $\bj$
  on $\Gamma$ has a unique representation as $\bj=\bj_R(r,q,k).$ We can therefore
  regard $\cN^{+}(k)$ as a Fredholm system of second kind for the normal
  components of $(\bxi_+,\bEta_+)$ in terms of $(r,q).$ In this case,
  Theorem~\ref{thm4.1} implies that a solution $(\bxi_{+},\bEta_{+})$ of
  \THME[$k$], with vanishing normal components is identically zero in $\Omega.$
  Hence $(\bxi_-,\bEta_-)$ satisfy~\eqref{eq:tngbdr1}, and we can therefore
  apply the argument leading up to Theorem~\ref{thm66} to prove that $E_+$ is
  disjoint from the closed upper half plane. The Fredholm alternative then
  implies that $\cN^+(k)$ is also surjective. The surjectivity of the rows of
  $\cT^{+}(k)$ is now immediate.
\end{proof}
\begin{remark} The poles of the scattering operator for the Maxwell system,
  defined by a self adjoint boundary condition on $\Gamma$, lie in the lower
  half plane. Nonetheless, it appears that the eigenvalues for the non-self
  adjoint boundary value problem defined by~\eqref{eqn:6.20.2} are unrelated to
  these poles, but are simply interior resonances, familiar from more
  traditional representations of solutions to Maxwell's equations (see the
  Introduction and Remark~\ref{bad_poles}). The non-self adjointness of this
  BVP places the interior resonances in the lower, non-physical, half
  plane. This leads, in the simply connected case, to numerically effective
  algorithms for solving the \THME[$k$], which do not suffer from the
  instabilities caused by interior resonances in the physical half plane.
\end{remark}

In the non-simply connected case we have the following theorem assuring the
existence of $k$-Neumann fields.
\begin{theorem}\label{thm6.1} For $k\in\Zup,$ the space of $k$-Neumann fields has
  dimension exactly $2g,$ and the rows of $\cT^{+}(k)$ are surjective.
\end{theorem}
\begin{proof}
   For $k\in\Zup,$ the solutions to \THME[$k$], defined by
  $$\cC_{H}=\{(0,0,\bj_H,\star_2\bj_H):\:\bj_H\in\cH^1(\Gamma)\}$$ 
  have a trivial intersection with those defined by data in
$$\cC_{R}=\{(r,q,\bj_R(r,q,k),\star_2\bj_R(r,q,k)):\: (r,q)\in\cM_{\Gamma,0}\}.$$
The solutions defined by data in $\cC_H$ may not themselves be $k$-Neumann
fields. To find solutions in $\cH_k(\Omega),$ we first use an element of $\cC_{H}$
to construct a solution $(\bxi_{0+},\bEta_{0+}).$ If
$k\notin E_+,$ then we can solve
\begin{equation}
  \cN^+(k)\left(\begin{matrix} r\\ q\end{matrix}\right)=
\left(\begin{matrix} i_{\bn}\bxi_{0+}\\i_{\bn}\star_3\bEta_{0+} \end{matrix}\right),
\end{equation}
and denote the solution of the \THME[$k$] defined by this data in $\Omega$ by
$(\bxi_{1+},\bEta_{1+}).$ By Theorem~\ref{thm66}, the difference
\begin{equation}
 (\bxi_{N+},\bEta_{N+})=(\bxi_{0+},\bEta_{0+})-(\bxi_{1+},\bEta_{1+})
\end{equation}
is a non-trivial $k$-Neumann field. These solutions depend analytically on
$k\in\Zup\setminus E_+.$ As $(\bxi_{N+},\bEta_{N+})$ is a non-zero solution to
the \THME[$k$] with vanishing normal components, Theorem~\ref{thm4.1} shows that
the cohomology class of $\bxi_{N+t}$ must be non-trivial. Thus for $k\notin
E_+,$ $\dim\cH_k(\Omega)$ is at least $2g.$ On the other hand, the proof of
Theorem~\ref{thm4.1} gives the upper bound $\dim\cH_k(\Omega)\leq 2g.$ Proving
the theorem in this case.

  Now suppose that $k_j\in E_+\cap\Zup.$ This means that there is a
  non-trivial, finite dimensional space of data $V_{k_j}\subset \cM_{\Gamma,0},$
  which defines $\Ker\cN^+(k_j)(r,q).$ Let $(\bxi_{\pm},\bEta_{\pm})$ denote the
  solution of the \THME[$k$] defined by a non-zero pair $(r,q)\in V_{k_j}.$ The
  restriction of $\bxi_+$  to $T\Gamma$ defines a cohomology class in
  $H^1_{\dR}(\Gamma).$ If this class is trivial, then Theorem~\ref{thm4.1}
  implies that the pair $(\bxi_+,\bEta_+)$ are identically zero. In this case,
  Theorem~\ref{thm66} implies that $\Im k_j<0,$ contradicting the assumption
  that it lies in $\Zup.$ This establishes that each non-trivial pair in
  $V_{k_j}$ defines a non-trivial $k_j$-Neumann field, thus a subspace of
  $\cH_{k_j}(\Omega)$ of dimension $d=\dim V_{k_j}.$

  The Fredholm alternative implies that the equations for the normal
  components: $\cN^+(k_j)(r,q)=(f,g)$ are solvable for pairs
  $(f,g)\in\cM_{\Gamma,0}$ satisfying exactly $d$ linear conditions. This means
  that within $\cC_{H}$ there is a subspace of dimension at least $2g-d,$
  for which the normal components can be removed, as above. We therefore get
  another subspace, $U_{k_j}\subset \cH_{k_j}(\Omega),$ of dimension at least
  $2g-d.$ As $V_{k_j}$ has a trivial intersection with the data defining
  $U_{k_j},$ Theorem~\ref{thm66} implies that these two subspaces of
  $\cH_{k_j}(\Omega)$ have a trivial intersection. The lower bound on the
  dimension of $U_{k_j}$ and the upper bound on $\dim \cH_{k_j}(\Omega)$ imply
  that $\dim U_{k_j}=2g-d;$ this completes the proof that
  $\dim\cH_{k}(\Omega)=2g,$ for all $k\in\Zup.$

  For $k\notin E_+,$ Theorem~\ref{thm4.1} combined with the fact that
  $\dim\cH_{k}(\Omega)=2g$ shows that the rows of $\cT^+(k)$ are surjective. If
  $k\in E_+\cap\Zup,$ then the range of $\cN^+(k)$ has codimension exactly
  $d=\dim V_k.$ On the other hand, there is a $d$-dimensional space of data
  in $\cH^1(\Gamma)$ for which the normal components span a complement to that
  in $\Im\cN^+(k).$ Once again we can find an outgoing solution to the
  \THME[$k$] with any specified normal components. Combined with the fact
  $\dim\cH_k(\Omega)=2g,$ Theorem~\ref{thm4.1} again shows that the rows of
  $\cT^+(k)$ are surjective.
\end{proof}

In the course of this argument we established:
\begin{corollary}\label{cor4.1} For $k\in\Zup\setminus\{0\},$ the map from $\cH_k(\Omega)$ to
  $H^1_{\dR}(\Gamma)$ defined by 
$$(\bxi_{N+},\bEta_{N+})\mapsto [\bxi_{N+t}]_{\Gamma}$$ 
is an isomorphism.
\end{corollary}

\begin{remark} It was shown by Picard that for each $k$ with non-negative real
  part there are families of interior $k$-Neumann fields, that is non-trivial
  solutions to the \THME[$k$] in $D$ with vanishing normal components,
  see~\cite{Picard1,Picard2}. For most values of $k$ there is a
  $2g$-dimensional family. In~\cite{Kress1} Kress showed that there is a countable
  set of positive real numbers $\{k_j\},$ with $k_j\to\infty,$ for which there are
  non-trivial interior $k$-Neumann fields with vanishing circulations.
\end{remark}
This theorem really asks more questions than it answers:
\begin{enumerate}
\item If $k=0,$ then Hodge theory essentially implies the existence of the
  $0$-Neumann fields. For $k\neq 0,$ what is the
  reason for the existence of the $k$-Neumann fields?  A possible explanation
  might go along the following lines: In Appendix~\ref{A4}, we express the
  \THME[$k$] in the form:
  \begin{equation}\label{eq7.10.1}
    L_k(\bxi+\bEta)=(d+d^*-ik\Lambda)(\bxi+\bEta)=0.
  \end{equation}
  Suppose that $L_k,$ acting on divergence free, outgoing data, which satisfies
  $i_{\bn}\bxi=\bEta\restrictedto_{\Gamma}=0$ is in some sense a Fredholm
  family. The boundary conditions defining the formal adjoint, $[L_k]^*,$ are
  $\bxi_t=0,(i_{\bn}\bEta)_t=0.$ Theorem~\ref{thm4class.1} implies that the
  nullspace of $[L_k]^*$ is trivial for $k\in\Zup.$ The nullspace of $L_k$ is
  exactly $\cH_k(\Omega).$  \emph{If} these operators are
  a Fredholm family, then the constancy of the Fredholm index would imply that 
  \begin{equation}
    \Ind(L_k)=\dim\cH_k(\Omega)=2g.
  \end{equation}
It is not  obvious, however, on what space the range of $L_k$ is closed.
\item Does $E_+$ have a non-trivial intersection with $\Im k\geq 0?$ If so,
  what is the physical significance of these numbers? As noted above, Kress
  proved that there is a countable set of positive real numbers for which there
  exist interior $k$-Neumann fields, with vanishing normal components and
  circulations. Are these numbers in any way related to $E_+?$
\end{enumerate}

\subsection{The hybrid system using forms}\label{sec7.1.1}
The operators defining tangential component of $\bxi_{\pm}$ are given by
\begin{equation}
  \cT^{\pm}_{\bxi}(k)\left(\begin{matrix} r\\q\\\bj\\ \bm\end{matrix}\right)
=
\frac{\mp \bm}{2}+
\left(\begin{matrix} -K_1 & ik K_{2,t} & -K_4\end{matrix}\right)
\left(\begin{matrix} r\\ \bj \\
  \bm\end{matrix}\right).
\label{eqn888.1}
\end{equation}
If we restrict to $\bj=\bj_R(r,q,k),$ and $\bm=\star_2\bj,$ then, acting on
$(r,q)\in\cM_{\Gamma,0},$ the only term of non-negative order is $-K_1,$ which
can be expressed as
$K_1r=d_{\Gamma}G_kr.$ Here
\begin{equation}
  G_kr(\bx)=\int\limits_{\Gamma}g_k(\bx-\by)r(\by)dA(\by),
\end{equation}
is an operator of order $-1.$
 We use $\cT^{\pm}_{\bxi}(k)(r,q)$ to denote this operator
restricted to this subspace of data. 

The hybrid system of integral
operators \ref{hybridopQ} is
\begin{equation}
  \cQ^{\pm}(k)\left(\begin{matrix} r\\q\end{matrix}\right)=
\left(\begin{matrix}
   -G_0\star_2d_{\Gamma}\cT^{\pm}_{\bxi}(k)\\\cN_{\bEta}^{\pm}(k)
\end{matrix}\right) \left(\begin{matrix} r\\q\end{matrix}\right).
\end{equation}
The range of $\star_2d_{\Gamma}\cT^{\pm}_{\bxi}(k)$ is contained in the space of
functions on $\Gamma$ with mean zero on every component. 
Proposition \ref{propFplus} holds in the form version as well.

Suppose now that $k\in F_+,$ and $\cQ^{+}(k)(r,q)=0,$ with
$(r,q)\in\cM_{\Gamma,0}\setminus \{0\},$ and let $(\bxi_+,\bEta_+)$ be the
solution to the \THME[$k$] defined by this data. The fact that
$\cQ^{+}(k)(r,q)=0,$ implies that
\begin{equation}
  d_{\Gamma}^*\bxi_{+t}=0\text{ and }\bEta_{+}\restrictedto_{T\Gamma}=0;
\end{equation}
the second condition implies that $d_{\Gamma}\bxi_{+t}=0,$ as well. If the
cohomology class $[\bxi_{+t}]_{\Gamma}=0,$ then $(\bxi_{+})_t$ vanishes and
Theorem~\ref{thm4class.1} implies that $(\bxi_+,\bEta_+)$ is identically
zero. When $\Gamma$ is simply connected, $H^1_{\dR}(\Gamma)=0,$ and this
proves Theorem \ref{skie_simpcconn}, written in terms of forms.

\begin{theorem}\label{skie_simpcconn.1}
If $\Gamma$ is simply connected, then
  $F_+$ is disjoint from the closed upper half plane.
Thus, the integral equation 
\begin{equation}
  \cQ^{+}(k)
  \left(\begin{matrix} r\\q\end{matrix}\right) =
  \left(\begin{matrix} f\\h\end{matrix}\right)
\end{equation}
provides a unique solution to the scattering problem from a perfect
conductor for any $k$ in the closed upper half plane. Here,
\begin{equation}\label{set_data.1}
f=G_0(d^*_{\Gamma}\bxi^{\In}_{t}), \qquad
ik hdA=d_{\Gamma}\bxi^{\In}_{t}=ik\bEta^{\In}\restrictedto_{T\Gamma},
\end{equation}
where $\bxi^{\In}_{t}$ is the tangential component of an incoming
electric field, and $\bEta^{\In}\restrictedto_{T\Gamma},$ the normal component
of the incoming magnetic field.
\end{theorem}

When applying our method in the non-simply connected case, the following result
is useful.
\begin{proposition}\label{prop6} Suppose that $k\notin E_+\cup F_+,$ and let
  $\psi\in\cH^1(\Gamma).$ The unique outgoing solution to the \THME[$k$] with
  $\bxi_{+t}=\psi$ is defined by data $(r,q,\bj_H)$ with $\bj_H\neq 0.$
\end{proposition}
\begin{proof}  As $k\notin E_+,$ the proof of Theorem~\ref{thm6.1} 
  produces a solution,$(\bxi_{N+},\bEta_{N+}),$ to the \THME[$k$] with
  vanishing normal components and $[\bxi_{N+t}]_{\Gamma}=[\psi]_{\Gamma}.$ The
  potentials corresponding to this solution take the form $(r_0,q_0,\bj_H)$
  with $\bj_H\neq 0.$ The condition $[\bxi_{N+t}]_{\Gamma}=[\psi]_{\Gamma}$
  shows that there is a function $f,$ of mean zero on every component of
  $\Gamma,$ satisfying
 \begin{equation}
   \bxi_{N+t}=\psi+d_{\Gamma}f.
 \end{equation}
Since $k\notin F_+$ we can therefore solve the equation
\begin{equation}
  \cQ^+(k)(r_1,q_1)=(0,G_0d_{\Gamma}^*d_{\Gamma}f).
\end{equation}
With $(\bxi_+,\bEta_+)$ the solution to the \THME[$k$] defined in $\Omega$ by
this data, we see that 
\begin{equation}
  (\bxi_{H+},\bEta_{H+})=(\bxi_{N+},\bEta_{N+})-(\bxi_+,\bEta_+),
\end{equation}
satisfies 
\begin{equation}\label{eqn6.19.1}
  d_{\Gamma}\bxi_{H+t}=d_{\Gamma}^*\bxi_{H+t}=0,
\end{equation}
and therefore $\bxi_{H+t}\in\cH^1(\Gamma).$ This solution corresponds to the
sources $(r_0-r_1,q_0-q_1,\bj_H),$ with $\bj_H\neq 0.$ Theorem~\ref{thm66}
then implies that $\bxi_{H+}\neq 0.$ While it is not clear that
$[\bxi_{H+t}]_{\Gamma}=[\psi]_{\Gamma},$ it follows from~\eqref{eqn6.19.1} and
Theorem~\ref{thm4class} that $[\bxi_{H+t}]_{\Gamma}\neq 0.$ Thus $\bj_H\mapsto
\bxi_{H+t}$ is an injective linear mapping from $\cH^1(\Gamma)$ to itself, and
therefore an isomorphism.
\end{proof}

\begin{remark} Let $\{\psi^1,\dots,\psi^{2g}\}$ be a basis for $\cH^1(\Gamma).$
  The proof of this proposition, along with that of Theorem~\ref{thm6.1} show
  that, for $k\notin E_+\cup F_+,$ we can effectively construct solutions
  $\{(\bxi_{+H}^1,\bEta_{+H}^1)\dots,(\bxi_{+H}^{2g},\bEta_{+H}^{2g})\},$ in
  $\Omega$ to the \THME[$k$], which satisfy
  \begin{equation}
    \bxi_{+Ht}^l=\psi^{l}\text{ for }l=1,\dots,2g.
  \end{equation}

Using Proposition~\ref{prop6}, the system of equations, $\cQ^{+}(k),$ can 
again be used to solve the perfect conductor problem, at least for $k\notin
E_+\cup F_+.$ Let $\bxi^{\In}_{t}$ be the tangential component of an incoming
solution and set $(f,h)$ as in (\ref{set_data.1}).
For any $k\notin F_+,$ with non-negative
imaginary part, there is a unique solution to
\begin{equation}
  \cQ^{+}(k)(r,q)=(f,h).
\end{equation}
We let $(\bxi_+,\bEta_+)$ be the solution to the \THME[$k$], defined in
$\Omega$ by this data. Let $(\tbxi_+,\tbEta_+)$ be the unique outgoing solution
with
\begin{equation}
  \tbxi_{+t}=\bxi^{\In}_{t}.
\end{equation}
The tangential component of the difference $\psi=\tbxi_{+t}-\bxi_{+t}$ belongs to
$\cH^1(\Gamma).$ In the simply connected case this is zero, and therefore in
this case we are done. In general, we
can use Proposition~\ref{prop6} to find the unique solution
$(\hat\bxi_+,\hat\bEta_+),$ to the \THME[$k$] with $\hat\bxi_{+t}=\psi.$ By
Theorem~\ref{thm4class}, the sum satisfies
\begin{equation}
(\tbxi_+,\tbEta_+)=(\bxi_+,\bEta_+)+(\hat\bxi_+,\hat\bEta_+),
\end{equation}
and therefore solves the original boundary value problem.
\end{remark}

\begin{remark} In our modification of $\cT^{+}_{\bxi}(k),$ we use
  $-G_0\star_2d_{\Gamma}$ as a ``preconditioner,'' and to obtain a scalar
  equation. Other choices are possible, for example $-G_{l}(\star_2d_{\Gamma}+\cE)$
  where $l$ is another complex number and $\cE$ is an order zero operator
  mapping 1-forms to functions. For numerical applications it may be important
  to find a good choice here.
\end{remark}

  A question of considerable interest is to characterize the sets $E_+$
  and $F_+.$ One might hope that $E_+$ is disjoint from the closed upper half
  plane. The set $F_+$ depends, to some extent on the choice of preconditioner.
  Indeed we can modify the preconditioner so that $F_+=E_+.$ We let
  $W$ be the $L^2$-closure of $\{d_{\Gamma}u+d_{\Gamma}^*(vdA):\:
  u,v\in\CI(\Gamma)\},$ then, 
  Theorem~\ref{thm4} shows that for $k\notin E_+,$ $W$ is a complement to the
  tangential $\bxi$-boundary values of the $k$-Neumann fields:
\begin{equation}
\cH_k(\Omega)_t=\{\bxi_{+t}:\:
  (\bxi_+,\bEta_+)\in\cH_k(\Omega)\}.
\end{equation}
 To define the ``optimal'' preconditioner, we use the following lemma
\begin{lemma} There is an analytic family of projection operators $\{P_k:\:
  k\in\Zup\setminus E_+\cup\{0\}\}$ satisfying
  \begin{itemize}
  \item $P_k\restrictedto_{\cH_k(\Omega)_t}=\Id$
\item $P_k\restrictedto_{W}=0$
  \end{itemize}
\end{lemma}

\begin{proof} The proof of Theorem~\ref{thm6.1} gives an algorithm to construct
  a basis,
$$\{(\bxi_{N+}^l(k),\bEta_{N+}^l(k)):\: l=1,\dots,2g\}$$ 
of $\cH_k(\Omega),$ which depends analytically on $k\in\bbC\setminus E_+\cup\{0\}.$
Let $Q_0$ denote the orthogonal projection onto $\cH^1(\Gamma).$ Let
$\{\psi^m:\: m=1,\dots,2g\}$ be a fixed orthonormal basis for $\cH^1(\Gamma).$
The matrix of the restriction $Q_0\restrictedto_{\cH_k(\Omega)_t},$ with
respect to these bases, is given by
  \begin{equation}
    A^{lm}(k)=\langle \bxi_{N+t}^l(k),\psi^m\rangle_{L^2(\Gamma)}.
  \end{equation}
This matrix is analytic and invertible in $\Zup\setminus E_+\cup\{0\}.$ The inverse
transformation $R_k:\cH^1(\Gamma)\to \cH_k(\Omega)_t$ is therefore also
analytic. We define $P_k$ to be
  \begin{equation}
    P_k(\balpha)=R_kQ_0(\balpha)
  \end{equation}
  As $Q_0$ annihilates $W,$ it is immediate that $P_k$ satisfies the conditions
  above.
\end{proof}

We can modify our hybrid system by letting
\begin{equation}
  \cQ^{\pm}_1(k)\left(\begin{matrix} r\\q\end{matrix}\right)=
\left(\begin{matrix}
  - G_0\star_2d_{\Gamma}[\Id+\star_2P_k\star_2]\cT^{\pm}_{\bxi}(k)\\\cN_{\bEta}^{\pm}(k)
\end{matrix}\right) \left(\begin{matrix} r\\q\end{matrix}\right).
\end{equation}
As $d_{\Gamma}^*P_k$ is a bounded, finite rank operator, the family
$\cQ^{\pm}_1(k)$ is again Fredholm of second kind. Suppose that
$\cQ^+_1(k)(r,q)=0$ for $(r,q)\neq (0,0),$ and let $(\bxi_+,\bEta_+)$ be the
solution of the \THME[$k$] defined by this data. We see that
\begin{equation}
  \label{eq:6.20.1}
  d_{\Gamma}\bxi_{+t}=0\text{ and }[d_{\Gamma}^*-d_{\Gamma}^*P_k]\bxi_{+t}=0.
\end{equation}
If $k\notin E_+\cup\{0\},$ then the boundary data of solutions in $\cH_k(\Omega)$
solve this system of equations. If there were another solution, then in fact we
could find a 1-form $\balpha$ defined on $\Gamma,$ which solves this equation,
and $P_k\balpha=0.$ This means that $\balpha=d_{\Gamma}u+d_{\Gamma}^*vdA,$ and
\begin{equation}
  d_{\Gamma}^*d_{\Gamma}u=d_{\Gamma}d_{\Gamma}^*vdA=0,
\end{equation}
which easily implies that $\balpha=0.$ Thus for $k\in\Zup\setminus E_+\cup\{0\},$
$\cH_k(\Omega)_t$ is the complete set of solutions to the system of equations
in~\eqref{eq:6.20.1} and therefore the nullspace of $\cQ^+_1(k)$ is trivial for
$k\in\Zup\setminus E_+\cup\{0\}.$ Hence for this choice of preconditioner
$F_+\subset E_+.$

Additional care is required near $k=0,$ as the rank of the map
$\bxi_+\mapsto[\bxi_{+t}]_{\Gamma}$ drops at $k=0$ from $2g$ to $g.$

\subsection{Low Frequency Behavior in the Non-simply Connected Case}\label{lowfrq2}

If $\Gamma$ is not simply connected, then the space of solutions defined by
data in $\cM_{\Gamma,0}$ converges, as in the simply connected case, to the
orthogonal complement of the span of the harmonic Dirichlet and Neumann 
fields. In this case we also need to consider what happens to the solutions
defined by data from $\cH^1(\Gamma).$ Using this data we also obtain the
harmonic Neumann fields. We consider the two types of data separately,
beginning with that from $\cM_{\Gamma,0}.$

Theorem~\ref{thm66} demonstrates that, for any
$k\notin E_+$ with $\Im k\geq 0,$ we can solve the boundary value problem:
\begin{equation}
  \label{eq:bvpn}
\begin{split}
  d\bxi_+=ik\bEta_+&\quad d^*\bEta_+=-ik\bxi_+\\
i_{\bn}\bxi_+=f&\quad i_{\bn}(\star_3\bEta_+)=h,
\end{split}
\end{equation}
for arbitrary $(f,h)$ in $\cM_{\Gamma,0}.$ Indeed as $0\notin E_+,$ and the
integral equations on $\Gamma$ are of the second kind and analytic in $k,$ it
follows that we can actually solve~\eqref{eq:bvpn} for $k$ in an open
neighborhood, $V,$ of $0.$ We now discuss what happens
to our solutions as $k$ tends to $0,$ within a relatively compact subset of
$V.$ In particular, we would like to characterize exactly which harmonic fields
arise as limits of fields of the form given in~\eqref{eqn29}, where $(r,q)$ are
obtained by solving~\eqref{eqn74}, $\bj=\bj_R(r,q,k)$ and
$\bm=\star_2\bj.$

To avoid confusion, we let $(\bxi_k,\bEta_k)$ denote the unique solution
to~\eqref{eq:bvpn}, for a fixed $(f,h)\in\cM_{\Gamma,0}.$ As $\bj_R(r,q,k)$ is
$O(k),$ it follows easily that, as $k$ tends to 0, $(\bxi_k,\bEta_k)$ converges
to
\begin{equation}
  \label{eq:k0lmt}
  \bxi_0=d\phi\quad \bEta_0=d^*\Phi_m.
\end{equation}
\begin{theorem}\label{thm77}
  The set of limits $(\bxi_0,\bEta_0)$ for $(f,h)\in\cM_{\Gamma,0}$ is the
  orthogonal complement to the span of both the harmonic Dirichlet and Neumann 
  fields. 
\end{theorem}
\begin{proof} As the components decouple at $k=0$ it suffices to check
  each separately. Since the Hodge star-operator interchanges the solutions, as
  well as, the 1- and 2-form Dirichlet/Neumann fields, we need only check the
  1-form case.

  First we show that $\bxi_0$ is orthogonal to the harmonic Dirichlet fields.
  These fields are of the form $\bxi_d=du,$ where $u$ is a harmonic function,
  constant on each component of $\Gamma.$ We observe that $u=(|\bx|^{-1})$ and
  $\bxi_0=O(|\bx|^{-2}),$ and this justifies the following integration by
  parts:
  \begin{equation}
    \label{eq:dirperp}
\begin{split}
    \langle\bxi_0,\bxi_d\rangle_{\Omega} &=\langle
    \bxi_0,dr\wedge u\rangle_{\Gamma}+
\langle d^*\bxi_0,u\rangle_{\Omega}\\
&=\langle
    i_{\bn}\bxi_0,u\rangle_{\Gamma}=0
\end{split}
  \end{equation}
The last equation follows as $i_{\bn}\bxi_0$ has mean zero over every component of
$\Gamma,$ and $u$ is constant on each component.

We now turn to the Neumann fields. A Neumann field, $\bxi_n,$ satisfies
$\bxi_n=O(|x|^{-2}).$ In this instance we use the equation $\bxi_0=d\phi$ to
conclude that,
 \begin{equation}
    \label{eq:neuperp}
\begin{split}
    \langle\bxi_0,\bxi_n\rangle_{\Omega} &=\langle
    dr\wedge \phi , \bxi_n\rangle_{\Gamma}+
\langle \phi,d^*\bxi_n\rangle_{\Omega}\\
&=\langle\phi,
    i_{\bn}\bxi_n\rangle_{\Gamma}=0
\end{split}
  \end{equation}
The last equality follows as $i_{\bn}\bxi_n\restrictedto_{\Gamma}\equiv 0,$ by
definition. 

An outgoing harmonic 1-form is determined by it normal components along $\Gamma,$ up to
the addition of an arbitrary Neumann field. As the limit $\bxi_0$ is required
to have mean zero on every component of $\Gamma,$ but is otherwise
unrestricted, it follows that every outgoing harmonic 1-form, $\bxi,$ has a unique
orthogonal decomposition as:
\begin{equation}
  \bxi=\bxi_0+\bxi_d+\bxi_n
\end{equation}
One simply chooses $\bxi_d$ so that $i_{\bn}(\bxi-\bxi_d)$ has mean zero on
every component of $\Gamma.$ This then uniquely determines $\bxi_0.$ Recalling
that the Dirichlet and Neumann harmonic 1-forms are themselves orthogonal, the
Neumann component is then determined by orthogonally projecting
$(\bxi-\bxi_d-\bxi_0)$ onto the Neumann fields. This completes the proof of the
theorem. 
\end{proof}

We now turn to data from $\cH^1(\Gamma).$ The first cohomology group of
$\Gamma$ splits into two disjoint subspaces, one is the image of the
restriction map $H^1_{\dR}(D)\to \cH^1(\Gamma),$ the other the image of
$H^1_{\dR}(\Omega)\to \cH^1(\Gamma).$ By the Mayer-Vietoris sequence, these
restriction maps are injective, and so, by a small abuse of terminology, we may
speak of $H^1_{\dR}(\Omega)$ and $H^1_{\dR}(D)$ as subspaces of
$H^1_{\dR}(\Gamma),$ and write
\begin{equation}
\label{h1splt}
H^1_{\dR}(\Gamma)=H^1_{\dR}(\Omega)\oplus H^1_{\dR}(D),
\end{equation}
see~\cite{Vick}.  With this notation, $H^1_{\dR}(\Omega)$ is dual to the
``A-cycles,'' shown in Figure~\ref{fig1}, while $H^1_{\dR}(D)$ is dual
to the ``B-cycles.''

The solutions to \THME[$0$] are harmonic fields, which satisfy the
\emph{decoupled} equations:
\begin{equation}
  d\bxi_{\pm}=d^*\bxi_{\pm}=0\text{ and }d\bEta_{\pm}=d^*\bEta_{\pm}=0.
\end{equation}
From these equations it is clear that $\bxi_{\pm}$ and $\star_3\bEta_{\pm}$ are
closed 1-forms and therefore define classes in their respective
$H^1_{\dR}$-groups. We see that
\begin{equation}\label{1frmbv}
[\bxi_{+t}]_{\Gamma}, [(\star_3\bEta_{+})_t]_{\Gamma}\in H^1_{\dR}(\Omega),
\text{ and }[\bxi_{-t}]_{\Gamma}, [(\star_3\bEta_{-})_t]_{\Gamma}\in
H^1_{\dR}(D).
\end{equation}
The jump relations show that the solution of the \THME[$0$] defined by the data
$(0,0,\bj_H),$ satisfies
\begin{equation}\label{jmprel}
[\bxi_{+t}-\bxi_{-t}]=-\bj_H\text{ and }
[(\star_3\bEta_{+})_t-(\star_3\bEta_{-})_t]=\star_2\bj_H.
\end{equation}

Choose a basis of harmonic 1-forms, $\{\psi^l:\: l=1,\dots,g\}$ for
$H^1_{\dR}(\Omega)\subset H^1_{\dR}(\Gamma).$ Their Hodge duals
$\{\psi^{g+l}=\star_2\psi^l:\: l=1,\dots,g\}$ are also harmonic and are a basis
for $H^1_{\dR}(D)\subset H^1_{\dR}(\Gamma).$ We say that such a basis is
adapted to the splitting of $H^1 _{\dR}(\Gamma)$ in~\eqref{h1splt}. We let
$\{(\bxi_{\pm}^l,\bEta_{\pm}^l):\: l=1,\dots,2g\},$ denote the solutions to the
\THME[$0$] defined by this data, with $(r,q)=(0,0).$ The relations
in~\eqref{1frmbv} imply that image of each of the restriction maps
$\bxi_+\mapsto [\bxi_{+t}]_{\Gamma},$ $\bxi_-\mapsto [\bxi_{-t}]_{\Gamma},$
spans a subspace of $H^1_{\dR}(\Gamma)$ of dimension at most $g.$ On the other
hand, the jump relations show that the differences,
\begin{equation}
  [\bxi^l_{+t}]_{\Gamma}-[\bxi^l_{-t}]_{\Gamma}=[\psi^l]_{\Gamma},
\end{equation}
span all of $H^1_{\dR}(\Gamma).$ These relations, and analogous ones for the
$\bEta$-components, along with~\eqref{jmprel}, easily imply the following
result:
\begin{proposition}\label{thm111} The solutions $(\bxi^l_{\pm},\bEta^l_{\pm}),$
  satisfy
\begin{enumerate}
\item For $l=1,\dots,g,$ the restrictions $[\bxi^l_{+t}]_{\Gamma}$ span
$H^1_{\dR}(\Omega)\subset H^1_{\dR}(\Gamma),$ while the restrictions
$[(\star_3\bEta_{+}^l)_t]_{\Gamma}=0.$ 
\item For $l=g+1,\dots,2g,$ the restrictions $[(\star_3\bEta^l_{+})_t]_{\Gamma}$ span
$H^1_{\dR}(\Omega)\subset H^1_{\dR}(\Gamma),$ while the restrictions
$[\bxi_{+t}^l]_{\Gamma}=0.$ 
\item For $l=g+1,\dots,2g,$ the restrictions $[\bxi^l_{-t}]_{\Gamma}$ span
$H^1_{\dR}(D)\subset H^1_{\dR}(\Gamma),$ while the restrictions
$[(\star_3\bEta_{-}^l)_t]_{\Gamma}=0.$ 
\item For $l=1,\dots,g,$ the restrictions $[(\star_3\bEta^l_{-})_t]_{\Gamma}$ span
$H^1_{\dR}(D)\subset H^1_{\dR}(\Gamma),$ while the restrictions
$[\bxi_{-t}^l]_{\Gamma}=0.$ 
\end{enumerate}
\end{proposition}

We now recall the basis of $k$-Neumann fields,
$N_k=\{(\bxi_{N+}^l(k),\bEta_{N+}^l(k)):\: l=1,\dots,2g\},$ constructed in the
proof of Theorem~\ref{thm6.1}. This is an analytic family in a neighborhood of
$0,$ as it only requires the solvability of the normal equations. We now assume
that we define these fields, using a basis, $\{\psi^l:\: l=1,\dots,2g\},$ of
$\cH^1(\Gamma),$ which is adapted to the splitting in~\eqref{h1splt}.
Proposition~\ref{thm111} shows that at $k=0$ the fields in $N_k$ continue to
span a $2g$-dimensional vector space of solutions to the \THME[$0$], and that
$\{\bxi_{N+}^l(0):\: l=1,\dots,g\}$ are a basis for the space of outgoing,
harmonic 1-forms, with vanishing normal component along $b\Omega.$ Note that
for $l=1,\dots, g$ we have
\begin{equation}\label{eq6.24.1}
  i_{\bn}\bxi_{N+}^{l+g}(0)=0\text{ and }[\bxi_{N+t}^{l+g}(0)]_{\Gamma}=0.
\end{equation}
These fields are outgoing, harmonic 1-forms, with vanishing normal components
along $b\Omega,$ hence there must be constants $\{a_1,\dots,a_g\}$ so that
\begin{equation}
  \bxi_{N+}^{l+g}(0)=\sum_{m=1}^ga_m\bxi_{N+}^{m}(0).
\end{equation}
The second equation in~\eqref{eq6.24.1} implies that all the coefficients are
zero. This proves that the basis $N_k$ reduces at $k=0,$ to a basis of the
form:
\begin{equation}\label{h10}
  \{(\bxi_{N+}^{l}(0),0), (0,\bEta_{N+}^{l+g}(0)):\: l=1,\dots,g\}.
\end{equation}
We summarize these results in a theorem.
\begin{theorem}\label{thm88} There is an open neighborhood $U$ of $0\in\bbC,$ and $2g$
  analytic families of outgoing solutions
  $\{(\bxi_{N+}^l(k),\bEta_{N+}^l(k)):\: l=1,\dots,2g\}$ to the \THME[$k$],
  which, for each $k\in U$ are a basis for the $k$-Neumann fields
  $\cH_k(\Omega).$ At $k=0$ the $\bxi$- and $\bEta$-components decouple and
  satisfy~\eqref{h10}.
\end{theorem}

Theorems~\ref{thm77} and~\ref{thm88} give a clear picture of the behavior of
the space of solutions to the \THME[$k$] in a neighborhood of zero, defined by
the representation~\eqref{eqn29}. They show that, in a reasonable sense, this
representation does not suffer from low frequency breakdown. Let $U\subset\bbC$
be a neighborhood of zero, and $\{\balpha(k):\: k\in U\}$ be a continuous
family of 1-forms defined on $\Gamma.$ If $\balpha$ is orthogonal to
$\cH^1(\Gamma)$ and $d_{\Gamma}\balpha(k)/k$ has a limit as $k$ tends to zero,
then it is clear that the hybrid system provides a continuous family of
solutions, in a neighborhood of zero, to the \THME[$k$] with
$\bxi_{+t}(k)=\balpha(k).$ In a subsequent publication we will consider
conditions on the projection of $\balpha(k)$ into $\cH^1(\Gamma),$ which are
needed to conclude the existence of such a continuous family of solutions.

\section{The Normal Component Equations on the Unit Sphere}\label{thesphere}
In this and the following section we determine the exact form of the systems of
Fredholm equations derived above for the special case of the unit sphere in
$\bbR^3.$ We make extensive usage of spherical harmonics and ``vector''
spherical harmonics, in the exterior form representation. As this is not
standard, these formul{\ae} are derived in Appendix~\ref{appvsh}. The equations
decouple, and very nicely illustrate the general properties described above. As
the equations for the normal components are a bit simpler, we begin with them.

The integral equations for the normal components of $\bxi$ and $\bEta,$ can be
solved simply and explicitly when $\Gamma$ is 
the unit sphere centered at $0.$ This reveals the close
connection between our equations and the Mie-Debye solution.  
We are representing
$\bxi$ and $\bEta$ in terms of the potentials $\balpha,\balpha_m, \phi,$ and
$\Phi_m,$ with $\bj$ a 1-form on $S^2_1$ and $\bm=\star_2\bj.$ The Debye
sources $r, q$ satisfy:
  \begin{equation}
    \label{eq:3.11.8.1}
    ik rdA= d_{S^2_1}\star_2\bj\text{ and }
ikqdA= d_{S^2_1}\bj.
  \end{equation}
If we assume that
\begin{equation}
  \label{eq:3.11.8.2}
  \begin{split}
    r=\sum_{lm}a_{lm}Y_l^m &\quad  q=\sum_{lm}b_{lm}Y_l^m\\
\bj=\sum_{lm}[ \alpha_{lm}d_{S^2_1}Y_l^m & +\beta_{lm}\star_2 d_{S^2_1}Y_l^m ],
  \end{split}
\end{equation}
then~\eqref{eq:3.11.8.1} implies that
\begin{equation}\label{eq:3.11.8.11}
  -l(l+1)\alpha_{lm}=ik a_{lm}\text{ and
   }-l(l+1)\beta_{lm}=ik b_{lm}.
\end{equation}
Suppose that the normal components of $\bxi$ and $\bEta$ are represented in
terms of spherical harmonics by
\begin{equation}
  i_{\bn}\bxi=\sum_{lm} c_{lm}Y_l^m\text{ and }i_{\bn}\star_3\bEta=
\sum_{lm} d_{lm}Y_l^m.
\end{equation}

Using the results of Propositions~\ref{propB8} and~\ref{propB9}, in the
appendix, we see that the integral equations in~\eqref{eqn74} for the different
spherical harmonic components decouple. The equation for the coefficient of the
$lm$-component of $i_{\bn}\bxi$ becomes:
\begin{multline}
  c_{lm}Y_l^m=\\
i_{\pa_r}\Bigg[
ik \left(\alpha_{lm}G_k\left[(d_{S^2_1}Y_l^m)\cdot
  d\bx\right]+ \beta_{lm}G_k\left[(\star_2d_{S^2_1}Y_l^m)\cdot
  d\bx\right]\right)-
\\d G_k( a_{lm}Y_l^m)\\-\star_3d\left(\alpha_{lm}G_k\left[(\star_2 d_{S^2_1}Y_l^m)\cdot
  d\bx\right]-\beta_{lm}G_k\left[(d_{S^2_1}Y_l^m)\cdot
  d\bx\right]\right)
\Bigg]
\end{multline}
Using the results of these propositions we see that this equation reduces to:
\begin{multline}
  \label{3.11.8.4}
  c_{lm}=ik\alpha_{lm}\left(\frac{ikl(l+1)}{2l+1}\right)\left[j_{l-1}(k)\hone_{l-1}(k)-
j_{l+1}(k)\hone_{l+1}(k)\right]-\\ik^2a_{lm}j_l(k)\pa_k\hone_l(k)+
  ikl(l+1)\alpha_{lm}j_l(k)\hone_l(k).
\end{multline}
Standard recurrence relations for the spherical Bessel functions imply that
\begin{multline}\label{3.11.8.5}
  \frac{k\left(j_{l-1}(k)\hone_{l-1}(k)-j_{l+1}(k)\hone_{l+1}(k)\right)}{2l+1}=\\
\frac{j_l(k)\hone_l(k)}{k}+j_l(k)\pa_k\hone_l(k)+\pa_kj_l(k)\hone_l(k),
\end{multline}
see~\cite{jackson}.  Using this identity and the relations
in~\eqref{eq:3.11.8.11}, we obtain
\begin{equation}
  c_{lm}=a_{lm}k\hone_l(k)\left(ij_l(k)+ikj_l'(k)+
kj_l(k)\right).
\end{equation}
We define the function
\begin{equation}
\label{eq:3.19.4}
  m_n(k,l)=k\hone_l(k)\left((i+k)j_l(k)+ikj_l'(k)\right).
\end{equation}
An essentially identical sequence of steps leads to the relations:
\begin{equation}\label{nrmlHeqn}
  d_{lm}=-m_n(k,l)b_{lm}.
\end{equation}
The diagonal entries of the block diagonal matrix we need to invert to solve
the normal component problem for the unit sphere, at frequency $k,$ are simply
$\{m_n(k,l):\: l=1,2,\dots\}.$

\begin{remark}
The connection to classical Debye theory is now easy to establish.
If we expand the Debye potentials $u,v$ in (\ref{debyerep}) as 

\begin{eqnarray*}
v(r,\theta,\phi) &=& \sum_{l,m} a_{lm} \hone_l(kr) Y_l^m  \\
u(r,\theta,\phi) &=& \sum_{l,m} b_{lm} \hone_l(kr) Y_l^m \, , 
\end{eqnarray*}
then a straightforward calculation \cite{Papas} 
shows that (in terms of the normal components)
\begin{eqnarray*}
a_{lm} &=& \frac{1}{l(l+1)} c_{lm} \\
b_{lm} &=& -\frac{1}{l(l+1)} d_{lm} \, .
\end{eqnarray*}
Thus, our generalized Debye sources, defined only on the surface,
are analogous (but not equivalent) to the
restrictions of the Debye potentials to the sphere. 
The classical Debye approach
requires that the potentials themselves (defined in $\bbR^3\setminus\Gamma$) 
be expanded in surface harmonics, preventing
the approach from being extensible to arbitrary geometry. 
\end{remark}

We apply the Wronskian identity,
\begin{equation}
  j_l(k)\pa\hone_l(k)-\pa_k j_l(k)\hone_l(k)=\frac{i}{k^2},
\end{equation}
to find that:
\begin{equation}
  m_n(k,l)=\left(1+ikj_l(k)\hone_l(k)+ik^2j_l(k)\pa_k
  \hone_l(k)+ k^2j_l(k)\hone_l(k)\right).
\end{equation}
Applying standard asymptotic formul{\ae} for $j_l$ and $\hone_l$ to this
representation shows that, for a fixed $k$ with non-negative real part, we
have:
\begin{equation}
  m_n(k,l)\sim \frac{1}{2}-\frac{i}{2l+1}+O(l^{-2}).
\end{equation}
This agrees with the fact that the integral equation for $i_{\bn}\bxi$ is of the
form $\frac 12 +K(k),$ where $K$ is compact. It is also the case that
\begin{equation}
  m_n(0,l)= \frac{1}{2}-\frac{i}{2l+1},
\end{equation}
which shows that these equations do not exhibit low frequency breakdown.
 For integral $l,$ and $k$ along the real axis we have
\begin{equation}
  m_n(k,l)\sim 1+O(k^{-1}),
\end{equation}
as $\Im k$ tends to infinity we have:
\begin{equation}
  m_n(k,l)\sim 1+O(k^{-1}).
\end{equation}
Figure~\ref{fig111}(a) shows plots of $\{|m_n(k,l)|:\: k=1,10,100\}.$ The
condition number increases with the frequency, $k.$

\begin{figure}[H]
  \centering
  {\epsfig{file=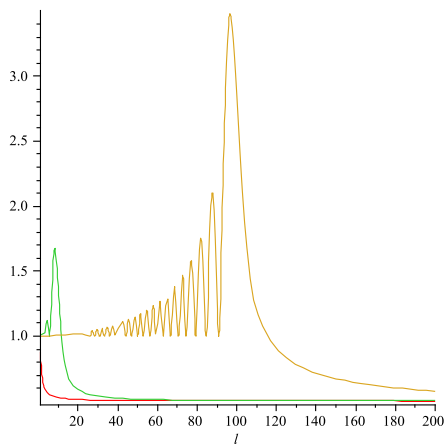, width=6cm}}
 \caption{Plots of $|m_n(1,l)|,$ $|m_n(10,l)|,$ $|m_n(100,l)|.$}
\label{fig111}
\end{figure}

For fixed $l$  the solutions of $m_n(k,l)=0$ have $\Im k<0.$ If $l$ is
fixed  then the imaginary parts of the roots of $m_n(k,l)=0,$
decrease in proportion to minus the log of the real part,
\begin{equation}
  \Im k\propto -\frac 12\log\Re k.
\end{equation}
The first 50 zeros of $m_n(k,1),$ $m_n(k,5),$ and $m_n(k,7)$ are shown in
Figure~\ref{fig112}.

\begin{figure}[H]
  \centering 
\subfigure[Zeros of $m_n(k,1).$]
  {\epsfig{file=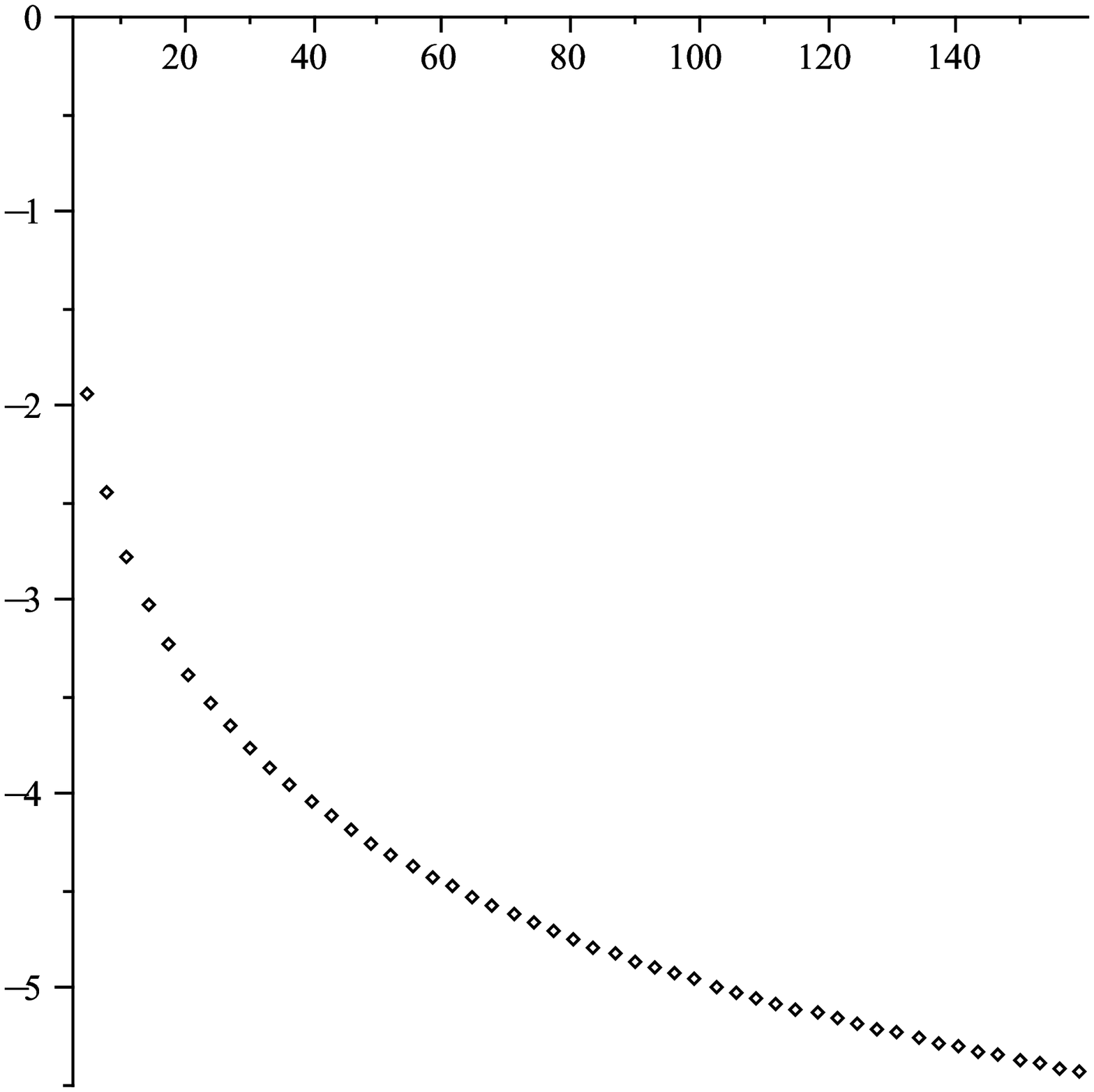, width=3.75cm}}
\subfigure[Zeros of $m_n(k,5).$]
  {\epsfig{file=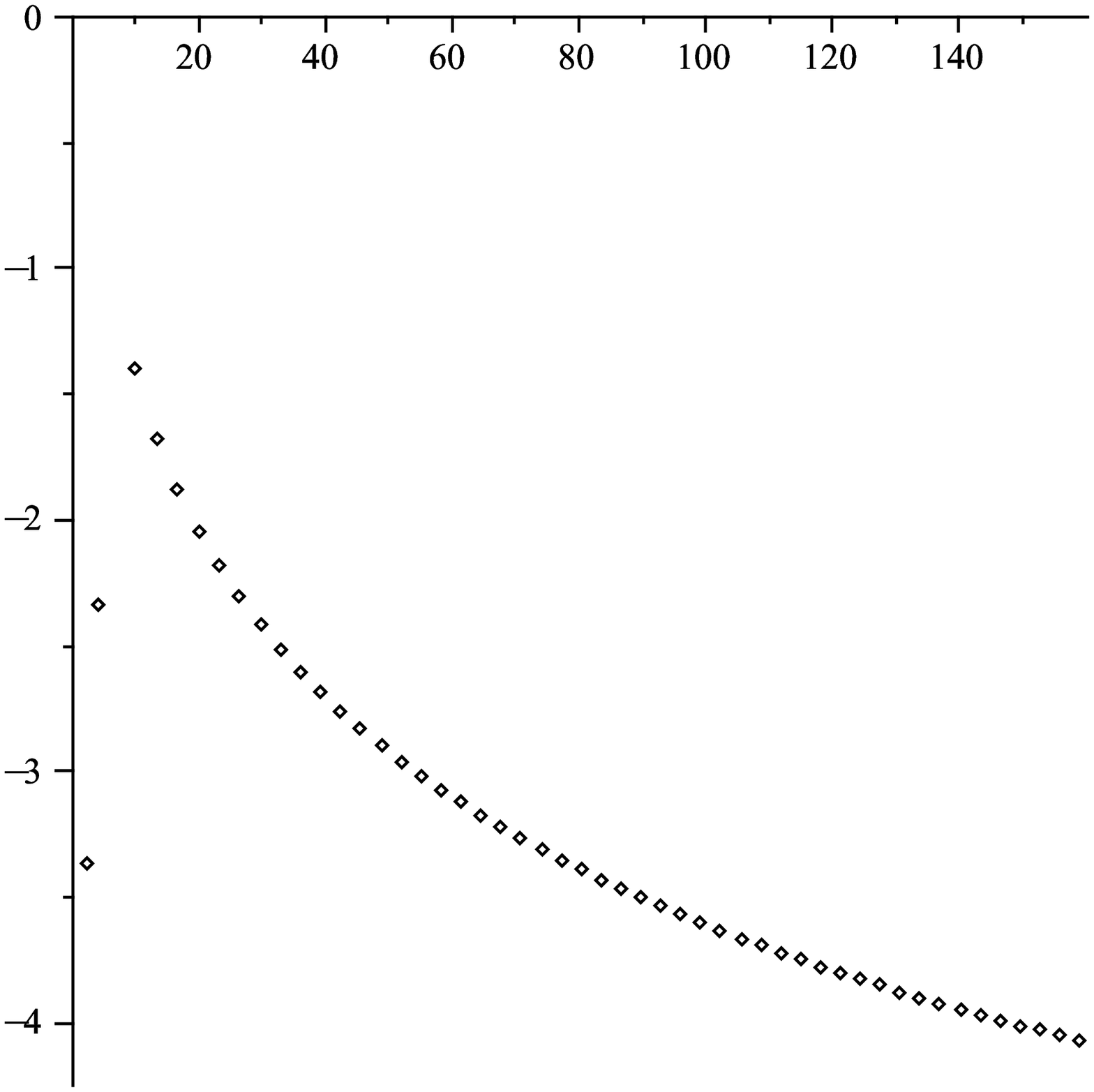, width=3.75cm}}
\subfigure[Zeros of $m_n(k,7).$]
  {\epsfig{file=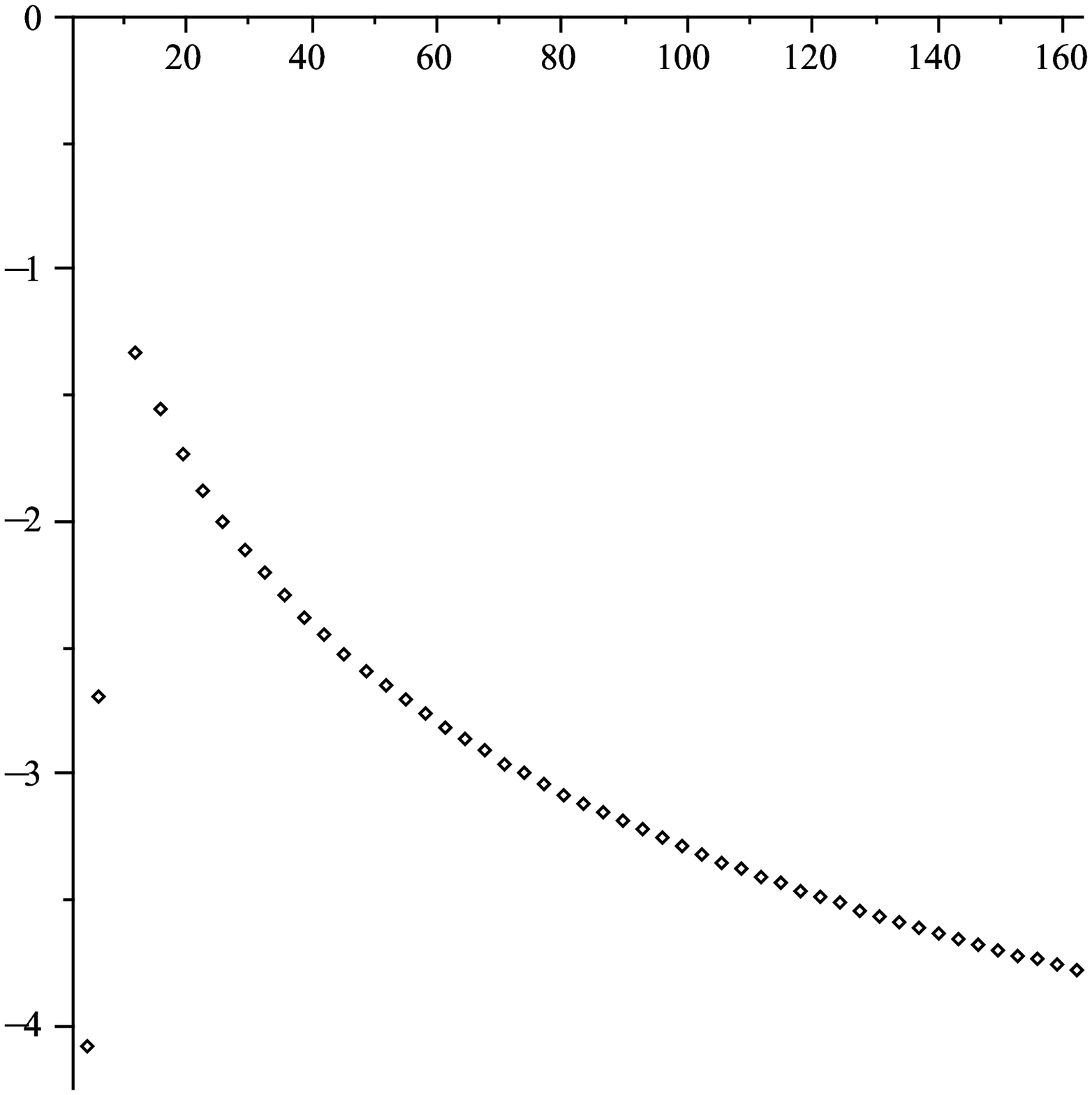, width=3.75cm}}
\caption{Graphs of the first 50 zeros of the multipliers $m_n(k,l)$ for
  $l=1,5,7.$ }
\label{fig112}
\end{figure}

Using the three term recurrence relations for spherical Bessel functions,
$\{z_l\},$ we easily obtain the functional equation:
\begin{equation}
  z_l(-\bar{k})=(-1)^l\overline{z_l(k)};
\end{equation}
using this identity and~\eqref{eq:3.19.4} it is not difficult to show that
\begin{equation}
  \label{eq:3.19.6}
  m_n(-\bar{k},l)=\overline{m_n(k,l)}.
\end{equation}
The function $\hone_l$ can be factored as
\begin{equation}
  \label{eq:3.19.1}
  \hone_l(k)=p_l(k)\frac{e^{ik}}{k^{l+1}},
\end{equation}
where $p_l$ is a polynomial of degree $l.$ Thus the multiplier takes the form
\begin{equation}
  \label{eq:3.19.2}
\begin{split}
  m_n(k,l) &=p_l(k)e^{ik}\left(\frac{(i+k)j_l(k)+ikj_l'(k)}{k^l}\right)\\
&=p_l(k)e^{ik}\left(\frac{(k-il)j_l(k)+ikj_{l-1}(k)}{k^l}\right)
\end{split}
\end{equation}
It is easy to see that the numerator has a zero of order $l$ at $k=0,$ and
therefore the quotient is regular and non-vanishing there.  The polynomial
contributes $l$ roots; the symmetry,~\eqref{eq:3.19.6}, shows that, when $l$ is
odd, one root lies on the negative imaginary axis.   A plot showing the
roots with smallest imaginary part, and positive real part is shown in
Figure~\ref{fig11a}.
\begin{figure}[H]
  \centering 
  {\epsfig{file=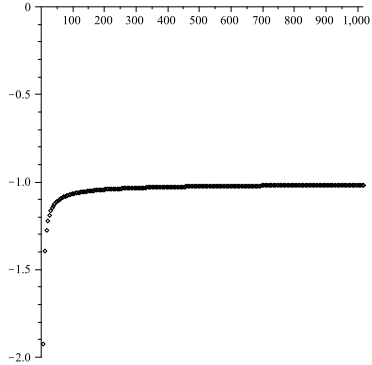, width=5cm}}
  \caption{The solutions of $m_n(k,l)=0$ with smallest modulus, and positive
    real part, for $l=1$ to $1000.$ The real part increases monotonely with
    $l.$}
\label{fig11a}
\end{figure}

\begin{remark}\label{bad_poles}
  Theorem~\ref{thm66} shows that the frequencies $k,$ with $\Im k\geq 0$, for
  which equation~\eqref{eqn74} has a non-trivial null-space coincides with the
  eigenvalues of the \emph{interior} boundary value problem for Maxwell's
  equations defined by
\begin{equation}
  \label{eq:tngbdr10}
  \bxi_-\restrictedto_{\Gamma}=
i_{\bn}\bEta_-\restrictedto_{\Gamma}.
\end{equation}
Of course there are no eigenvalues, or resonances with $\Im k\geq 0.$
Calculations like those above, though simpler, show that the boundary condition
holds for the vector spherical harmonics of order $l$ provided:
\begin{equation}
  \label{eq:bcj}
  (i+k)j_l(k)+ikj_l'(k)=0\text{ with }k\neq 0.
\end{equation}
From~\eqref{eq:3.19.2}, we see that the left hand side of~\eqref{eq:bcj} is a
factor of $m_n(k,l).$ Thus the eigenvalues of the interior problem are a subset
of the resonances of the exterior problem. These eigenvalues are the
``non-physical'' interior resonances connected with our
representation~\eqref{eqn29} of $\bxi$ and $\bEta$ in terms of potentials. They
are familiar from the EFIE and MFIE representation, but shifted to the lower
half plane, where they do no serious harm. The roots of $\hone_l(k)$ are known
to be related to scattering resonances for scattering off of a conducting sphere.
\end{remark}

\section{The Hybrid Equations on the Unit Sphere}
To find the precise form of the hybrid operator, $\cQ^{+}(k),$ on the unit
sphere, we only need to work out $-G_0d^*_{S^2_1}\star_2\cT^{+}_{\bxi}(k).$ The normal
equation is given by~\eqref{nrmlHeqn}. 
As in the previous section, we represent $\bxi$ and $\bEta$ in terms of the
potentials $\balpha,\balpha_m, \phi,$ and $\Phi_m,$ with $\bj=\bj_R(r,q,k)$ a
1-form on $S^2_1$ and $\bm=\star_2\bj.$ The generalized Debye sources $r, q$
satisfy~\eqref{eq:3.11.8.2} and~\eqref{eq:3.11.8.11}. The $\bxi$-field is
given, in terms of the potentials by
\begin{equation}
  \bxi=[ikG_k\bj\cdot d\bx-dG_kr-\star_3dG_k\star_2\bj\cdot d\bx].
\end{equation}
Using the expressions for $r$ and $\bj$ in terms of spherical, resp. vector
spherical harmonics, we see that the tangential components of $\bxi_+$ modulo
$\ker d_{S^2_1}^*,$ are given by
\begin{equation}
  \begin{split}
    &\bxi_{+t}\text{mod}\ker d_{S^2_1}^* =\\
    &\sum_{l,m}d_{S^2_1}Y_l^m\Bigg[\left(\frac{-k^2\alpha_{lm}}{2l+1}\right)
[(l+1)j_{l-1}(k)\hone_{l-1}(k)+lj_{l+1}\hone_{l+1}(k)]\\
-&ika_{lm}j_l(k)\hone_l(k)+
ik\alpha_{lm}j_l(k)[\hone_l(k)+k\pa_k\hone_l(k)]\Bigg]
  \end{split}
\end{equation}

Using~\eqref{eq:3.11.8.11} and the identity, 
\begin{equation}
G_0d_{S^2_1}^*[d_{S^2_1}Y_l^m]=\frac{l(l+1)}{2l+1}Y^l_m,
\end{equation}
we easily obtain that
\begin{equation}\label{eqn6.29.1}
\begin{split}
  G_0d^*_{S^2_1}&\cT^+_{\bxi}(k)\left(\begin{matrix} r\\ q\end{matrix}\right) =
\sum_{l,m}\left(\frac{a_{lm}Y_l^m}{2l+1}\right)\times\\
&\Bigg[-ikl(l+1)j_l(k)\hone_l(k)+
k^2j_l(k)[\hone_l(k)+k\pa_k\hone_l(k)]\\
&+\frac{ik^3[(l+1)j_{l-1}(k)\hone_{l-1}(k)+lj_{l+1}\hone_{l+1}(k)]}{2l+1}\Bigg].
\end{split}
\end{equation}
 As with the system of normal equations, the hybrid equations are decoupled,
 providing one equation for the coefficients of $r$ and one for the
 coefficients of $q.$ The only term on the right hand side of~\eqref{eqn6.29.1}
 that is not $O(l^{-1})$ is
 \begin{equation}
   \frac{-ikl(l+1)j_l(k)\hone_l(k)}{2l+1}=\frac{-1}{4}+O(l^{-1}),
 \end{equation}
in agreement with~\eqref{hybreqn1}. We use the identity
\begin{equation}
  \pa_{k}(k\hone_l(k))=k\hone_{l-1}(k)-l\hone_l(k)
\end{equation}
to remove the derivative from~\eqref{eqn6.29.1}, and define
the multiplier for the tangential equation:
\begin{equation}
\begin{split}
  m_t(k,l)=&\left(\frac{-k}{2l+1}\right)\Bigg[il(l+1)j_l(k)\hone_l(k)-
kj_l(k)[k\hone_{l-1}(k)-l\hone_l(k)]\\
&-\frac{ik^2[(l+1)j_{l-1}(k)\hone_{l-1}(k)+lj_{l+1}(k)\hone_{l+1}(k)]}{2l+1}\Bigg].
\end{split}
\end{equation}

This multiplier behaves much like the multiplier $m_n(k,l)$ found for the
normal equations. For fixed $l$ its roots, as a function of $k,$ lie in the
lower half plane. The plots in Figure~\ref{fig333} show contours of
$\log|m_t(k,l)|$ for $l=1,10,20.$ The $x$-axis is shown as a black horizontal
line. They clearly show that the zeros lie in the lower half plane, and
indicate the moderate behavior of the multiplier in the upper half plane. The
plots in Figure~\ref{fig334} show the $|m_t(k,l)|$ for $l$ between $1$ and
$20,$ and $20$ and $200,$ respectively, for and $k=1, 10$ and $100.$ It should
be recalled that there is a certain amount of arbitrariness in the definition
of this multiplier, resulting from the arbitrariness in the choice of
preconditioner, $G_0d_{S^2_1}^*$ in the present instance.

If the incoming tangential data are given by
\begin{equation}
  \bxi_t^{\In}=\sum_{l,m}\left[ p_{lm}\frac{d_{S^2_1}Y_l^m}{\sqrt{l(l+1)}}+
 q_{lm}\frac{\star_2d_{S^2_1}Y_l^m}{\sqrt{l(l+1)}}\right],
\end{equation}
(here we used the normalized basis elements) then
\begin{equation}
\begin{split}
  i_{\bn}\star_3\bEta^{\In}=\frac{\star_2d_{S^2_1}\bxi_t^{\In}}{ik}=\sum_{l,m} 
 q_{lm}\frac{\star_2d_{S^2_1}\star_2d_{S^2_1}Y_l^m}{ik\sqrt{l(l+1)}}\\
=-\frac{1}{ik}\sum_{l,m} \sqrt{l(l+1)}q_{lm}Y_l^m.
\end{split}
\end{equation}

To find the tangential data for the hybrid system we apply $G_0d_{S^2_1}^*$ to
$\bxi_t^{\In},$ obtaining:
\begin{equation}
  G_0d_{S^2_1}^*\bxi_t^{\In}=\sum_{l,m} \frac{\sqrt{l(l+1)}}{2l+1}p_{lm}Y_l^m.
\end{equation}
For the unit sphere, the hybrid equations are therefore:
\begin{equation}
  \begin{split} 
m_n(k,l)b_{lm}&=\frac{\sqrt{l(l+1)}q_{lm}}{ik}\\
m_t(k,l)a_{lm}&= \frac{\sqrt{l(l+1)}}{2l+1}p_{lm}.
\end{split}
\end{equation}
These equations display a mild sort of low frequency breakdown, in that the
coefficients of the normal data, $\{\sqrt{l(l+1)}q_{lm}\},$ must be uniformly
$O(\omega)$ in order for this system of equations to be stable. Of course the
incoming data $(\bxi_+^{\In},\bEta_+^{\In})$ is assumed to be a solution of the
\THME[$k$], so these estimates should automatically hold. Indeed if
$\bEta_+^{\In}$ is given, then there is no need to differentiate
$\bxi_+^{\In},$ and divide by $k$ to find the data for the normal equation.

\begin{figure}[H]
  \centering 
\subfigure[Contour plot of of $\log|m_t(k,1)|.$]
{\epsfig{file=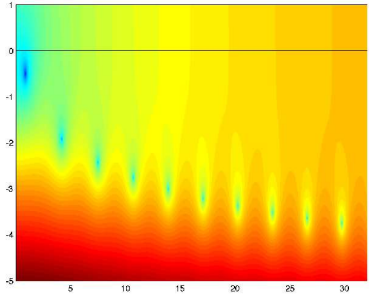, width=6cm}} 
\subfigure[Contour plot of of $\log|m_t(k,10)|.$]
  {\epsfig{file=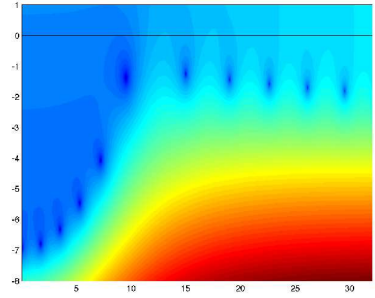, width=6cm}}
\subfigure[Contour plot of of $\log|m_t(k,25)|.$]
  {\epsfig{file=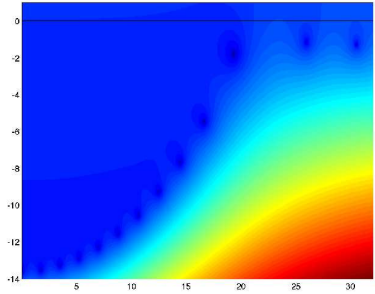, width=6cm}}
\caption{Plots of $\log|m_t(k,1)|,$ $\log|m_t(k,10)|,$ $\log|m_t(k,25)|.$ The
  black horizontal line indicates the $x$-axis. The zeros are located near the
  deep blue dots.}
\label{fig333}
\end{figure}

The use of $G_0$ in the preconditioner also leads to growth in the multiplier
$m_t(k,l)$ are $k$ increases for fixed $l$. Figure~\ref{fig334}(c) show
$|m_t(k,1)|, |m_t(k,10)|,$ and $|m_t(k,20)|,$ for real $k\in [0,100].$ In the
interval $0<k<l$ these functions oscillate around a small non-zero value. When
$k$ exceeds $l$ these functions show linear growth. Replacing $G_0,$ with
something like $G_{i|k|}$ should fix this problem.

\begin{figure}[H]
  \centering 
\subfigure[Small $l$]
{\epsfig{file=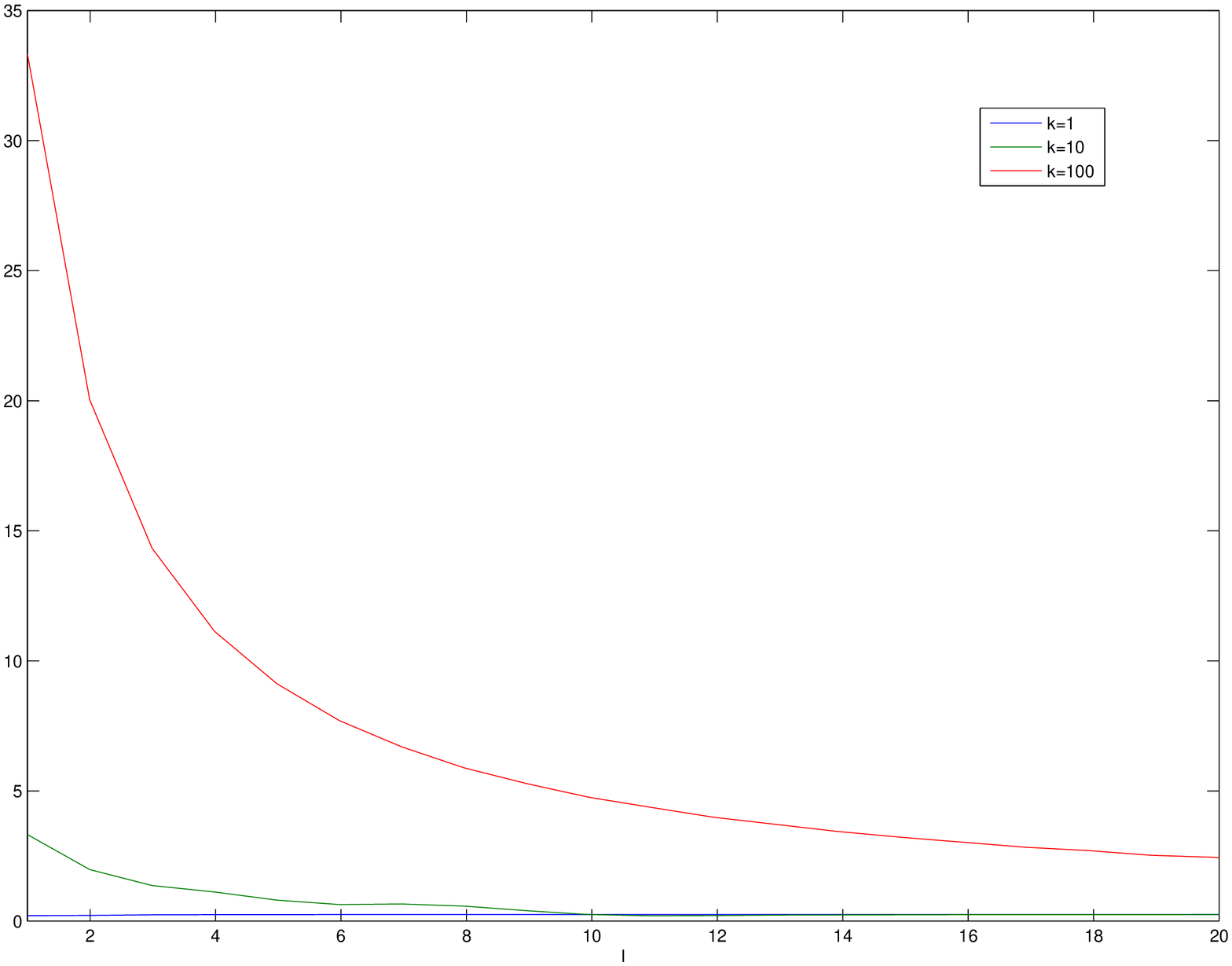, width=4.5cm}} 
\subfigure[Large $l$]
  {\epsfig{file=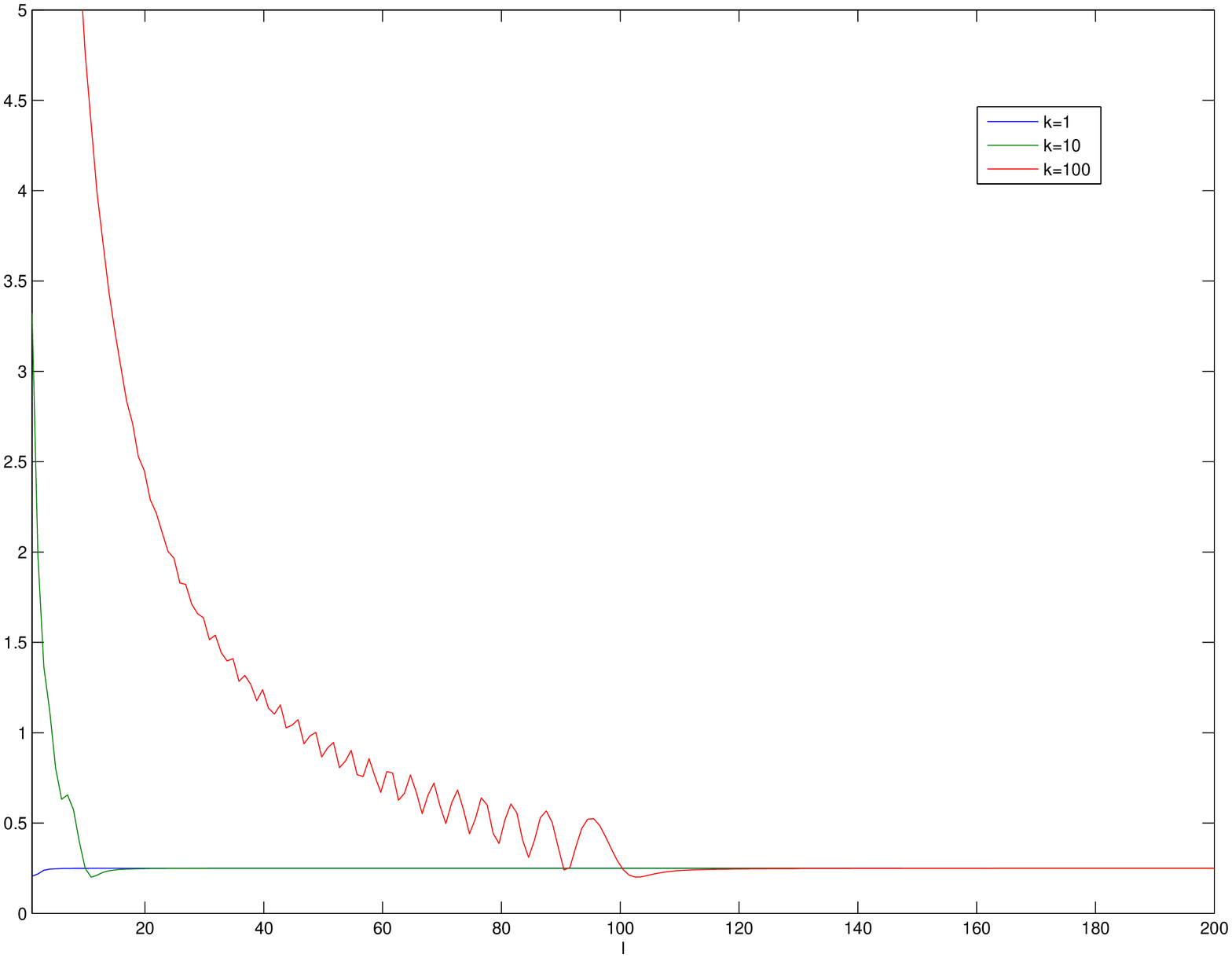, width=4.5cm}}
\subfigure[Fixed $l$]
  {\epsfig{file=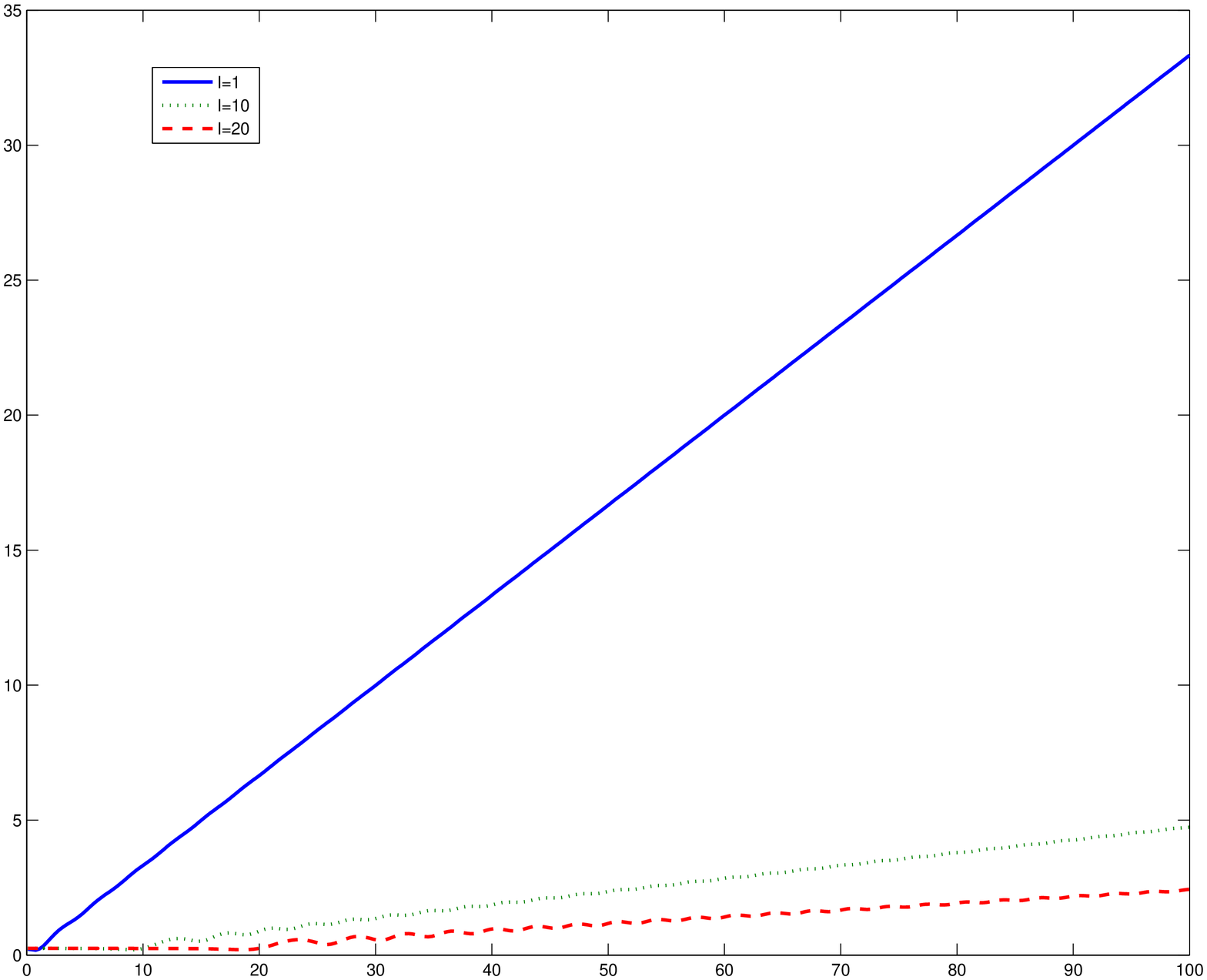, width=4.5cm}}
\caption{Plots of $|m_t(k,1)|,$ for a large range of $l$ and $k=1,10,100,$ and
  fixed $l=1,10,20$ with $k\in [0,100].$}
\label{fig334}
\end{figure}

\section{Conclusions}\label{conclusions}


In this paper, we have developed a new representation for solutions of the time
harmonic Maxwell equations, exterior to closed surfaces, based on two scalar
densities. In the zero frequency limit, these densities are uncoupled and
correspond to electric and magnetic charge.  At non-zero frequency, however,
they do not correspond directly to physical variables.  They are simply used to
construct electric and magnetic currents, after which the classical scalar and
vector potentials and anti-potentials are employed (in the usual Lorenz gauge).
Because of the close connection to the Lorenz-Debye-Mie formalism when the
analysis is restricted to the unit sphere, we refer to our unknowns as {\em
  generalized Debye sources}.  The natural boundary data for our unknowns are
the normal components of the electric and magnetic field and we have provided a
detailed uniqueness theory for this boundary value problem for boundary
surfaces of arbitrary genus 
(Theorems \ref{thm1}, \ref{mcnormalthm}, \ref{thm4.1}). In the course
of this analysis, we have given a new proof of the existence (in the non-simply
connected case) of families of nontrivial solutions with zero boundary data,
which we refer to as $k$-Neumann fields. They generalize, to non-zero wave
numbers, the classical harmonic Neumann fields (Theorem \ref{thm6.1}).

We have also introduced a new Fredholm integral equation of the second
kind for scattering from a perfect electrical conductor and have shown that it is 
invertible (in the simply connected case) for all wave numbers $k$ 
in the closed upper half plane.
There is a natural extension of the approach to the case of
a dielectric interface, which will be reported at a later date.

The work begun here gives rise to a new set of analytic and computational
issues.  In order to use the Debye sources as unknowns, one needs an efficient
and accurate method for inverting the surface Laplacian.  For surfaces $\Gamma$
of genus $g>0$, we also need to be able to efficiently construct a basis for
the harmonic forms $\cH^1(\Gamma).$ 
Finally, additional work is required to extend our
approach to open surfaces (see, for example, \cite{JiangRokhlin}).  
These arise as common and important idealizations
in the analysis of thin plates, cylindrical conductors and metallized surface
patches in radar, medical imaging, chip design and remote sensing applications.

\appendix
\vskip .25in
\noindent
{\LARGE{\bf Appendix}}

\section{Exterior Forms, Maxwell Equations and Vector
Spherical Harmonics}\label{app1}
In the traditional approach to electricity and magnetism Maxwell's equations are
expressed in terms of relationships between four vector fields $\bE, \bD$ and
$\bB, \bH$
defined on $\bbR^3\times\bbR:$
\begin{equation}
\begin{split}
\frac{\pa\bD}{\pa t}=c\nabla\times\bH-4\pi\bJ &\quad \frac{\pa\bB}{\pa
  t}=-c\nabla\times\bE\\
\nabla\cdot\bD=4\pi\rho &\quad \nabla\cdot\bB=0;
\end{split}
\label{vectmaxeq}
\end{equation}
$c$ is the speed of light.  Here $\bJ$ is the current density and $\rho$ is the
charge density, they satisfy the conservation of charge:
\begin{equation}
\frac{\pa\rho}{\pa t}=-\nabla\cdot\bJ.
\end{equation} 
The differential symmetries of this system of equations are rooted in the
exactness of the sequence:
\begin{equation}\label{clexctseq}
\CI(U)\overset{\nabla}{\longrightarrow}\CI(U;T\bbR^3)
\overset{\nabla\times}{\longrightarrow}\CI(U;T\bbR^3)
\overset{\nabla\cdot}{\longrightarrow}\CI(U),
\end{equation}
here $U\subset\bbR^3$ is an open set, and $\CI(U;T\bbR^3)$ are the smooth
vector fields defined in $U:$
\begin{equation}
\CI(U;T\bbR^3)=\{a_1(x)\pa_{x_1}+a_2(x)\pa_{x_2}+a_3(x)\pa_{x_3}:\: a_1, a_2,
 a_3\in\CI(U)\}.
\end{equation}
The exactness of the sequence is equivalent to the classical
identities $\nabla\times\nabla=0,$ and $\nabla\cdot\nabla\times=0.$

While this representation is traditional, physically and geometrically it makes
more sense to regard Maxwell's equations as a relationship amongst
differential, or exterior forms. In the second part of this paper we usually
work with the fields $\bE$ and $\bH.$ It turns out to be convenient to use a
1-form to represent $\bE$ and a 2-form to represent $\bH.$ We use the
correspondences
\begin{equation}
\begin{split}
\bH=h_1\pa_{x_1}+h_2\pa_{x_2}+h_3\pa_{x_3}&\leftrightarrow h_1dx_2\wedge
dx_3+h_2dx_3\wedge dx_1+h_3dx_1\wedge dx_2=\bEta\\
\bE=e_1\pa_{x_1}+e_2\pa_{x_2}+e_3\pa_{x_3}&\leftrightarrow e_1dx_1+e_2dx_2+e_3dx_3=\bxi
\end{split}
\label{2fv}
\end{equation}
Under this correspondence $\bxi$ is the metric dual of $\bE$ and $\star\bEta$
($\star$ is the Hodge star-operator defined by the metric on $\bbR^3$) is the
metric dual of $\bH.$ 

It is natural to think of the electric field as a 1-form, for the electric
potential difference is then obtained by integrating this 1-form:
\begin{equation}
\phi_P-\phi_Q=\int\limits_{\gamma}\bxi,
\end{equation}
with $\gamma$ a path from $P$ to $Q.$ Similarly, it is reasonable to think
of the magnetic field as a 2-form, for the flux of $\bH$ through a surface
$\Sigma$ is then obtained by integrating:
\begin{equation}
\text{Flux of $\bH$ through }\Sigma=\int\limits_{\Sigma}\bEta.
\end{equation}
While these are the most basic measurements associated to electric and magnetic
fields, there are times when it is natural to integrate the $\bE$-field over a
surface, or the $\bH$-field over a curve. This is done, in the form language,
by using the Hodge-star and interior product operations, $\star, i_{\bv},$
introduced below. A detailed exposition of this approach to Maxwell's equations
can be found in~\cite{Axelsson}.

For the sake of completeness we recall the definition of $d$ on forms defined
on $\bbR^3:$
\begin{equation}
\begin{split}
&0\text{-forms }\Lambda^0\bbR^3:
\quad da=\pa_{x_1}adx_1+\pa_{x_2}adx_2+\pa_{x_3}adx_3\\
&1\text{-forms }\Lambda^1\bbR^3:\quad d(a_1dx_1+a_2dx_2+a_3dx_3)=
(\pa_{x_1}a_2-\pa_{x_1}a_1)dx_1\wedge
dx_2\\&\phantom{ll}+
(\pa_{x_3}a_1-\pa_{x_1}a_3)dx_3\wedge dx_1+ (\pa_{x_2}a_3-\pa_{x_3}a_2)dx_2\wedge dx_3\\
&2\text{-forms }\Lambda^2\bbR^3:\quad d(a_1dx_1\wedge dx_2+a_2dx_3\wedge dx_1+
a_3dx_2\wedge dx_3)=\\
&\phantom{\text{ll}}(\pa_{x_1}a_1+\pa_{x_2}a_2+\pa_{x_3}a_3)dx_1\wedge dx_2\w dx_3\\
&3\text{-forms }\Lambda^3\bbR^3:\quad d(adx_1\w dx_2\w dx_3)=0.
\end{split}
\end{equation}

The sequence~\eqref{clexctseq} becomes:
\begin{equation}
\CI(U)\overset{d}{\longrightarrow}\CI(U;\Lambda^1\bbR^3)
\overset{d}{\longrightarrow}\CI(U;\Lambda^2\bbR^3)
\overset{d}{\longrightarrow}\CI(U;\Lambda^3\bbR^3).
\end{equation}
\emph{All} of the classical differential relations are simply $d^2=0.$

\subsection{Exterior forms on a manifold}
Generally we can define the smooth exterior $p$-forms on an $n$-dimensional
manifold $M,$ $\CI(M;\Lambda^pT^*M),$ as sections of the vector bundle
$\Lambda^pT^*M.$ The exterior derivative is a canonical map
\begin{equation}
  d:\CI(M;\Lambda^pT^*M)\to \CI(M;\Lambda^{p+1}T^*M).
\end{equation}
If $(x_1,\dots,x_n)$ are local coordinates, then a $p$-form can be expressed as
\begin{equation}
  \balpha=\sum_{I\in\cI_p}a_I(x)dx_{i_1}\w\cdots\w dx_{i_p},\quad a_I\in\CI.
\end{equation}
Here $\cI_p$ is the set of increasing $p$-multi-indices $1\leq
i_1<\cdots<i_p\leq n.$ In these local coordinates, the exterior derivative of a
function $f(x)$ is defined to be
\begin{equation}
  df=\sum_{j=1}^n\pa_{x_j}f(x)dx_j,
\end{equation}
and of a $p$-form
\begin{equation}
  d\balpha=\sum_{I\in\cL_p}da_I(x)\w dx_{i_1}\w\cdots\w dx_{i_p}.
\end{equation}
The remarkable fact is that this is invariantly defined; though it is really nothing
more than the chain rule. The fact that mixed partial derivatives commute
easily applies to show that $d^2=0.$
If $\balpha$ and $\bBeta$ are forms, then we have the Leibniz Formula:
\begin{equation}
d(\balpha\w\bBeta)=(d\balpha)\w\bBeta+(-1)^{\deg\balpha}\balpha\w d\bBeta.
\end{equation}

If $\balpha$ is a $k$-form defined in an open subset, $U$ of $M$ then we can
integrate it over any smooth, oriented compact submanifold $\Sigma^k\subset\!\subset U,$ of
dimension $k,$  with or without boundary. We denote this pairing by
\begin{equation}
\left<\balpha, [\Sigma]\right>=\int\limits_{\Sigma}\balpha.
\end{equation}
If $\balpha$ is an exact form, that is, $\balpha=d\bBeta,$ and $\Sigma$ is a
smooth, oriented submanifold with or without boundary, then Stokes' theorem states that
\begin{equation}
\int\limits_{\Sigma}d\bBeta=\int\limits_{\pa\Sigma}\bBeta.
\end{equation}
The boundary must be given the induced orientation.  Note, in particular, that
if $\pa\Sigma=\emptyset,$ then the integral of $d\bBeta$ over $\Sigma$ vanishes.  

There is a second natural operation on exterior forms that satisfies a Leibniz
formula.  If $\bv$ is a vector field, then the \emph{interior product} of $\bv$
with a $k$-form, $\omega,$ is a $(k-1)$-form, $i_{\bv}\omega,$ defined by:
\begin{equation}
i_{\bv}\omega(\bv_2,\dots,\bv_{k})\overset{d}{=}\omega(\bv,\bv_2,\dots,\bv_k).
\end{equation}
If $\omega$ and $\eta$ are exterior forms, then
\begin{equation}
i_{\bv}[\omega\wedge\eta]=[i_{\bv}\omega]\wedge\eta+(-1)^{\deg\omega}\omega\wedge[i_{\bv}\eta].
\end{equation}

Using forms simplifies calculations considerably because forms can be
automatically integrated over submanifolds of the ``correct'' dimension, keep
track of orientation, and all the differential relationships follow from the
fact that $d^2=0.$ Moreover, Stokes' theorem subsumes all the classical
integration by parts formul{\ae} is one simple package.

\subsection{Hodge star-operator}\label{A2}
To write the Maxwell equations we need one further operation, called
the \emph{Hodge star-operator}. This operation can be defined on an
oriented Riemannian manifold. Suppose that $\omega_1,\dots,\omega_n$ is an
local orthonormal basis of one forms, and
\begin{equation}
dV=\omega_1\wedge\cdots\wedge\omega_n,
\end{equation}
defines the orientation. If $1\leq i_1<\cdots<i_p\leq n,$ and
$j_1<\cdots<j_{n-p}$ are  complementary indices, then we define
\begin{equation}
\star[\omega_{i_1}\wedge\cdots\wedge\omega_{i_p}]=
(-1)^{\epsilon}\omega_{j_1}\wedge\cdots\wedge\omega_{j_{n-p}},
\end{equation}
with $\epsilon=0$ or $1,$ chosen so that
\begin{equation}
\omega_{i_1}\wedge\cdots\wedge\omega_{i_p}\wedge
\star[\omega_{i_1}\wedge\cdots\wedge\omega_{i_p}]=dV.
\end{equation}

For example, the Hodge operator is defined on the standard orthonormal basis of
exterior forms for $\bbR^3$ by setting:
\begin{equation}
\begin{split}
\star 1&=dx_1\wedge dx_2\wedge dx_3\\
\star dx_1=dx_2\wedge dx_3\quad \star dx_2&=dx_3\wedge dx_1\quad \star
dx_3=dx_1\wedge dx_2\\
\star dx_1\wedge dx_2 =dx_3\quad \star dx_3\wedge dx_1&= dx_2\quad \star
dx_2\wedge dx_3=dx_1\\
\star dx_1\wedge dx_2\wedge dx_3 &= 1
\end{split}
\label{star}
\end{equation}
 From these formul{\ae} it is clear that, in 3-dimensions,
$\star^2=I.$ Notice that applying the $\star$-operator exchanges the two
correspondences between vector fields and forms in~\eqref{2fv}.

The star-operator is simply related to the metric: If $\balpha,\bBeta$ are real
$p$-forms, then
\begin{equation}
  [\balpha\wedge\star\bBeta]_x=(\balpha,\bBeta)_xdV_x,
\label{starmetric}
\end{equation}
where $(\cdot,\cdot)$ is the (real) inner product defined by the metric on
$p$-forms. Generally, 
\begin{equation}
  \label{eq:str2}
  \star^2=(-1)^{p(n-p)},
\end{equation}
on $p$-forms defined on an $n$-dimensional manifold. Using this observation
and~\eqref{starmetric} we easily show that $\star$ is a pointwise isometry:
\begin{equation}
  \label{eq:striso}
  (\balpha,\bBeta)_x=(\star\balpha,\star\bBeta)_x
\end{equation}

A fundamental role of the Hodge $\star$-operator is to define a Hilbert space
inner product on forms.  If $\balpha$ and $\bBeta$ are (possibly complex) forms
of the same degree defined in $U,$ then $\balpha\w\star\overline{\bBeta}$ is a
$n$-form, which can therefore be integrated:
\begin{equation}
\begin{split}
\langle\balpha,\bBeta\rangle&=\int\limits_{U}\balpha\w\star\overline{\bBeta}.\\
&=\int\limits_{U}(\balpha,\bBeta)_xdV(x).
\end{split}
\label{inprd}
\end{equation}
We assume that $(\cdot,\cdot)$ is extended to define an Hermitian inner product
on complex valued forms. The extended metric continues to satisfy~\eqref{eq:striso}.

\subsection{Adjoints, Integration-by-parts and the Hodge Theorem}\label{adjintprt}
On an $n$-dimensional manifold the expression for the formal adjoint, with
respect to the pairing in~\eqref{inprd}, of the $d$-operator, acting  on a
$p$-form $\bBeta,$ is:
\begin{equation}
d^*\bBeta=\begin{cases}-\star d\star\bBeta&\text{ if }n\text{ is even}\\
(-1)^p\star d\star\bBeta&\text{ if }n\text{ is odd.}\end{cases}
\label{dadj}
\end{equation}
Let $G$ be a bounded domain with a smooth boundary. Let $r$ be a function that
is negative in $G$ and vanishes on $bG.$ Suppose moreover that $(dr,dr)_x\equiv
1$ for $x\in bG,$ and  $\bn$ is the outward pointing unit
normal along $bG.$ The basic integration by parts formul{\ae} for $d$ and $d^*$
can be expressed in terms of this inner product by
\begin{equation}
  \label{eq:intprts1}
  \begin{split}
\langle d\balpha,\bBeta\rangle_{G}&=\int\limits_{bG}(dr\wedge\balpha,\bBeta)_xdS(x)+
\langle \balpha,d^*\bBeta\rangle_{G}\\
\langle d^*\balpha,\bBeta\rangle_{G}&=-\int\limits_{bG} (i_{\bn}\balpha,\bBeta)_xdS(x)+
\langle \balpha,d\bBeta\rangle_{G}.
\end{split}
\end{equation}
Here we use $dS$ to denote surface measure on $bG.$
It is important to recall that, with respect to the pointwise inner product,
\begin{equation}
  (dr\wedge \balpha,\bBeta)_x=
(\balpha,i_{\bn}\bBeta)_x.
\end{equation}

The (positive) Laplace operator, acting  on any form degree is given by  formula
\begin{equation}
d d^*+d^* d.
\end{equation}
In $\bR^3$ this would give $-(\pa^2_{x_1}+\pa^2_{x_2}+\pa^2_{x_3}).$
To avoid confusion with standard usage in E\&M, we use $\Delta$ to denote the
negative operator $-(d d^*+d^* d).$ If $M$ is a compact manifold without
boundary, then the de Rham cohomology groups are defined, for $0\leq k\leq\dim M,$ as
\begin{multline}
  H_{\dR}^k(M)\\=\Ker\{d:\CI(M;\Lambda^kT^*M)\to\CI(M;\Lambda^{k+1}T^*M)\}/d\CI(M;\Lambda^{k-1}T^*M).
\end{multline}
It is a classical theorem that these abelian group are topological invariants,
see~\cite{Vick}.  We let $\cH^k(M)$ denote the nullspace of the Laplacian
acting on $\CI(M;\Lambda^kT^*M).$ Stokes' theorem shows that
\begin{equation}
  \langle\Delta\omega,\omega\rangle=\langle d\omega,d\omega\rangle+
\langle d^*\omega,d^*\omega\rangle.
\end{equation}
Thus $\cH^k(M)\simeq \Ker d\cap\Ker d^*.$ The Hodge theorem
states that
\begin{theorem}[Hodge] If $(M,g)$ is a compact Riemannian manifold, without
  boundary, then, for each $0\leq k\leq\dim M,$
  \begin{equation}
    \cH^k(M)\simeq H^k_{\dR}(M).
  \end{equation}
and, as an $L^2$-orthogonal direct sum, we have:
\begin{equation}
  \CI(M;\Lambda^kT^*M)=d \CI(M;\Lambda^{k-1}T^*M)\oplus d^*\CI(M;\Lambda^{k+1}T^*M)\oplus\cH^k(M).
\end{equation}
  \end{theorem}
As $\Delta$ is an elliptic operator, the Hodge theorem has the following very
useful corollary:
\
\begin{corollary}
  If $(M,g)$ is a compact Riemannian manifold, without
  boundary, then for each $0\leq k\leq\dim M,$ the $\dim H^k_{\dR}(M)$ is finite.
\end{corollary}

\subsection{Maxwell's Equations in terms of exterior forms}\label{A4}
With these preliminaries we can state the correspondences between the
differential operators, $\nabla, \nabla\times,$ and $\nabla\cdot$ and the
corresponding objects acting on forms. For a scalar function $\phi,$
$\nabla\phi$ corresponds to $d\phi.$ An elementary calculation shows that if
$\bE\leftrightarrow\bxi,$ a 1-form, then $\nabla\times\bE\leftrightarrow
d\bxi,$ and $\nabla\cdot\bE\leftrightarrow d^*\bxi.$ Moreover with
$\bH\leftrightarrow\bEta,$ a 2-form, we have $\nabla\times\bH\leftrightarrow
d^*\bEta,$ and $\nabla\cdot\bH\leftrightarrow d\bEta.$ The operator $d^*$ also
acts on $3$-forms.

If we let $\bE\leftrightarrow \bxi,$
$\bH\leftrightarrow\bEta,$ and $\bJ\leftrightarrow\bj$ (a 1-form)
as in~\eqref{2fv}, then Maxwell's equations in a vacuum become:
\begin{equation}
\begin{split}
\frac{\pa\bxi}{\pa t}=cd^*\bEta-4\pi\bj &\quad \frac{\pa\bEta}{\pa
  t}=-cd\bxi\\
d^*\bxi=4\pi\rho \quad d\bEta&=0\quad \frac{\pa\rho}{\pa t}=-d^*\bj.
\end{split}
\label{formmaxeq}
\end{equation}
If $\bxi$ and $\bEta$ are time harmonic with time dependence  $e^{-it\omega},$ then in the
absence of sources, we easily derive the Helmholtz equations:
\begin{equation}
c^2\Delta\bxi+\omega^2\bxi=0\quad c^2\Delta\bEta+\omega^2\bEta=0.
\label{helm}
\end{equation}

We let $D\subset\bbR^3$ denote a bounded set with smooth boundary and let
\begin{equation}
\Omega =\bbR^3\setminus \overline{D}\quad \Gamma=bD.
\end{equation}
In this paper $D$ is usually taken to be a perfect conductor, lying in a
bounded domain with smooth boundary, and $\Omega$ a dielectric.  We assume that
$\epsilon$ is the electrical permittivity, $\mu$ is the magnetic permeabilty
and $\sigma$ the electrical conductivity of $\Omega.$ As above, we identify the
$\bE$-field with a 1-form, $\bXi:$
\begin{equation}
e_1\pa_{x_1}+e_2\pa_{x_2}+e_3\pa_{x_3}=\bE\leftrightarrow\bXi=
e_1dx_1+e_2dx_2+e_3dx_3,
\end{equation}
and $\bH$ with a 2-form, $\bN:$ 
\begin{equation}h_1\pa_{x_1}+h_2\pa_{x_2}+h_3\pa{x_3}=\bH\leftrightarrow \bN=
h_1dx_2\wedge dx_3+h_2dx_3\wedge dx_1+h_3dx_1\wedge dx_2.
\end{equation}
In terms of exterior forms, we set
\begin{equation}
\bXi(x,t)=\left[\frac{\omega}{\omega\epsilon+i\sigma}\right]^{\frac
  12}\bxi(x)e^{-i\omega t}\quad
\bN(x,t)=\mu^{-\frac 12}\bEta(x)e^{-i\omega t};
\end{equation}
the time harmonic Maxwell equations become:
\begin{equation}
\begin{split}
&d\bxi=ik\bEta\quad d^*\bxi=0\\
&d^*\bEta=-ik\bxi\quad d\bEta=0.
\end{split}
\label{MxEqnk}
\end{equation}
Here $k$ is the square root of $\mu(\epsilon\omega^2+i\sigma\omega),$ with
non-negative imaginary part. 

With this choice of correspondence between the vector and form representations,
we can write the Maxwell equations in a very succinct and symmetric form:
\begin{equation}
(d+d^*)(\bxi+\bEta)=ik\Lambda(\bxi+\bEta);
\end{equation}
here $\Lambda$ is the operation defined on forms by
\begin{equation}
  \Lambda(\balpha)=(-1)^{\deg\balpha}\balpha.
\end{equation}
Simple calculations shows that
\begin{equation}
  (d+d^*)\Lambda=-\Lambda(d+d^*)\text{ and }\Lambda^2=I,
\end{equation}
implying that
\begin{equation}
  (d+d^*-ik\Lambda)^2=-(\Delta+k^2).
\end{equation}
Thus, acting on forms, the operator $\Delta+k^2$ has a local square root. Or,
put differently, $d+d^*-ik\Lambda$ is an operator of Dirac-type,
see~\cite{BBW}. Indeed, we could write the vacuum Maxwell equations in the form
\begin{equation}
  [c(d+d^*)+\Lambda\pa_t](\bxi+\bEta)=4\pi c\rho,
\end{equation}
noting that
\begin{equation}
  [c(d+d^*)+\Lambda\pa_t]^2=\pa_t^2-c^2\Delta.
\end{equation}

\subsection{Vector Spherical Harmonics}\label{appvsh}
The domains defined in $\bbR^3$ as complements of a round sphere are very
important in applications. They also provide a context where the integral
equations defined in the earlier sections can be diagonalized and solved
explicitly in terms of classical special functions. In this appendix we give a
brief treatment of the theory of vector spherical harmonics in the exterior
form representation. A classical treatment is given in sections 9.6-7
of~\cite{jackson}. But for the classical theory of (scalar) spherical harmonics
(which can also be found in Jackson), our discussion is essentially self contained.

We begin with the relationship between the (negative) Laplace operators in $\bbR^3$ and on
the unit sphere $S^2_1\subset\bbR^3.$ Recall that on any Riemannian manifold,
$(X,g),$ the Laplace operator on $k$-forms is given by
$\Delta^X_k=-(d^*d+dd^*),$ where $*$ is defined by $g.$ The following is simply
the usual change of variables formula for spherical polar coordinates.
\begin{proposition}
 Let $r^2=x_1^2+x_2^2+x_3^2,$ the scalar Laplace operator on $\bbR^3$ can be expressed as
  \begin{equation}
    \label{eq:B1}
    \Delta^{\bbR^3}_0=\frac{1}{r^2}\pa_r
    r^2\pa_r+\frac{1}{r^2}\Delta^{S^2_1}_0,
\text{ where } \pa_r=\frac{x_1\pa_{x_1}+x_2\pa_{x_2}+x_3\pa_{x_3}}{r}.
  \end{equation}
\end{proposition}
We also need to relate the Laplacians on 1-forms. 
\begin{proposition} Let $\balpha$ be a 1-form on $\bbR^3$ such that
  $i_{\pa_r}\balpha=0,$ then 
\begin{equation}
  \label{eq:B2}
  \Delta_1^{\bbR^3}\balpha=
  \frac{1}{r^2}\Delta_1^{S^2_1}\balpha+L_r\balpha+
\frac{1}{r^2}\left(d_{S^2_1}^*\balpha\right) dr,
\end{equation}
where
\begin{equation}
  L_r\balpha=i_{r^2\pa_r}d(i_{r^{-2}\pa_r}d\balpha)+
\frac{2}{r^2}\balpha. 
\end{equation}
\end{proposition}
This formula follows by a calculation using a local co-frame field. If we
express $\balpha=\sum \alpha_j dx_j,$ then
\begin{equation}
  L_r\balpha=\sum_{j=1}^3(\pa_r^2\alpha_j) dx_j.
\end{equation}
While~\eqref{eq:B2} is a good deal more complicated than~\eqref{eq:B1}, it
allows for a careful analysis of the eigenforms of $\Delta_1^{S^2_1}.$

We begin with the standard description of the eigenspaces of
$\Delta_0^{S^2_1}.$ Let $\cE^0_l$ denote the
linear space of scalar eigenfunctions on $S^2_1$ satisfying:
\begin{equation}
  \label{eq:B3}
  \Delta_0^{S^2_1}f=-l(l+1)f.
\end{equation}
These spaces are represented in terms of classical spherical harmonics by
\begin{equation}
  \cE^0_l=\Span\{Y_{l}^m:\: m=-l,\dots, l\}
\end{equation}
A basis of eigenforms of $\Delta_2^{S^2_1},$ with eigenvalue $l(l+1),$ is given by
$$\{\star_2Y_{l}^m=Y_{l}^mdA:\: m=-l\dots l\}.$$ 
The fact that $H^1(S^2;\bbR)=0$ and the Hodge theorem imply that the 1-forms on
$S^2_1$ are the $L^2$-orthogonal direct sum
  \begin{equation}
\CI(S^2_1;\Lambda^1)=d_{S^2_1}\CI(S^2_1;\Lambda^0)\oplus 
d_{S^2_1}^*\CI(S^2_1;\Lambda^2).
\end{equation}
This observation coupled with the fact that $d$ and $d^*$ commute with $\Delta$
imply that the eigenspaces of $\Delta_1^{S^2_1}$ are given, for $l\in\bbN,$ by
\begin{equation}
    \label{eq:B5}
    \cE^1_l=\Span\big[\{d_{S^2_1}Y_l^m:\: m=-l,\dots,l\}\oplus 
\{\star_2d_{S^2_1}Y_l^m:\: m=-l,\dots,l\}\big],
  \end{equation}
  with eigenvalue $-l(l+1).$ From this representation, the classical
  orthogonality relations are quite easy:
\begin{equation}
\begin{split}
  &\langle d_{S^2_1}Y_l^m,d_{S^2_1}Y_{l'}^{m'}\rangle=\langle
  d_{S^2_1}^*d_{S^2_1}Y_l^m,Y_{l'}^{m'}\rangle=\delta_{ll'}\delta^{mm'}l(l+1)\\
 &\langle \star_2d_{S^2_1}Y_l^m,\star_2d_{S^2_1}Y_{l'}^{m'}\rangle=\langle
  d_{S^2_1}Y_l^m,d_{S^2_1}Y_{l'}^{m'}\rangle=\delta_{ll'}\delta^{mm'}l(l+1)\\
 &\langle d_{S^2_1}Y_l^m,\star_2d_{S^2_1}Y_{l'}^{m'}\rangle=\langle 
Y_l^m,d_{S^2_1}^2Y_{l'}^{m'}\rangle=0.
\end{split}
\end{equation}
The second line from the fact that $\star_2$ is an orthogonal transformation,
and the last relation follows from Stokes theorem.

The eigenforms $\star_2d_{S^2_1}Y_l^m$ are divergence free. If we extend them
to $\bbR^3$ so they annihilate $\pa_r,$ and express them in the form
$\star_2dY_l^m=a_1dx_1+a_2dx_2+a_3dx_3,$ where the coefficients are extended to
be homogeneous of degree $0,$ then it follows from~\eqref{eq:B2} and the
equation
\begin{equation}
  \Delta_1^{\bbR^3}(a_1dx_1+a_2dx_2+a_3dx_3)= (\Delta_0^{\bbR^3}a_1)dx_1+
(\Delta_0^{\bbR^3}a_2)dx_2+(\Delta_0^{\bbR^3}a_3)dx_3 
\end{equation}
that is, where $r=1,$ we have
\begin{equation}
  \label{eq:B6}
  \Delta_0^{S^2_1} a_j=-l(l+1) a_j.
\end{equation}
In other words the coefficients of $\star_2d_{S^2_1}Y_l^m$ lie in $\cE^0_l.$
These eigenforms correspond to the classical eigenfields of the form
$\{\br\times\nabla Y_l^m\}.$  

The members of the other family, $\{d_{S^2_1} Y_l^m\},$ which corresponds to
$\{\br\times(\br\times \nabla Y_l^m)\},$ are not divergence free and their
coefficients with respect to $dx_j$ lie in $\cE^0_{l-1}\oplus\cE^0_{l+1}.$
These coefficients are easily found; if we think of $Y_l^m$ as a homogeneous
function of degree $0$ on $\bbR^3,$ then $i_{\pa_r}d_{\bbR^3}Y_l^m=0,$ 
\begin{equation}
  d_{S^2_1}Y_l^m=\sum_{j=1}^3\frac{\pa Y_l^m}{\pa x_j}dx_j\restrictedto_{S^2_1}.
\end{equation}
In order to determine the action of the Green's function on the coefficients of
these forms, we need to represent them in terms of spherical harmonics.
Let $U_{l}^m=r^lY_l^m.$ This is a homogeneous harmonic polynomial of degree
$l.$ We see that
\begin{equation}
  \label{eq:B7}
  \frac{\pa U_l^m}{\pa x_j}=r^l\frac{\pa Y_l^m}{\pa x_j}+l\frac{x_j}{r^2}U_l^m.
\end{equation}
We apply the Laplace operator to $x_jU_l^m$ to obtain
\begin{equation}
  \label{eq:B8}
  \Delta^{\bbR^3}_0(x_jU_l^m)=2\frac{\pa U_l^m}{\pa x_j},
\end{equation}
and therefore 
\begin{equation}
  x_jU_l^m=u_{l+}^{mj}+r^2u_{l-}^{mj}.
\end{equation}
Here $u_{l+}^{mj}, u_{l-}^{mj}$ are homogeneous harmonic polynomials of
degrees $l+1$ and $l-1,$ respectively. Once again applying the Laplace operator
to this relation, we see that
\begin{equation}
  u_{l-}^{mj}=\frac{1}{2l+1}\frac{\pa U_l^m}{\pa x_j},
\end{equation} and therefore
\begin{equation}
  \label{eq:6.27.1}
  \sum_{j=1}^3u_{l+}^{mj}dx_j=U_l^mrdr-\frac{r^2}{2l+1}dU_l^m.
\end{equation}
Using the homogeneity we also see that
\begin{equation}
  \label{eq:B9}
  \sum_{j=1}^3x_ju_{l+}^{mj}=\frac{l+1}{2l+1}r^2U_l^m.
\end{equation}
Restricting to $r=1,$ gives
\begin{equation}
  \label{eq:B10}
  \frac{\pa Y_l^m}{\pa x_j}\restrictedto_{r=1}=\frac{l+1}{2l+1}\frac{\pa U_l^m}{\pa
  x_j}\restrictedto_{r=1}- lu_{l+}^{mj}\restrictedto_{r=1}.
\end{equation}
The functions on the right hand side belong $\cE^0_{l-1}$ and $\cE^0_{l+1}$
respectively. 
Employing these relations, we can work out the action of outgoing Green's function on
$\cE^1_l.$

The outgoing Green's function for frequency $k,$ with $\Im k\geq 0,$ is given
by
\begin{equation}
  \label{eq:B11}
  g_k(\bx,\by)=\frac{e^{ik|\bx-\by|}}{4\pi|\bx-\by|}.
\end{equation}
If $r=|\bx|>1$ and $|\by|=1,$ then we can expand $g$ as
\begin{equation}
  \label{eq:B12}
  g_k(\bx,\by)=ik\sum_{l=0}^{\infty}j_l(k)\hone_l(kr)\bP_l, 
\end{equation}
here $\bP_l$ is the orthogonal projection onto $\cE^0_l,$ and
\begin{equation}
  j_l(z)=\sqrt{\frac{\pi}{2z}}J_{l+\frac 12}(z)\text{ and }
\hone_l(z)=\sqrt{\frac{\pi}{2z}}H^{(1)}_{l+\frac 12}(z),
\end{equation}
see~\cite{jackson}. When $k=0$ formula~\eqref{eq:B12} reduces to
\begin{equation}
  \label{eq:B122}
  g_0(\bx,\by)=\sum_{l=0}^{\infty}\frac{\bP_l}{(2l+1)r^{l+1}}.
\end{equation}
 We let
\begin{equation}
  G_kf(\bx)=\int\limits_{S^2_1}g_k(\bx,\by)f(\by)dA(\by).
\end{equation}

If $\balpha$ is a 1-form on $S^2_1,$ then it has a unique extension to
$T\bbR^3\restrictedto_{S^2_1}$ that annihilates $\pa_r,$ which we denote by
$\balpha\cdot d\bx.$ The extended form has a well defined representation along
$S^2_1$ as
\begin{equation}
  \balpha\cdot d\bx=\sum_{j=1}^3\alpha_jdx_j.
\end{equation}
If we extend the coefficients to be homogeneous functions of degree zero, then
$i_{\pa_r}\balpha\cdot d\bx=0$ implies that
\begin{equation}\label{eq6.28.1}
  \sum_{j=1}^3x_jd\alpha_j=-\balpha\cdot d\bx,
\end{equation}
which will prove useful below.

\begin{proposition}\label{propB8} If $|\by|=1$ and $r=|\bx|>1,$ then, applied component-wise,
  the action of $G_k$ is given by:
  \begin{equation}
    \label{eq:B13}
    \begin{split}
&G_kY_l^m=ikj_l(k)\hone_l(kr)Y_l^m\\
      &G_k\left[(\star_2d_{S^2_1}Y_l^m)\cdot d\bx\right]=ikj_l(k)\hone_l(kr)
(\star_2d_{S^2_1}Y_l^m)\cdot d\bx\\
      &G_k\left[(d_{S^2_1}Y_l^m)\cdot
        d\bx\right]=ik\Bigg[\left(\frac{dU_l^m}{2l+1}\right)
\frac{\big[(l+1)j_{l-1}(k)\hone_{l-1}(kr)+lj_{l+1}(k)\hone_{l+1}(kr)\big]}{r^{l-1}}
\\
&\phantom{mmmmmmmmm}-U_l^mdr\frac{lj_{l+1}(k)\hone_{l+1}(kr)}{r^l}\Bigg].
    \end{split}
      \end{equation}
      On the right hand sides of~\eqref{eq:B13}, $Y_l^m$ is homogeneous of
      degree zero, as are the coefficients of $d_{S^2_1}Y_l^m\cdot d\bx,$ and
      $\star_2d_{S^2_1}Y_l^m\cdot d\bx.$ As above, $U_l^m$ 
      is the homogeneous harmonic polynomial of degree
      $l,$ defined by $Y_l^m.$

Along the unit sphere the normal components are given by
 \begin{equation}
    \label{eq:B14}
    \begin{split}
&i_{\pa_r} dG_k Y_{l}^m=ik^2j_l(k)\pa_k\hone_l(k)Y_{l}^m\\
      &i_{\pa_r}G_k\left[(\star_2d_{S^2_1}Y_l^m)\cdot d\bx\right]=0\\
      &i_{\pa_r}G_k\left[(d_{S^2_1}Y_l^m)\cdot
        d\bx\right]=ik\left(\frac{l(l+1)}{2l+1}\right)
\left[j_{l-1}(k)\hone_{l-1}(k)-j_{l+1}(k)\hone_{l+1}(k)\right]Y_l^m.
\end{split}
      \end{equation}
Along the unit sphere the tangential components are given by
 \begin{equation}
    \label{eq:B144}
    \begin{split}
&[dG_k Y_{l}^m]\restrictedto_{TS^2_1}=ikj_l(k)\hone_l(k)d_{S^2_1}Y_{l}^m\\
     &\left[G_k[(\star_2d_{S^2_1}Y_l^m)\cdot d\bx]\right]\restrictedto_{TS^2_1}=
ikj_l(k)\hone_l(k)\star_2d_{S^2_1}Y_l^m\\
      &\left[G_k[(d_{S^2_1}Y_l^m)\cdot
        d\bx]\right]\restrictedto_{TS^2_1}=\\
&\phantom{mmm}\left(\frac{ik}{2l+1}\right)
\bigg[(l+1)j_{l-1}(k)\hone_{l-1}(k)+lj_{l+1}(k)\hone_{l+1}(k)\bigg]d_{S^2_1}Y_l^m.
\end{split}
      \end{equation}
\end{proposition}
\noindent
We also need to compute the effect of $\star_3d$ on
these eigenforms.
\begin{proposition}\label{propB9}
  Along the unit sphere we have:
 \begin{equation}
    \label{eq:B15}
    \begin{split}
&i_{\pa_r}\star_3dG_k\left[(\star_2d_{S^2_1}Y_l^m)\cdot d\bx\right]
=-ikl(l+1)j_l(k)\hone_l(k)Y_l^m\\
 &i_{\pa_r}\star_3dG_k\left[(d_{S^2_1}Y_l^m)\cdot
 d\bx\right]=0.
\end{split}
      \end{equation}
\end{proposition}

\noindent
Finally we need to calculate the tangential components of these forms; the
identity satisfied by a 2-form
\begin{equation}\label{eqn212}
  \star_3\balpha\restrictedto_{TS^2_1}=\star_2[i_{\pa_r}\balpha\restrictedto_{TS^2_1}],
\end{equation}
along with~\eqref{eq6.28.1} facilitate these computations.
\begin{proposition}
  Along the unit sphere we have:
 \begin{equation}
    \label{eq:B155}
    \begin{split}
&\left[\star_3dG_k[(\star_2d_{S^2_1}Y_l^m)\cdot d\bx]\right]\restrictedto_{TS^2_1}
=-ik d_{s^2_1}Y_l^mj_l(k)[\hone_l(k)+k\pa_k\hone_l(k)]\\
 &\left[\star_3dG_k[(d_{S^2_1}Y_l^m)\cdot
 d\bx]\right]\restrictedto_{TS^2_1}=
\left(\frac{ik\star_2d_{s^2_1}Y_l^m}{2l+1}
\right)\Bigg[l(l+2)j_{l+1}(k)\hone_{l+1}(k)\\
&\phantom{mmmm}-(l-1)(l+1)j_{l-1}(k)\hone_{l-1}(k)+k[(l+1)j_{l-1}(k)\pa_k\hone_{l-1}(k)\\
&\phantom{mmmm}+lj_{l+1}(k)\pa_k\hone_{l+1}(k)]\Bigg].
\end{split}
\end{equation}
\end{proposition}

Given~\eqref{eq:B6},~\eqref{eq:6.27.1},~\eqref{eq:B10},~\eqref{eq:B9},~\eqref{eq:B12},
and~\eqref{eqn212} the formul{\ae} in these propositions are elementary
calculations, which follow from the fact that
\begin{equation}
  \bP_l\restrictedto_{\cE_l^0}=I_{\cE_l^0}\text{ and }
\bP_l\restrictedto_{\cE_k^0}=0\text{ if }k\neq l.
\end{equation}

{\small \bibliographystyle{siam}
{\bibliography{debye6d3}}}

\end{document}